\documentclass{amsart}
\usepackage[a4paper,left=20mm,top=5mm,right=20mm,bottom=5mm,includeheadfoot,headheight=20mm, headsep=5mm, footskip=15mm]{geometry}

\usepackage[utf8]{inputenc}
\usepackage[all]{xy}
\usepackage{amsthm,amsmath,amssymb,amsfonts,amscd}
\usepackage[backend=bibtex]{biblatex}
\usepackage{bbm}
\usepackage{booktabs}
\usepackage{csquotes}
\usepackage[english]{babel}
\usepackage{enumerate}
\usepackage{fancyhdr}
\usepackage[final]{pdfpages}
\usepackage{graphicx}
\usepackage{imakeidx}
\usepackage{inconsolata}
\usepackage{latexsym,keyval,ifthen}
\usepackage[makeroom]{cancel}
\usepackage{mathalfa}
\usepackage{mathdots}
\usepackage{multicol}
\usepackage{multicol}
\usepackage{nameref}
\usepackage{newpxmath}
\usepackage{prerex}
\usepackage{pst-node}
\usepackage{randomwalk}
\usepackage{soul}
\usepackage{tgpagella}
\usepackage{theoremref}
\usepackage{tikz}
\usepackage{tikz-cd}
\usetikzlibrary{arrows,arrows.meta}
\usepackage{wrapfig}
\usepackage{xgalley}
\usepackage{yfonts}
\usepackage{calc}

\newtheorem{defn0}{Definition}[section]
\newtheorem{prop0}[defn0]{Proposition}
\newtheorem{thm0}[defn0]{Theorem}
\newtheorem{lemma0}[defn0]{Lemma}
\newtheorem{claim0}[defn0]{Claim}
\newtheorem{corollary0}[defn0]{Corollary}
\newtheorem{example0}[defn0]{Example}
\newtheorem{remark0}[defn0]{Remark}
\newtheorem{assumption0}[defn0]{Assumption}
\newtheorem{conjecture0}[defn0]{Conjecture}
\newtheorem{notation0}[defn0]{Notation}
\newtheorem{question0}[defn0]{Question}

\newenvironment{definition}{\begin{defn0}\rm}{\end{defn0}}
\newenvironment{proposition}{\begin{prop0}}{\end{prop0}}
\newenvironment{theorem}{\begin{thm0}}{\end{thm0}}
\newenvironment{lemma}{\begin{lemma0}}{\end{lemma0}}

\newenvironment{remark}{\begin{remark0}\rm}{\end{remark0}}
\newenvironment{assumption}{\begin{assumption0}\rm}{\end{assumption0}}

\newcommand{\ipa}[1]{\left(#1\right)}
\newcommand{\Gal}{{\mathrm {Gal}}}

\newcommand{\ord}{\mathrm{ord}}

\newcommand{\M}{\mathrm{M}}

\newcommand{\Ind}{{\mathrm{Ind}}}
\newcommand{\PGL}{{\mathrm{PGL}}}

\newcommand{\GL}{{\mathrm{GL}}}
\newcommand{\SU}{{\mathrm{SU}}}

\newcommand{\SO}{{\mathrm {SO}}}

\newcommand{\SL}{{\mathrm {SL}}}

\newcommand{\Z}{{\mathbb Z}}
\newcommand{\A}{{\mathbb A}}
\newcommand{\Q}{{\mathbb Q}}
\newcommand{\C}{{\mathbb C}}
\newcommand{\R}{{\mathbb R}}

\newcommand{\N}{{\mathbb N}}

\newcommand{\cA}{{\mathcal A}}

\newcommand{\cC}{{\mathcal C}}
\newcommand{\cF}{{\mathcal F}}
\newcommand{\cB}{{\mathcal B}}

\newcommand{\cG}{{\mathcal G}}

\newcommand{\cV}{{\mathcal V}}
\newcommand{\cT}{{\mathcal T}}

\newcommand{\cI}{{\mathcal I}}
\newcommand{\cL}{{\mathcal L}}

\newcommand{\cP}{{\mathcal P}}
\newcommand{\cO}{{\mathcal O}}

\newcommand{\Hom}{{\mathrm {Hom}}}

\bibliography{biblio}
\begin{document}

\title{Waldspurger formulas in higher cohomology}
\author{Santiago Molina}

\newcommand{\Addresses}{{
		\bigskip
		\footnotesize
	
		\medskip
		Santiago Molina; Universitat de Lleida\\Campus Universitari Igualada - UdL
            Av. Pla de la Massa, 8\\
            08700 Igualada, Spain\par\nopagebreak
		\texttt{santiago.molina@udl.cat}
		
}}

\maketitle

\begin{abstract}
The classical Waldspurger formula, which computes integral periods of quaternionic automorphic forms in maximal torus, has been used in a wide variety of arithmetic applications, such as the Birch and Swinnerton-Dyer conjecture in rank 0 situations. This is why this formula is considered the rank 0 analogue of the celebrated Gross-Zagier formula.

On the other hand, the Harder-Eichler-Shimura correspondence allows us to interpret this quaternionic automorphic form as a cocycle in group cohomology spaces of certain arithmetic groups. In this way we can realize the corresponding automorphic representation in the \'etale cohomology of certain Shimura varieties. In this work we find a formula, analogous to that of Waldspurger, which relates  cap-products of this cocycle and canonical fundamental classes associated with maximal torus with special values of Rankin-Selberg L-functions.
\end{abstract}



\section{Introduction}

Let $F$ be any number field, and let $E/F$ be any quadratic extension of relative discriminant $D$. Write $\eta_T$ for the quadratic character associated with $E/F$. Let $B$ be a quaternion algebra over $F$ admitting an embedding $E\hookrightarrow B$. Let $\pi$ be an irreducible automorphic representation of $B^\times$ and let $\Pi$ be its Jacquet-Langlands lift to $\GL_2/F$.
The classical Waldspurger formula computes the period integral of an automorphic form $f\in \pi$ twisted by a Hecke character $\chi$ of $E^\times$
in terms of a critical value of the Rankin-Selberg L-function $L(s,\Pi,\chi)$ associated with $\Pi$. 
More concretely, if $d^\times t=\prod_vd^\times t_v$ is the usual Tamagawa measure of $\A_E^\times/E^\times\A_F^\times$ and we write
\[
\ell(f,\chi)=\int_{\A_E^\times/E^\times\A_F^\times}\chi(t)f(t)d^\times t,
\]
then the formula reads in one of the best known forms:
\begin{equation}\label{WFintro}
    \frac{\ell(f_1,\chi)\cdot \ell(f_2,\chi^{-1})}{\langle f_1,f_2\rangle}=\frac{\Lambda_F(2)\cdot\Lambda(1/2,\Pi,\chi)}{2\Lambda(1,\Pi,{\rm ad})}\cdot\prod_v\beta_v(f_{1,v},f_{2,v}),\qquad f_1,f_2\in\pi,
\end{equation}
where $\Lambda$ stands for the complete global L-functions, that is, those defined by the
 product of local L-functions over all places, $\Lambda_F(s)=\zeta_F(s)\prod_{v\mid\infty}\zeta_{F_v}(s)$ is the complete L-series for the trivial character,
\begin{equation*}
     \beta_v(f_{1,v},f_{2,v}):=\frac{L(1,\eta_{T,v})\cdot L(1,\Pi_v,{\rm ad})}{\zeta_{v}(2)\cdot L(1/2,\Pi_v,\chi_v)}\int_{E_v^\times/F_v^\times}\chi_v(t_v)\frac{\langle\pi_v(t_v)f_{1,v},f_{2,v}\rangle_v}{\langle f_{1,v},f_{2,v}\rangle_v} d^\times t_v,
\end{equation*}
and $\langle\;,\;\rangle=\prod_v\langle\;,\;\rangle_v$ is the natural inner product attached to the usual Tamagawa measure 
of $B^\times\A_F^\times\backslash (B\otimes_F\A_F)^\times$
\begin{equation}\label{defpetersson}
    \langle f_1,f_2\rangle=\int_{B^\times\A_F^\times\backslash (B\otimes_F\A_F)^\times}f_1(g)f_2(g)d^\times g.
\end{equation}
The reader will find in \S \ref{classWalds} more details on these classical formulas.

Waldspurger formula has many very important arithmetic applications. When $F=\Q$, $\Pi$ is associated with an elliptic curve, $E/\Q$ is imaginary and $B$ is  definite, one can relate the corresponding period integrals with the reduction of certain points in the elliptic curve called Heegner points. This has been used to relate certain Euler systems with critical values of the $L$-function $L(s,\Pi)$, providing as a consequence rank zero instances of the Birch and Swinnerton-Dyer conjecture. Indeed, one can think of Waldspurger formula as a rank zero counterpart of the celebrated Gross and Zagier formula, which is one of the most important breakthroughs towards the Birch and Swinnerton-Dyer conjecture.

Let $\Sigma_B$ be the set of split archimedean places of $B$, namely, the set of archimedean completions $F_\sigma\subseteq\C$ where $B\otimes_{F} F_\sigma=\M_2(F_\sigma)$, and write $r=\#\Sigma_B$. The Harder-Eicher-Shimura isomorphism realizes the automorphic representation $\pi$ in the $r$-cohomology space of certain arithmetic groups attached to our quaternion algebra $B$. We proceed to give a more precise description of such cohomology spaces: We will assume throughout this paper that $\pi$ has trivial central character. Let $G$ be the algebraic group associated to $B^\times/F^\times$, hence, $\pi$ can be seen as a representation of $G(\A_F)$. For any $\underline{k}-2=(k_{\tilde\sigma}-2)_{\tilde\sigma}\in (2\N)^{[F:\Q]}$, let us consider the $\C$-vector spaces
\[
V(\underline{k}-2):=\bigotimes_{{\tilde\sigma}:F\hookrightarrow\C}{\rm Sym}^{k_{\tilde\sigma}-2}(\C^2).
\]
Then $V(\underline{k}-2)$ is endowed with a natural action of $G(F)$. For any open compact subgroup $U\subset G(\A_F^\infty)$ of the finite adelic points of $G$, we consider the group cohomology spaces
\begin{equation*}
H^r(G(F)_+,\cA^\infty(V(\underline{k}-2))^U)=\bigoplus_{g\in {\rm Pic}(U)} H^r(\Gamma_g,V(\underline{k}-2)), \qquad {\rm Pic}(U)=G(F)_+\backslash G(\A_F^\infty)/U,
\end{equation*}
where $\Gamma_g=G(F)_+\cap gUg^{-1}$ and $G(F)_+$ is the subgroup of elements with positive norm at all real places. 
The Harder-Eichler-Shimura morphism realizes the automorphic representation $\pi^\infty=\pi\mid_{G(\A_F^\infty)}$ of weight $\underline{k}$ in 
\begin{equation*}
H^r_\ast(G(F)_+,\cA^\infty(V(\underline{k}-2)))^{\lambda}=\bigcup_{U\subset G(\A_F^\infty)}H^r(G(F)_+,\cA^\infty(V(\underline{k}-2))^U)^{\lambda},
\end{equation*}
where the superindex given by a fixed character $\lambda:G(F)/G(F)_+\rightarrow\pm 1$ stands for the corresponding $\lambda$-isotypical component. Such a realization is provided by a Hecke-equivariant morphism 
\[
{\rm ES}_\lambda:M_{\underline k}(U)\longrightarrow H^r(G(F)_+,\cA^\infty(V(\underline{k}-2))^U)^\lambda,
\]
where $M_{\underline k}(U)$ is the space of modular forms for $G$ of weight ${\underline k}$ and level $U$ (see definition \ref{defAFweightk}). 

For many arithmetic applications, it is more convenient to consider the realization of $\pi^\infty$ in the above cohomology spaces, 
rather than in the space of automorphic forms. Indeed, since the finite dimensional complex vector spaces $V(\underline{k}-2)$ admit rational models, one can show that $\pi^\infty$ is in fact the extension of scalars of a representation defined over a number field, called the coefficient field of $\pi$. 

On the other hand, canonical homology classes associated to quadratic extensions $E/F$ can be defined using the group of relative units of $E^\times$. Such fundamental classes lie in an $r$-th-homology group ($r=\#\Sigma_B$) if we make the following hypothesis:
\begin{assumption}\label{assuSigmaSigma}
The set $\Sigma_B$ coincides with the set of archimedean places $\sigma$ where $E$ splits. We will write 
\[
\Sigma_B=\Sigma_T^\R\cup\Sigma_T^\C;\qquad \Sigma_T^\R=\{\sigma\mid\infty;\;E_\sigma=\R\times\R\},\qquad\Sigma_T^\C=\{\sigma\mid\infty;\;F_\sigma=\C\},
\]
and $r_\R=\#\Sigma_T^\R$, $r_\C=\#\Sigma_T^\C$ with $r=r_\R+r_\C$.
\end{assumption}
Throughout the paper the previous assumption will be fulfilled, hence, we will be able to construct a fundamental class 
\[
\eta\in H_r(T(F)_+,C^0_c(T(\A_F^\infty),\Z)), 
\]
where $T$ is the algebraic group associated with $E^\times/F^\times$, $C^0$ denotes the subspace of locally constant functions, $C^0_c$ those with compact support, and $T(F)_+$ is again the subgroup of elements with positive norm at all real places. 

As pointed out above, there are many arithmetic applications where both a cohomology class $\phi^\lambda\in\pi^\infty\subseteq {\rm Im}({\rm ES}_\lambda)$ and the fundamental class $\eta$ play a very important role. For instance, in \cite{Darmon2001}, \cite{guitart2017automorphic}, \cite{fornea2021plectic} and \cite{HerMol1}, classes $\phi^\lambda$ and $\eta$ in certain precise situations are used to construct $p$-adic points in elliptic curves over $F$. These are the classical Stark-Heegner points in rank one situations, and the more recent plectic points constructed in higher rank situations. It is conjectured that these $p$-adic points are, in fact, defined over precise abelian extensions of $E$. Indeed, much progress towards the algebraicity of Stark-Heegner points over $\Q$ has been made recently in \cite{BDRSV}. 
The strategy of \cite{BDRSV} is based on establishing a connection between these points and other key objects that are the $p$-adic L-functions.
An anticyclotomic $p$-adic L-function is a $p$-adic avatar of the classical L-function $L(s,\Pi,\chi)$. In \cite{fornea2021plectic}, \cite{HerMol1} and \cite{HM2} anticyclotomic $p$-adic L-functions are constructed using $\phi^\lambda$ and $\eta$. The link between such $p$-adic L-functions and the classical  $L(s,\Pi,\chi)$ is manifested with the so-called interpolation formulas. In fact, the main theorem of this paper, theorem \ref{mainTHMintro}, is used in  \cite{HerMol1} and \cite{HM2} to obtain explicit interpolation formulas. There is a close relationship between Stark-Heegner and plectic points and anticyclotomic $p$-adic L-functions. Indeed in \cite{Bertolini1998}, \cite{fornea2021plectic}, \cite{HerMol1} and \cite{HM2} 
several formulas relating these special points and higher derivatives of the $p$-adic L-functions are obtained in the spirit of the Birch and Swinnerton-Dyer conjecture.

The embedding $E\hookrightarrow B$ provides an inclusion $T\hookrightarrow G$ as algebraic groups. Moreover, we have a natural $T(F)_+$-equivariant pairing maintaining $\lambda$-isotypical components (see remark \ref{canonicalphi} for further details)
\begin{equation}\label{introvarphi}
\varphi:\cA^\infty(V(\underline{k}-2))\times \left(C^0(T(\A_F^\infty),\C)\otimes V(\underline{k}-2)\right)\longrightarrow C^0(T(\A_F),\C).    
\end{equation}
We will regard $C^0(T(\A_F^\infty),\C)\otimes V(\underline{k}-2)$ as a subspace of locally polynomial functions in $T(\A_F)$. Since the Tamagawa measure on $T(\A_F)$ provides a pairing between $C^0(T(\A_F),\C)$ and $C_c^0(T(\A_F^\infty),\Z)$,  for any $T(F)$-invariant locally polynomial character $\chi\in H^0(T(F),C^0(T(\A_F^\infty),\C)\otimes V(\underline{k}-2))$, we can consider the cup product
\[
\cP(\phi^\lambda,\chi):=\varphi(\phi^\lambda\cup\chi)\cap\eta\in \C.
\]
We can think of $\cP(\phi^\lambda,\chi)$ as an analogue in higher cohomology of the period integral $\ell(f,\chi)$. In fact, if $\Sigma_B=\emptyset$ both concepts coincide. Thus, it is natural to wonder if there is a formula analogous to that of \eqref{WFintro} computing $\cP(\phi^\lambda,\chi)$, and that is precisely what main result of this paper is about. We will give a formula for $\cP(\phi^\lambda,\chi)$ in an explicit form, rather than express it in terms of local factors $\beta_v$ as in \eqref{WFintro}. 
Indeed, one can find in the literature explicit versions of \eqref{WFintro} where the local factors $\beta_{v}(f_{1,v},f_{2,v})$ are computed at the cost of specifying the forms $f_i$ 
(see  \cite[theorem 1.8]{CST}). Let $N$ be the conductor of $\pi$ and let $c$ be the conductor of $\chi$.
Although our main result is completely general, to avoid technical details in this introduction we will restrict ourselves to the case that, for all finite places $v$, either $\ord_v(c)\geq \ord_v(N)$ or $\ord_v(c)=0$ with $\ord_v(N)\leq 1$ if $E_v$ is non-split. Under this hypothesis the presentation of the result is much simplified. If for all finite places $v$ the local root numbers $\epsilon(1/2,\pi_v,\chi_v)=\eta_{T,v}(-1)\epsilon(B_v)$, where $\epsilon(B_v)=1$ if $B_v$ is a matrix algebra and $\epsilon(B_v)=-1$ otherwise, then we can define an (admissible) Eichler order $\cO_N\subset B$ of discriminant $N$ such that
 the space of $U=\hat\cO_N^\times$-invariant elements of $\pi^\infty$ is one dimensional. 
If this is the case, we will fix non-zero $\phi_{0}^\lambda\in (\pi^\infty)^U$. 
In \S \ref{classWalds} we will introduce natural pairings 
\[
\langle\;,\;\rangle: M_{\underline k}(U)\times M_{\underline k}(U)\longrightarrow \C,
\]
closely related with the classical Petersson inner product (see remark \ref{classesASHilbertform}). At this point we have all the ingredients to state a simplified version of the main result of the paper (theorem \ref{THMwaldsHC1}):
\begin{theorem}\label{mainTHMintro}
Assume that 
$\chi:T(\A_F)/T(F)\rightarrow\C^\times$ is a locally polynomial character of conductor $c$ such that  \[
\chi\mid_{T(F_\infty)}(t)=\chi_0(t)\prod_{\tilde\sigma:F\hookrightarrow\C}t_{\tilde\sigma}^{m_{\tilde\sigma}},\qquad \underline{m}=(m_{\tilde\sigma})\in \Z^{[F:\Q]},\qquad \frac{2-k_{\tilde\sigma}}{2}\leq m_{\tilde\sigma}\leq\frac{k_{\tilde\sigma}-2}{2},
\] 
for some locally constant character $\chi_0$. 
Given 
$\phi^\lambda
\in\pi^\infty\subset H_\ast^r(G(F),\cA^\infty(V(\underline{k}-2))(\lambda))$ ,
we have $\cP(\phi^\lambda,\chi)=0$ unless $\chi_{0}=\lambda$ and the root number $\epsilon(1/2,\pi_v,\chi_v)=\chi_v\eta_{T,v}(-1)\epsilon(B_v)$, for all $v\nmid\infty$. Moreover, if this is the case and $\phi_1^\lambda$ and $\phi_2^\lambda$ differ from $\phi_{0}^\lambda$ in a finite set of places $\mathfrak{S}$ then  
\begin{equation}\label{eqmainintro}
\cP(\phi_1^\lambda,\chi)\cdot \cP(\phi_2^\lambda,\chi^{-1})=\frac{2^{\#S_D}L_{c}(1,\eta_{T})^2h^2
C(\underline k,\underline m)}{|c^2 D|^{\frac{1}{2}}}\cdot L^S(1/2,\Pi,\chi)\cdot\frac{\langle \Phi_1,\Phi_2\rangle}{\langle \Psi,\Psi\rangle}\cdot\frac{{\rm vol}(U_0(N))}{{\rm vol}(U)}\prod_{v\in\mathfrak{S}}\frac{\beta_{v}(\phi_{1,v}^\lambda,\phi_{2,v}^\lambda)}{\beta_{v}(\phi_{0,v}^\lambda,\phi_{0,v}^\lambda)},
\end{equation}
where $S:=\{v\mid (N,c)\}$, $S_D:=\{v\mid (N,D);\;\ord_v(c)=0\}$, $L^S(s,\Pi,\chi)$ is the L-function with the local L-functions at places in $S\cup\infty$ removed, and $L_{c}(s,\eta_{T})$ is the product of the local L-functions at places dividing $c$, 
    \[
    C(\underline k,\underline m)=(-1)^{\left(\sum_{\sigma\not\in\Sigma_B}\frac{k_{\sigma}-2}{2}\right)}4^{r_\R}\cdot(32\pi)^{r_\C}\left(\frac{1}{\pi}\right)^{d-r-r_\C}\prod_{\tilde\sigma:F\hookrightarrow\C}\frac{\Gamma(\frac{k_{\tilde\sigma}}{2}-m_{\tilde\sigma})\Gamma(\frac{k_{\tilde\sigma}}{2}+m_{\tilde\sigma})}{(-1)^{m_{\tilde\sigma}}(2\pi)^{k_{\tilde\sigma}}},
    \]
    $U_0(N)\subset\PGL_2(\A_F)$ is the usual open compact subgroup of upper triangular matrices modulo $N$, $\Psi\in M_{\underline k}(U_0(N))$ is a normalized form generating $\Pi$,
    the forms $\Phi_i$ are such that ${\rm ES}_\lambda(\Phi_i)=\phi_i^\lambda$
and $h$ is the cardinal of the subgroup in $T(\cO_F)$ of torsion elements with positive norm at all archimedean places. 
\end{theorem}
In order to prove this formula, we will relate the value of $\cP(\phi^\lambda,\chi)$ to a period $\ell(f_{\phi^\lambda},\chi)$, for some automorphic form $f_{\phi^\lambda}\in \pi$ associated with $\phi^\lambda$ via the explicit description of the Harder-Eicher-Shimura morphism given in \cite{ESsanti}, and then we will apply the classical Waldspurger formula. We would like to highlight that, to obtain our result one has to compute the local terms $\beta_\sigma(f_{\phi_2^\lambda,\sigma},f_{\phi_1^\lambda,\sigma})$ for all archimedean places $\sigma$. To accomplish this, one needs the deep analysis of the archimedean representation $\pi_\sigma$ performed in \S \ref{loctheory}.

In addition to what was mentioned above regarding interpolation properties for anticyclotomic $p$-adic L-functions, we believe that our formula can have other interesting arithmetic applications.
In fact, it has been recently used in \cite{Venk} to show algebraicity of Stark-Heegner cycles associated to Bianchi modular forms that are base-change of classical elliptic cuspforms. Moreover, the formula can also be used to explore rationality of critical values of L-functions. Indeed, in case that $F$ is totally real and $G=\PGL_2$, by assumption \ref{assuSigmaSigma} the field $E$ must also  be totally real. 
Since all Eichler orders of level $N$ are locally conjugated, we have in this situation that $U=\gamma U_0(N)\gamma^{-1}$ for some $\gamma\in \PGL_2(\A_F^\infty)$, hence we can choose $\Phi=\Phi_1=\Phi_2=\gamma\Psi$. We will see that in this situation $\Psi$ corresponds to a normalized Hilbert newform (see remark \ref{classesASHilbertform}). Furthermore, in \cite{ESsanti} we prove that ${\rm ES}_\lambda(\Psi)$ coincides the classical construction of the Eichler-Shimura isomorphism (see remark \ref{difES}). Thus, by the theory of modular symbols 
\[
\frac{{\rm ES}_\lambda(\Psi)}{\Omega_\lambda}\in H^r(G(F)_+,\cA^\infty(V(\underline{k}-2)_{\bar\Q})^U)^\lambda=\bigoplus_{g\in {\rm Pic}(U)} H^r(\Gamma_g,V(\underline{k}-2)^\lambda_{\bar\Q});\qquad V(\underline{k}-2)_{\bar\Q}:=\bigotimes_{{\tilde\sigma}:F\hookrightarrow\bar\Q}{\rm Sym}^{k_{\tilde\sigma}-2}(\bar\Q),
\]
where $\Omega_\lambda$ is the classical period associated with $\Psi$. 
Since the morphism $\varphi$ of \eqref{introvarphi} preserves models over $\bar\Q$, we deduce using \eqref{eqmainintro} that if  $\chi\mid_{T(F_\infty)}(t)=\lambda(t)\cdot\prod_\sigma t^{m_\sigma}$ then
\[
\Omega_\lambda^{-2}\cdot \cP({\rm ES}_\lambda(\Phi),\chi)\cdot \cP({\rm ES}_\lambda(\Phi),\chi^{-1})\in\bar\Q;\qquad \mbox{and therefore}\qquad \Omega_\lambda^{-2}\cdot L(1/2,\Pi,\chi)\cdot\prod_\sigma(2\pi)^{-k_\sigma}\in\bar\Q.
\]
Although this result preceded by the work of Shimura over $\Q$ 
was already known (see for example \cite{JST}), 
one can try to replicate the above argument for other quaternion algebras. 
Indeed, if $G\neq\PGL_2$ one can normalize the $U$-invariant $\Phi\in (\pi^\infty)$ so that $\langle\Phi,\Phi\rangle=\langle\Psi,\Psi\rangle$, and define the period $\Omega_\lambda^G$ in such a way that ${\rm ES}_\lambda(\Phi)/\Omega_\lambda^{G} \in H^r(G(F)_+,\cA^\infty(V(\underline{k}-2)_{\bar\Q})^U)^\lambda$. Thus, if we assume that $\chi=\chi_F\circ{\rm Norm}_{E/F}$, for some quadratic character $\chi_F:\A_F^\times/F^\times\rightarrow\pm 1$, then by \eqref{eqmainintro}
\[
(\Omega_\lambda^G)^{-2}\cdot C(\underline k,0)\cdot L(1/2,\Pi,\chi)=(\Omega_\lambda^G)^{-2}\cdot C(\underline k,0)\cdot L(1/2,\Pi,\chi_F)\cdot L(1/2,\Pi,\chi_F\eta_T)\in\bar\Q,
\]
where the second equality follows from Artin formalism. In the work in progress \cite{GM} we expect to use the aforementioned link between $L(1/2,\Pi,\chi_F)$ and $\Omega_\lambda$ to obtain a correlation between $\Omega_\lambda^G$ and $\Omega_\lambda$. This could lead to new instances of the Oda conjecture stated in \cite{Oda}.

\subsection*{Acknowledgements.}
This project has received funding from the European Research Council
(ERC) under the European Union's Horizon 2020 research and innovation
programme (grant agreement No. 682152), and the projects PID2021-124613OB-I00 and PID2022-137605NB-I00 from Ministerio de Ciencia e innovaci\'on. I would also like to thank Guhan Venkat for his interesting comments.

\subsection{Notation}\label{not}

We write $\hat\Z:=\prod_{\ell}\Z_\ell$ and $\hat R:=R\otimes\hat\Z$, for any ring $R$.
Let $F$ be a number field with integer ring $\cO_F$, and let $\A_F$ and $\A_F^\infty=\hat\cO_F\otimes\Q$ be its ring of adeles and finite adeles, respectively. For any place $v$ of $F$, we write $F_v$ for the completion at $v$. Write $\Sigma_F$ for the set of archimedean places of $F$. We will sometimes write $\sigma\mid \infty$ to denote $\sigma\in\Sigma_F$. For any embedding $\tilde\sigma:F\hookrightarrow\C$ whose equivalence class is $\sigma\in\Sigma_F$, we will write $\tilde\sigma\mid\sigma$. We define 
\[
\Z^{\Sigma_F}:=\left\{\underline{k}=(k_\sigma)_{\sigma\mid\infty}; k_\sigma=(k_{\tilde\sigma})_{\tilde\sigma\mid\sigma}\in \Z^{[F_\sigma:\R]}\right\}\simeq\Z^{[F:\Q]}.
\]
Given $\underline{k}\in\Z^{\Sigma_F}$ and $x\in F_\infty=\prod_{\sigma\mid\infty}F_\sigma$, we write 
\[
x^{\underline{k}}:=\prod_{\sigma\mid\infty}\prod_{\tilde\sigma\mid\sigma} \tilde\sigma(x)^{k_{\tilde\sigma}}.
\]

Let $B$ be a quaternion algebra over $F$ with maximal order $\cO_B$.
Let $G$ be the algebraic group associated with the group of units of $B$ modulo scalars, namely, for any $\cO_F$-algebra $R$
\[
G(R):=(\cO_B\otimes_{\cO_F}R)^\times/R^\times.
\]
We denote by $\Sigma_B$ the set of split infinite places of $B$. For any $\sigma\mid\infty$, we write $G(F_\sigma)_+\subseteq G(F_\sigma)$ for the subgroup of elements of positive norm. We also write $G(F_\infty)_+:=\prod_{\sigma\mid\infty}G(F_\sigma)_+$ and $G(F)_+:=G(F_\infty)_+\cap G(F)$.

Assume that we have a fixed embedding $E\hookrightarrow B$ such that $\cO_{c_0}=E\cap \cO_B$ is an order of conductor $c_0$. We write 
\[
T(R):=(\cO_{c_0}\otimes_{\cO_F}R)^\times/R^\times.
\]
This implies that $T\subset G$ as algebraic groups. 

Write $\Sigma_T$ for the set of split infinite places of $E/F$. Thus,
\[
\Sigma_T=\Sigma_{T}^{\R}\cup\Sigma_T^\C,\qquad \Sigma_{T}^{\R}=\{\sigma\mid\infty,\;T(F_\sigma)=\R^\times\},\quad\Sigma_{T}^{\C}=\{\sigma\mid\infty,\;T(F_\sigma)=\C^\times\}.
\]
For any archimedean place $\sigma$, write $T(F_\sigma)_0\subset T(F_\sigma)$ to be the intersection of all connected subgroups $N$ for which the quotient $T(F_\sigma)/N$ is compact. Write also $T(F_\sigma)_+$ for the connected component of $1$ in $T(F_\sigma)$. We can visualize $T(F_\sigma)$, $T(F_\sigma)_+$, $T(F_\sigma)_0$ depending on the ramification type of $\sigma$ in the following table:
\[\]
\begin{center}
\begin{tabular}{ c| c c c c }
ramification &$T(F_\sigma)$ & $T(F_\sigma)_+$&$T(F_\sigma)_0$&$T(F_\sigma)/T(F_\sigma)_0$ \\ 
 \hline
 $\Sigma_{T}^{\R}$ & $\R^\times$ & $\R_+$& $\R_+$& $\pm 1$ \\  
 $\Sigma_{T}^{\C}$ & $\C^\times$ & $\C^\times$ & $\R_+$ & $S^1$ \\
  $\infty\setminus\Sigma_{T}$ & $\C^\times/\R^\times$ & $\C^\times/\R^\times$ & $1$ & $\C^\times/\R^\times$ 
\end{tabular}
\end{center}
\[\]

Given $t_\infty\in T(F_\infty)=E^\times_\infty/F^\times_\infty$, write $\tilde t_\infty\in E_\infty^\times$ for a representative. If $E_\sigma/F_\sigma$ splits, write $\lambda(\tilde t_\sigma)_1, \lambda(\tilde t_\sigma)_2 \in F_\sigma^\times$ for the two components of $\tilde t_\sigma\in F_\sigma^\times\times F_\sigma^\times$. If $E_\sigma/F_\sigma$ does not split, write $\lambda(\tilde t_\sigma)_1, \lambda(\tilde t_\sigma)_2 \in E_\sigma^\times $ for the image of $\tilde t_\sigma$ under the two $F_\sigma$-isomorphisms. In this last non-split case, we choose an embedding $\tilde\sigma_E:E_\sigma\rightarrow\C$ above $\tilde\sigma$. For any $\underline{m} \in	\Z^{\Sigma_F}$, we write
\[%
t_\infty^{\underline{m}}:=\prod_{\sigma\mid\infty}\prod_{\tilde\sigma\mid\sigma}\tilde\sigma_E\left(\frac{\lambda(\tilde t_\sigma)_1}{\lambda(\tilde t_\sigma)_2}\right)^{m_{\tilde\sigma}}\in\C.
\]

If $v$ is non-archimedean, write $\cO_{F_v}$ for the  integer ring of $F_v$, let ${\rm ord}_v(\cdot)$ be its valuation, $\kappa_v$ its residue field, $\varpi_v$ a fixed uniformizer, $d_{F_v}$ its different over $\Q_p$ and $q_v=\#\kappa_v$.

\subsection{Haar Measures}\label{haarmeasures}
For any number field $F$, let us consider the additive character $\psi:\A_F/F\rightarrow\R$,
\begin{equation}\label{defpsiad}
    \psi=\prod_v\psi_v,\qquad  \psi_v(a)=\left\{\begin{array}{lc}
       e^{2\pi i a},  &F_v=\R;  \\
        e^{2\pi i{\rm Re}(a)}, &F_v=\C;\\
        e^{-2\pi i[{\rm Tr}_{F_v/\Q_p}(a)]},&F_v/\Q_p;
    \end{array}\right.
\end{equation}
where $[\cdot]:\Q_p/\Z_p\hookrightarrow\Q/\Z$ is the natural inclusion. Let $dx_v$ be the Haar measure of $F_v$ normalized so that it is self-dual with respect to $\psi_v$, namely, $\hat{\hat\phi}(x_v)=\phi(-x_v)$, where $\hat\phi$ is the Fourier transform $\hat\phi(y_v)=\int_{F_v}\phi(x_v)\psi_v(x_vy_v)dx_v$.
Notice that $dx_v$ is $[F_v:\R]$ times the usual Lebesgue measure, if $v|\infty$, and the Haar measure satisfying ${\rm vol}(\cO_{F,v})=|d_{F_v}|_v^{1/2}$, if $v\nmid\infty$.
In any of these cases, we define $d^\times x_v=\zeta_v(1)|x_v|_v^{-1}dx_v$,
where $\zeta_v(s)=(1-q_v^{-s})^{-1}$, if $v\nmid\infty$, $\zeta_v(s)=\pi^{-s/2}\Gamma(s/2)$, if $F_v=\R$, and $\zeta_v(s)=2(2\pi)^{-s}\Gamma(s)$, if $F_v=\C$. One easily checks that, if $v$ is non-archimedean then ${\rm vol}(\cO_{F_v}^\times)=|d_{F_v}|_v^{1/2}$ as well.
The product of $d^\times x_v$ provides a Tamagawa measure $d^\times x$ on $\A_F^\times/F^\times$. In fact, such Haar measure satisfies
\[
{\rm Res}_{s=1}\int_{x\in\A_F^\times/F^\times,\;|s|\leq 1}|x|^{s-1}d^\times x={\rm Res}_{s=1}\Lambda_F(s),
\]
where $\Lambda_F(s)=\zeta_F(s)\prod_{v\mid\infty}\zeta_{v}(s)$ is the completed Riemann zeta function associated with $F$. This implies that, if we choose $d^\times t$ to be the quotient measure for $T(\A_F)/T(F)=\A_E^\times/\A_F^\times E^\times$, one has that ${\rm vol}(T(\A_F)/T(F))=2L(1,\eta_T)$.

Let us consider the Haar measure $db_v$ of $B_v$ which is self-dual with respect to $\psi_v$, in the sense that $\hat{\hat\phi}(b_v)=\phi(-b_v)$, where $\hat\phi$ is the Fourier transform $\hat\phi(a_v)=\int_{B_v}\phi(b_v)\psi_v(a_v\bar b_v+b_v\bar a_v)db_v$ and $(b_v\mapsto\bar b_v)$ is the usual involution on $B_v$. We define similarly as above
$d^\times g_v=\zeta_v(1)|g_v\bar g_v|_v^{-1}dg_v$. The product of such $d^\times g_v$ provides a Tamagawa measure for $G$ satisfying ${\rm vol}(G(F)\backslash G(\A_F))=2$ and (see \cite[lemma 3.5]{CST})
\[
\begin{array}{ll}
   {\rm vol}(\PGL_2(\cO_{F_v}))=\zeta_v(2)^{-1}|d_{F_v}|_v^{3/2},  &v\nmid\infty,\quad B_v=\M_2(F_v);  \\
   {\rm vol}(\cO_{B_v}^\times/\cO_{F_v}^\times)=\zeta_v(2)^{-1}(q_v-1)^{-1}|d_{F_v}|_v^{3/2},  & v\nmid\infty,\quad B_v\neq\M_2(F_v);\\
   {\rm vol}(B_\sigma^\times/F_\sigma^\times)=2\pi^2,&\sigma\mid\infty,\quad B_\sigma\neq\M_2(F_\sigma);
\end{array}
\]
Moreover, if $\sigma\in\Sigma_B$ and $F_\sigma=\R$, the measure $d^\times g_\sigma$ corresponds to 
\[
d^\times g_\sigma=\frac{dxdyd\theta}{y^2},\qquad g_\sigma=\big(\begin{smallmatrix}1&x\\&1\end{smallmatrix}\big) \big(\begin{smallmatrix}y^{1/2}&\\&y^{-1/2}\end{smallmatrix}\big)\big(\begin{smallmatrix}\cos\theta&\sin\theta\\-\sin\theta&\cos\theta\end{smallmatrix}\big),\qquad \theta\in [0,\pi).
\]

\subsection{Finite dimensional representations}\label{findimreps}

Let $k$ be a positive even integer.
Let $\cP(k)={\rm Sym}^k(\C^2)$ be the space of polynomials in 2 variables homogeneous of degree $k$ with $\PGL_2(\C)$-action:
\[
\left(\left(\begin{array}{cc}a&b\\c& d\end{array}\right)P\right)(X,Y)=(ad-bc)^{-\frac{k}{2}}P(aX+cY,bX+dY),\qquad P\in\cP(k).
\]  
Let us denote $V(k)=\cP(k)^\vee$ with dual $\PGL_2(\C)$-action:
\[
(g\mu)(P)=\mu(g^{-1}P),\qquad \mu\in V(k).
\]  
Notice that $V(k)\simeq \cP(k)$ by means of the isomorphism
\begin{equation}\label{dualVP}
V(k)\longrightarrow\cP(k),\qquad\mu\longmapsto\mu((Xy-Yx)^k).
\end{equation}

\subsubsection{Polynomials and torus}

The fixed embedding $\iota:E\hookrightarrow B$ provides an isomorphism $B\otimes_FE\simeq{\rm M}_2(E)$ that we will fix throughout the article. Indeed, $\iota$ induces a decomposition $B=E\oplus EJ$, where $J$ normalizes $E$ and $J^2\in F^\times$. Hence, we have the corresponding embedding
\begin{equation}\label{embEinB}
    B\hookrightarrow\M_2(E);\qquad e_1+e_2J\longmapsto \left(\begin{array}{cc}
    e_1 & J^2e_2 \\
    \bar e_2 & \bar e_1
\end{array}\right),
\end{equation}
where $(e\mapsto\bar e)\in \Gal(E/F)$ denotes the non-trivial automorphism.
For a given $\bar\sigma:F\hookrightarrow\C$, the composition $B\hookrightarrow B\otimes_FE\simeq {\rm M}_2(E)$ together with the fixed extension $\bar\sigma_E:E\hookrightarrow\C$ of $\bar\sigma$, gives rise to an embedding $G(F_\sigma)\hookrightarrow\PGL_2(\C)$, if $\bar\sigma\mid\sigma\mid\infty$. This provides an action of $G(F_\infty)$ on $V(\underline k)$ and $\cP(\underline k)$, for any $\underline k=(k_{\bar\sigma})\in (2\N)^{\Sigma_F}$. 

The composition $E\stackrel{\iota}{\hookrightarrow}B\hookrightarrow B\otimes_FE\simeq{\rm M}_2(E)$ maps $e$ to $\big(\begin{smallmatrix} e&\\&\bar e\end{smallmatrix}\big)$. Thus, we have a $T(F_\infty)$-equivariant morphism
\begin{equation}\label{PtoCoverC}
    \cP(\underline k)\longrightarrow C(T(F_\infty),\C);\qquad \bigotimes_{\bar\sigma} P_{\bar\sigma}\longmapsto\ipa{(t_\sigma)_{\sigma\mid\infty}\mapsto\prod_{\sigma\mid\infty}\prod_{\bar\sigma\mid\sigma}P_{\bar\sigma}\ipa{1,\bar\sigma_E\ipa{\frac{t_\sigma}{\bar t_\sigma}}}\bar\sigma_E\ipa{\frac{t_\sigma}{\bar t_\sigma}}^{-\frac{k_{\bar\sigma}}{2}}}.
\end{equation}

\section{Fundamental classes}

Let $\Sigma\subseteq\Sigma_F$ be a set of archimedean places, write $T(F_\Sigma)_0:=\prod_{\sigma\in\Sigma}T(F_\sigma)_0$ and let us consider the space of continuous functions
\[
C^\Sigma(T(\A_F),\C)=\{f\in C(T(\A_F),\C),\;f\mid_{T(F_\infty)}\mbox{ is $\cC^\infty$ and $T(F_{\Sigma_F\setminus\Sigma})_0$-invariant}\},
\]
with its natural $T(F)$-action.

For any $\sigma\in \Sigma_T$, the space $T(F_\sigma)$ is a $\R$-Lie group of dimension 1 or 2 depending if $\sigma\in \Sigma_T^\R$ or $\Sigma_T^\C$. In each of the cases, we can identify $\cL_{\sigma}={\rm Lie}(T(F_\sigma))$ with $\R$ or $\C$. The element $\delta_\sigma:=1\in \cL_{\sigma}$ generates the Lie algebra of $T(F_\sigma)_0$. This claim is clear for $\sigma\in\Sigma_T^\R$,  while if $\sigma\in\Sigma_T^\C$ it follows from
\begin{equation}\label{defdelta}
\delta_\sigma f(re^{i\theta})=\left.\frac{d}{dt}f(re^{i\theta}e^{t})\right|_{t=0}=r\frac{\partial}{\partial r}f(re^{i\theta}), \qquad f\in C^{\Sigma_F}(T(F_\sigma),\C).
\end{equation}
Notice that, for any $\sigma\in\Sigma$, the following $T(F)$-equivariant sequence is exact
\begin{equation}\label{exseqCs}
0\longrightarrow C^{\Sigma\setminus\sigma}(T(\A_F),\C)\longrightarrow C^\Sigma(T(\A_F),\C)\stackrel{\delta_\sigma}{\longrightarrow} C^\Sigma(T(\A_F),\C)\longrightarrow 0.
\end{equation}

\begin{definition}\label{partialT}
The composition of the connection morphisms of the corresponding long exact sequences in cohomology provides a morphism
\begin{eqnarray*}
&\partial:& H^0(T(F),C^{\Sigma_T}(T(\A_F),\C))\longrightarrow H^1(T(F),C^{\Sigma_T\setminus\sigma_1}(T(\A_F),\C))\longrightarrow\cdots\\
&\cdots&\longrightarrow H^{u-1}(T(F),C^{\sigma_u}(T(\A_F),\C))\longrightarrow H^u(T(F),C^\emptyset(T(\A_F),\C)),
\end{eqnarray*}
where $u:=\#\Sigma_T$.
\end{definition}

\subsection{Fundamental classes}

By the Dirichlet Unit Theorem the $\Z$-rank of $T(\cO_F)$ is $u$.  Write 
\[
T(\cO_F)_+:=T(\cO_F)\cap T(F_\infty)_+,\qquad T(F_\infty)_+:=\prod_{\sigma\mid\infty}T(F_\sigma)_+.
\] 
Notice that $T(F_{\infty})_0$ is isomorphic to $\R^{u}$ by means of the homomorphism $z\mapsto (\log|\sigma z|)_{\sigma\in\Sigma_T}$. Moreover, under such isomorphism the image of $T(\cO_F)_+$ becomes a $\Z$-lattice $\Lambda$ of $\R^r$. Thus $T(F_{\infty})_0/\Lambda$ is a $u$-dimensional torus and the \emph{fundamental class} $\xi$ is a generator of $H_{u}(T(F_{\infty})_0/\Lambda,\Z)\simeq \Z$. 
Notice that
\[
T(F_\infty)_+= T(F_{\infty})_0\times S,
\] 
where $S$ is isomorphic to a cartesian product of $S^1$. Let $H=T(\cO_F)_+\cap S$. Since $T(\cO_F)_+$ is discrete and $S$ is compact, $H$ is finite. It is easy to check that $H$ is the torsion subgroup of $T(\cO_F)_+$ and $T(\cO_F)_+\simeq\Lambda\times H$ (see \cite[Lemma 3.1]{HerMol1}). In particular,
\begin{equation}\label{eqaster}
T(F_\infty)_+/T(\cO_F)_+\simeq T(F_{\infty})_0/\Lambda\times S/H.
\end{equation}

Since $M:=T(F_{\infty})_0\simeq\R^u$, the de Rham complex $\Omega^\bullet_M$ is a resolution for $\R$. This implies that we have an edge morphism of the induced spectral sequence
\[
e:H^0(\Lambda,\Omega_M^u)\longrightarrow H^u(\Lambda,\R).
\]
We can identify $c\in H_{u}(M/\Lambda,\Z)$ with a group cohomology element $c\in H_{u}(\Lambda,\Z)$ by means of the relation
\[
\int_c \omega=e(\omega)\cap c,\qquad \omega\in H^0(\Lambda,\Omega_M^u)=\Omega_{M/\Lambda}^u.
\]
In particular, we can think of the fundamental class as an element $\xi\in H_u(\Lambda,\Z)$ satisfying
\begin{equation}\label{eqeint}
e(\omega)\cap \xi=\int_{T(F_\infty)_0/\Lambda} \omega,\qquad \omega\in H^0(\Lambda,\Omega_M^u).
\end{equation}

Let us consider the compact subgroup $U:=T(\hat\cO_F)\subset T(\A_F^\infty)$, and write
\[
T(F)_+:=T(F)\cap T(F_\infty)_+,\qquad {\rm Cl}(T)_+:=T(\A_F^\infty)/UT(F)_+.
\]
Clearly, the class group $ {\rm Cl}(T)_+$ is finite and, since $T(\cO_F)_+=T(F)_+\cap U$, we have an exact sequence 
\[
0\longrightarrow T(F)_+/T(\cO_F)_+\longrightarrow T(\A_F^\infty)/U\stackrel{p}{\longrightarrow} {\rm Cl}(T)_+\longrightarrow 0.
\]
Fix preimages in $\bar t_i\in T(\A_F^{\infty})$ of all the elements $t_i\in {\rm Cl}(T)_+$ and consider the compact set
\[
\cF:=\bigcup_i\bar t_i U\subset T(\A_F^\infty).
\]
By construction $T(F_\infty)_+/T(\cO_F)_+\times\cF$ is a fundamental domain of $T(\A_F)$ for the action of $T(F)$. Indeed,
for any $t\in T(\A_F)$, there is a unique $\tau_t\in T(F)/T(\cO_F)_+$ such that $\tau_t^{-1}t\in T(F_{\infty})_+\times\cF$ (see \cite[Lemma 3.2]{HerMol1}). In particular,
\begin{equation}\label{eqaster2}
T(\A_F)=\bigsqcup_{\tau\in T(F)/T(\cO_F)_+}\tau(T(F_\infty)_+\times\cF).
\end{equation}
Notice that the set of continuous functions $C(\cF,\Z)$ has an action of $T(\cO_F)_+$ (since $\cF$ is $U$-invariant) and the characteristic function $1_\cF$ is $T(\cO_F)_+$-invariant. 
Let us consider the canonical class
\[
\eta=1_\cF\cap\xi\in H_u(T(\cO_F)_+,C(\cF,\Z));\qquad 1_\cF\in H^0(T(\cO_F)_+,C(\cF,\Z)),
\]  
where $\xi\in H_r(T(\cO_F)_+,\Z)$ is by abuse of notation the image of $\xi$ through the correstriction morphism.

By \eqref{eqaster2}, the natural $T(\cO_F)_+$-equivariant embedding 
\[
\iota:C(\cF,\Z)\hookrightarrow C_c^\emptyset(T(\A_F),\Z),\qquad \iota \phi(t_\infty,t^\infty) =1_{T(F_\infty)_+}(t_\infty)\cdot\phi(t^\infty)\cdot1_\cF(t^\infty).
\]
provides an isomorphism of $T(F)$-modules
$\Ind_{T(\cO_F)_+}^{T(F)}(C(\cF,\Z))\simeq C_c^\emptyset(T(\A_F),\Z)$,
where $C_c^\emptyset(T(\A_F),\Z)$ is the set of functions in $C^\emptyset(T(\A_F),\C)$ that are $\Z$-valued and compactly supported when restricted to $T(\A_F^\infty)$ (see \cite[Lemma 3.3]{HerMol1}).
Thus, by Shapiro's Lemma one may regard 
\[
\eta\in H_u(T(F),C_c^{\emptyset}(T(\A_F),\Z)).
\]

\subsection{Pairings and fundamental classes}
Notice that, by our choices in \S\ref{haarmeasures}, we have that ($v=\sigma\mid\infty$)
\[
        d^\times x_\sigma=\frac{dx_\sigma}{|x_\sigma|}=\pm \frac{dr}{r},\quad F_\sigma=\R,\;r=|x_\sigma|;\qquad 
        d^\times x_\sigma=\frac{2dsdt}{\pi(s^2+t^2)}=\frac{2drd\theta}{\pi r},\quad F_\sigma=\C,\;x_v=s+it=re^{i\theta}.
\]
This implies that the restriction of $dt_\sigma^\times$ on $T(F_\sigma)_+$ is given by
\begin{equation}\label{defdtx}
    d^\times t_\sigma=\left\{\begin{array}{ll}
         r^{-1}dr,& T(F_\sigma)_+=\R_+,\quad r\in (0,\infty);\\
         \pi^{-1}2d\theta\cdot r^{-1}dr,& T(F_\sigma)_+=\C^\times,\quad r\in (0,\infty),\quad\theta\in [0,2\pi);\\
         \pi^{-1}2d\theta,&T(F_\sigma)_+=\C^\times/\R^\times,\quad\theta\in [0,\pi).
    \end{array}\right.
\end{equation}
Hence, given the decomposition $T(F_\sigma)_+=T(F_\sigma)_0\times S$, we have that $d^\times t_\sigma=d^\times t_{\sigma,0} d^\times s_\sigma$, where $d^\times t_{\sigma,0}=r^{-1}dr$ if $T(F_\sigma)_0\neq 1$. The product of $d^\times t_0:=\prod_{\sigma\in\Sigma_T}d^\times t_{\sigma,0}$ provides a Haar measure on $T(F_\infty)_0$. Similarly, the product $d^\times s:=\prod_{\sigma\mid\infty}d^\times s_{\sigma}$ provides a Haar measure on $S$.

On the other side, the product 
$d^\times t_f:=\prod_{v\nmid\infty}d^\times t_v$ provides a Haar measure on $T(\A_F^\infty)$. As seen in \S \ref{haarmeasures}, we have that 
\begin{equation}\label{volhatO}
    {\rm vol}(\hat\cO_E^\times/\hat\cO_F^\times)=\prod_{v\nmid\infty}{\rm vol}(\cO_{E_v}^\times){\rm vol}(\cO_{F_v}^\times)^{-1}=|d_F D|^{-\frac{1}{2}}
\end{equation}
where $d_{F}$ is the different of $F$. 

Given $\phi\in H^0(T(F),C^{\Sigma_T}(T(\A_F),\C))$, we can consider the complex number
\[
\partial\phi\cap \eta\in \C,
\]
where the cap product corresponds to the $T(F)$-invariant pairing
\[
C^\emptyset(T(\A_F),\C)\times C_c^\emptyset(T(\A_F),\Z)\longrightarrow \C \qquad
\langle f,\phi\rangle=\sum_{z\in T(F_\infty)/T(F_\infty)_+}\int_{S}\int_{T(\A_F^\infty)}f(zs,t_f)\phi(zs,t_f)d^\times t_f d^\times s 
\]

\begin{proposition}\label{partcapeta}
We have that
\[\partial\phi\cap \eta=h\int_{T(\A_F)/T(F)}\phi(t) d^\times t,\]
where $h=\#T(\cO_F)_{+,{\rm tors}}$.
\end{proposition}
\begin{proof}
Recall that $\eta=\xi\cap \iota(1_{\cF})$, hence
\[
    \partial\phi\cap \eta=\sum_{z\in T(F_\infty)/T(F_\infty)_+}\left(\int_{S}\int_{T(\A_F^\infty)}1_{T(F_\infty)_+}(zs)1_{\cF}(t_f)(\partial\phi)(zs,t_f)d^\times t_fd^\times s\right)\cap\xi
    =\left(\int_{S}\int_{\cF}(\partial\phi)(s,t_f)d^\times t_f d^\times s\right)\cap\xi.
\]
As above, write $M=T(F_\infty)_0$ and notice that the function 
\[
M\ni t_0\mapsto\left(\int_{S}\int_{\cF}\phi(t_0s, t_f)d^\times t_fd^\times s\right)
\]
lies in $H^0(\Lambda,C^\infty(M,\C)) $
by the $T(F)$-invariance of $\phi$, and the $T(\cO_F)_+$-invariance of $\cF$. Multiplying by the $\Lambda$-invariant differential $d^\times t_0$ of $T(F_{\infty})_0$ we obtain a differential $\omega\in H^0(\Lambda,\Omega_M^r)$. It is clear by the definition of $\partial$, $d^\times t_0$ and $e$ that
\[
e(\omega)=\left(\int_{S}\int_{\cF}(\partial\phi)(s,t_f)d^\times t_fd^\times s\right)\in H^r(\Lambda,\C).
\]
Thus, using equations \eqref{eqaster} and \eqref{eqeint}, we deduce
\begin{eqnarray*}
\partial\phi\cap \eta&=&e(\omega)\cap\xi=\int_{M/\Lambda}\omega=\int_{M/\Lambda}\int_{S}\int_{\cF}\phi(t_0 s, t_f)d^\times t_f d^\times sd^\times t_0\\
&=&\#H\int_{M/\Lambda}\int_{S/H}\int_{\cF}\phi(t_0s, t_f)d^\times t_f d^\times sd^\times t_0=\#H\int_{T(F_\infty)_+/T(\cO_F)_+}\int_{\cF}\phi(t)d^\times t.
\end{eqnarray*}
Finally, the result follows from the fact that $\#H=\#T(\cO_F)_{+,{\rm tors}}=h$ (since $\Lambda$ is $\Z$-free), and $T(F_\infty)_+/T(\cO_F)_+\times\cF$ is a fundamental domain for $T(\A_F)/T(F)$.
\end{proof}

\section{Cohomology of arithmetic groups and Waldspurger formulas}

For any pair finite dimensional $G(F)$-representations $M$ and $N$ over $\C$ and an open compact subgroup $U\subset G(\A_F^\infty)$, we define $\cA^{\infty}(M,N)^U$ to be the set of functions
\[
\phi: G(\A_F^{\infty})/U\longrightarrow\Hom_{\C}(M,N).
\]
We write $\cA^{\infty}(N)^U:=\cA^{\infty}(\C,N)^U$. Notice that $\cA^{\infty}(M,N)^U$ has natural action of $G(F)$: 
\[
(\gamma\phi)(g)=\gamma\phi(\gamma^{-1}g);\qquad \phi\in \cA^{\infty}(M,N)^U,\quad \gamma\in G(F).
\]
We denote by $\cA^{\infty}(M,N)^U(\lambda)$ the twist of $\cA^{\infty}(M,N)^U$ by a character $\lambda:G(F)\rightarrow \C^\times$. 

\subsection{Lower Eichler-Shimura for central automorphic forms}\label{ES}

In this section we will study Harder-Eichler-Shimura morphisms for automorphic forms of even weight $\underline{k}=(k_\sigma)_{\sigma\mid\infty}\in (2\N)^{\Sigma_F}$. 

Given $U\subset G(\A_F^\infty)$ open compact subgroup, write $\cA(\C)^U$ for the set of functions $f:G(\A_F)/U\rightarrow\C$,
that are:
\begin{itemize}
\item $\cC^\infty$ when restricted to $G(F_\infty)$.

\item Right-$K_\sigma$-finite, for all $\sigma\in\Sigma_F$, where $K_\sigma$ is the maximal compact subgroup of $G(F_\sigma)$.

\item Right-$\mathcal{Z}_\sigma$-finite, for all $\sigma\in\Sigma_F$, where $\mathcal{Z}_\sigma$ is the centre of the universal enveloping algebra of $G(F_\sigma)$.
\end{itemize}

Notice that automorphic forms are left-$G(F)$-invariant elements of $\cA(\C)^U$, for some $U\subset G(\A_F^\infty)$.

For any $\sigma\in\Sigma_B$, let $\cG_\sigma$ be the Lie algebra of $G(F_\sigma)$ and $K_\sigma$ a maximal compact subgroup.
By a $(\cG_\infty,K_\infty)$-module we mean the tensor product of $(\cG_\sigma,K_\sigma)$-modules, for $\sigma\in\Sigma_B$, and finite-dimensional representations of $G(F_\sigma)$, for $\sigma\in\Sigma_F\setminus\Sigma_B$. Given a $(\cG_\infty,K_\infty)$-module $\cV$, we define
\[
\cA^{\infty}(\cV,\C)^U:=\Hom_{(\cG_\infty,K_\infty)}\left(\cV,\cA(\C)^U\right).
\]
This is consistent with the notation of the beginning of the section since, by \cite[Lemma 2.3]{guitart2017automorphic}, for $V$ a finite dimensional $G(F_\infty)$-representation,
$\cA^{\infty}(V,\C)^U=\cA^{\infty}(\cV,\C)^U$,
where $\cV$ is the $(\cG_\infty,K_\infty)$-module associated with $V$. 

Starting from this section we will assume that the ramification sets of $T$ and $G$ coincide, namely, $\Sigma_T=\Sigma_B$ and $u=r$. 
For any $\sigma\in \Sigma_B$, we fix isomorphisms $E_\sigma\simeq F_\sigma^2$. Recall that the embedding $B\hookrightarrow \M_2(E)$ of \eqref{embEinB} maps any $e\in E$ to the matrix
$\big(\begin{smallmatrix}
    e&\\&\bar e
\end{smallmatrix}\big)$. Thus, the above isomorphisms induce 
identifications $G(F_\sigma)\simeq\PGL_2(F_\sigma)$ sending $T(F_\sigma)$ to the diagonal torus. 
Due to this, for any $\underline{k}=(k_\sigma)_\sigma\in\Z^{\Sigma_F}$ with $k_\sigma\geq 2$ even, we can consider the $(\cG_\infty,K_\infty)$-module 
\[
D(\underline{k})=\bigotimes_{\sigma\in\Sigma_B}D(k_\sigma)\otimes\bigotimes_{\sigma\in\Sigma_F\setminus\Sigma_B}V(k_\sigma-2), \qquad \underline{k}=(k_\sigma),
\]
where $V(k_\sigma-2)$ and $D(k_\sigma)$ are the $(\cG_\sigma,K_\sigma)$-modules and finite dimensional representations described in \S \ref{findimreps}, \S \ref{GKR} and \S \ref{GKC}. 
\begin{definition}\label{defAFweightk}
    A modular form of weight $\underline{k}$ and level $U$ for $G$ is an element of
\[
M_{\underline k}(U):= H^0(G(F),\cA^{\infty}(D(\underline{k}),\C)^U)=\Hom_{(\cG_\infty,K_\infty)}\left(D(\underline{k}),\mathfrak{A}(G)^U\right),
\] 
where $\mathfrak{A}(G)^U=H^0(G(F),\cA(\C)^U)$ is the space of $U$-invariant automorphic forms of $G$.

\end{definition}

Fix a character $\lambda:G(F)/G(F)_+=G(F_\infty)/G(F_\infty)_+\rightarrow\pm1$.
For any $\sigma\in\Sigma_B$,
let us consider the exact sequence of $(\cG_\sigma,K_\sigma)$-modules described in \eqref{exseq1} and \eqref{exseq2}
\begin{equation}\label{ESDBV}
    0\longrightarrow D(k_\sigma)\stackrel{\imath_{\lambda_\sigma}}{\longrightarrow}\tilde\cB_\sigma^{\lambda_\sigma} \stackrel{\rho_{\lambda_\sigma}}{\longrightarrow}V(k_\sigma-2)(\lambda_\sigma) \longrightarrow 0.
\end{equation}
 For any $\Sigma\subset\Sigma_B$ and $\tau\in\Sigma$, we can also consider the intermediate $(\cG_\infty,K_\infty)$-modules
\begin{eqnarray*}
D_\Sigma^\lambda(\underline{k})&:=&\bigotimes_{\sigma\in\Sigma}D(k_\sigma)\otimes\bigotimes_{\sigma\in\Sigma_F\setminus\Sigma}V(k_\sigma-2)(\lambda_\sigma), \\ 
\tilde\cB^\lambda_{\Sigma,\tau}&:=&\bigotimes_{\sigma\in\Sigma\setminus\{\tau\}}D(k_\sigma)\otimes \tilde\cB_\tau^{\lambda_\tau}\otimes\bigotimes_{\sigma\in\Sigma_F\setminus\Sigma}V(k_\sigma-2)(\lambda_\sigma),
\end{eqnarray*}
lying in the exact sequence 
\[
0\longrightarrow D^\lambda_\Sigma(\underline{k})\longrightarrow\tilde\cB_{\Sigma,\tau}^\lambda \longrightarrow D^\lambda_{\Sigma\setminus\{\tau\}}(\underline{k}) \longrightarrow 0.
\]
Thus, we obtain the associated exact sequence of $G(F)$-modules
\begin{equation}\label{exseqglob}
0\longrightarrow \cA^{\infty}(D^\lambda_{\Sigma\setminus\{\tau\}}(\underline{k}),\C)^U\longrightarrow\cA^{\infty}(\tilde\cB^\lambda_{\Sigma,\tau},\C)^U \longrightarrow \cA^{\infty}(D^\lambda_\Sigma(\underline{k}),\C )^U\longrightarrow 0.
\end{equation}
Since $D(\underline{k})=D^\lambda_{\Sigma_B}(\underline{k})$ and $\Hom(D^\lambda_{\emptyset}(\underline{k}),\C)=V(\underline{k}-2)(\lambda):=\bigotimes_{\sigma\in\Sigma_F}V(k_\sigma-2)(\lambda_\sigma)$, the connection morphisms in the corresponding long exact sequences provide the \emph{Harder-Eichler-Shimura morphism} (see \cite{ESsanti})
\begin{eqnarray*}
&{\rm ES}_\lambda:&M_{\underline k}(U)=H^0(G(F),\cA^{\infty}(D(\underline{k}),\C)^U)\longrightarrow H^1(G(F),\cA^{\infty}(D^\lambda_{\Sigma_B\setminus\{\sigma_1\}}(\underline{k}),\C)^U)\longrightarrow \cdots\\
&\cdots&\longrightarrow H^{r-1}(G(F),\cA^{\infty}(D^\lambda_{\{\sigma_r\}}(\underline{k}),\C)^U)\longrightarrow H^{r}(G(F),\cA^{\infty}(V(\underline{k}-2))^U(\lambda)),
\end{eqnarray*}
where $r=\#\Sigma_B$. By means of ${\rm ES}_\lambda $ we realize automorphic representations of $G$ in the cohomology group 
\[
H_\ast^{r}(G(F),\cA^{\infty}(V(\underline{k}-2))(\lambda)) =\varinjlim_U H^{r}(G(F),\cA^{\infty}(V(\underline{k}-2))^U(\lambda)).
\]

\begin{remark}\label{difES}
    As explained in remark \ref{remarkonESconmor}, the exact sequence of \eqref{exseqglob} differs by a factor of $\pi$ from the one explained in \cite{ESsanti}. Hence, each connection morphism there is given by $\pi^{-1}$ times the connection morphism here. This implies that, if $F$ is totally real, the Harder-Eichler-Shimura morphism explained in \cite{ESsanti} is $\pi^{-r}{\rm ES}_\lambda$.
\end{remark}

\subsection{Classical Waldspurger formulas}\label{classWalds}

In this section we will recall the classical Waldspurger formula. We will essentially follow the reference \cite{CST}, but we will use slightly different test vectors at archimedean places.
Let $\pi$ be an automorphic representation of $G$ and let $\Pi$ be its Jacquet-Langlands lift to $\PGL_2/F$. Given the torus $T\subset G$ considered previously and a character $\chi$ of $T(\A_F)/T(F)$, we define the period integral
\begin{equation}\label{defellWalds}
    \ell(\cdot,\chi):\pi\longrightarrow \C,\qquad \ell(f,\chi)=\int_{T(\A_F)/T(F)}f(t)\cdot\chi(t)d^\times t.
\end{equation}
The aim of this section is to compute these period integrals. The first result that one can think of is the classical Waldspurger formula (see \cite{YZZ}) that asserts that, given decomposable $f_1=\bigotimes_v'f_{1,v},f_2=\bigotimes_v'f_{2,v}\in\pi$,
\begin{equation}\label{classicalWF}
    \frac{\ell(f_1,\chi)\cdot \ell(f_2,\chi^{-1})}{\langle f_1,f_2\rangle}=\frac{\Lambda_F(2)\cdot\Lambda(1/2,\Pi,\chi)}{2\Lambda(1,\Pi,{\rm ad})}\cdot\prod_v\beta_v(f_{1,v},f_{2,v}),
\end{equation}
where $\Lambda$ stands for complete global L-functions, that is, those
as products of local L-functions over all places, 
\begin{equation}\label{defalphaWalds}
     \beta_v(f_{1,v},f_{2,v}):=\frac{L(1,\eta_{T,v})\cdot L(1,\Pi_v,{\rm ad})}{\zeta_{v}(2)\cdot L(1/2,\Pi_v,\chi_v)}\int_{T(F_v)}\chi_v(t_v)\frac{\langle\pi_v(t_v)f_{1,v},f_{2,v}\rangle_v}{\langle f_{1,v},f_{2,v}\rangle_v} d^\times t_v
\end{equation}
and $\langle\;,\;\rangle=\prod_v\langle\;,\;\rangle_v$ is the inner product
\begin{equation}\label{defpetersson}
    \langle f_1,f_2\rangle=\int_{G(F)\backslash G(\A_F)}f_1(g)f_2(g)d^\times g.
\end{equation}

In order to get rid of the factor $2\Lambda(1,\Pi,{\rm ad})\Lambda_F(2)^{-1}$, one can use the Petersson pairing formula (see \cite[proposition 2.1]{CST}). It states that, given decomposable $\mathfrak{f}_1,\mathfrak{f}_2\in \Pi$,
\begin{equation}\label{Petform}
    \langle \mathfrak{f}_1,\mathfrak{f}_2\rangle=2\Lambda(1,\Pi,{\rm ad})\cdot \Lambda_F(2)^{-1}\cdot\prod_v\alpha_v(W_{\mathfrak{f}_1,v},W^-_{\mathfrak{f}_2,v}),
\end{equation}
where the product $\langle\;,\;\rangle$ is defined as in \eqref{defpetersson}, the elements of the Whittaker model
\begin{eqnarray*}
    \prod_vW_{\mathfrak{f}_i,v}&=&W_{\mathfrak{f}_i}=\int_{\A_F/F}\mathfrak{f}_i\left(\left(\begin{array}{cc}1&x\\&1\end{array}\right)g\right)\psi(-x)dx;\qquad \prod_vW^-_{\mathfrak{f}_i,v}=W^-_{\mathfrak{f}_i}=\int_{\A_F/F}\mathfrak{f}_i\left(\left(\begin{array}{cc}1&x\\&1\end{array}\right)g\right)\psi(x)dx;\\
\end{eqnarray*}
being $\psi$ the additive character defined in \eqref{defpsiad}, and 
\begin{equation}\label{calcWhitt}
    \alpha_v(W_{\mathfrak{f}_1,v},W^-_{\mathfrak{f}_2,v})=\frac{\zeta_v(2)\cdot\langle W_{\mathfrak{f}_1,v},W^-_{\mathfrak{f}_2,v} \rangle_v}{\zeta_v(1)\cdot L(1,\Pi_v,{\rm ad})};\qquad \langle W_{\mathfrak{f}_1,v},W^-_{\mathfrak{f}_2,v} \rangle_v=\int_{F_v^\times}W_{\mathfrak{f}_2,v}\left(\begin{array}{cc}a&\\&1\end{array}\right)W^-_{\mathfrak{f}_2,v}\left(\begin{array}{cc}a&\\&1\end{array}\right)d^\times a.
\end{equation}
Combining \eqref{classicalWF} and \eqref{Petform}, one obtains that for decomposable $f_1,f_2\in\pi$ and $\mathfrak{f}_1,\mathfrak{f}_2\in\Pi$,
\begin{equation}\label{WaldsconPet}
    \ell(f_1,\chi)\cdot \ell(f_2,\chi^{-1})=\Lambda(1/2,\Pi,\chi)\cdot\frac{\langle f_1,f_2\rangle}{\langle \mathfrak{f}_1,\mathfrak{f}_2\rangle}\cdot\prod_v\beta_v(f_{1,v},f_{2,v})\alpha_v(W_{\mathfrak{f}_1,v},W^-_{\mathfrak{f}_2,v}).
\end{equation}

Recall that, instead of working directly with automorphic forms, in \S \ref{ES} we have considered the space 
$M_{\underline k}(U)= H^0(G(F),\cA^{\infty}(D(\underline{k}),\C)^U)$, for some open compact subgroup $U\subset G(\A_{F}^\infty)$. Notice that, for any $V\in D(\underline{k})$ and $\Phi\in M_{\underline k}(U)$, we can construct the automorphic form $\Phi(V)$. Hence, for any automorphic representation such that $\pi_\infty=\pi\mid_{G(F_\infty)}\simeq D(\underline k)$, we can realize $\pi^\infty=\pi\mid G(\A_F^\infty)$ in the space
\[
\varinjlim_U M_{\underline k}(U).
\]
In the remainder of the section we will show a formula analogous to \eqref{WaldsconPet} but onvolving vectors in the above realization of $\pi^\infty$.

On the one hand, for a given character $\lambda:G(F_\infty)/G(F_\infty)_+\rightarrow\pm1$ and for any $\sigma\in\Sigma_B$, the morphism $\rho_{\lambda_\sigma}$ of \eqref{ESDBV} admits a $K_\sigma$-equivariant section (see \S\ref{splitGK} and lemma \ref{lemsect2})
\begin{eqnarray*}
s_{\lambda_\sigma}:V(k_\sigma-2)(\lambda_\sigma)&\longrightarrow& \tilde \cB_\sigma^{\lambda_\sigma}.
\end{eqnarray*}
On the other hand, we have the exponential map 
\begin{equation}\label{expmap}
    \exp_\sigma:F_\sigma\stackrel{\simeq}{\longrightarrow} {\rm Lie}(F_\sigma^\times)={\rm Lie}(T(F_\sigma))\subset\cG_\sigma.
\end{equation}
The element of the Lie algebra $\delta^T_\sigma:=\exp_\sigma(1)\in \cG_\sigma$ gives rise to a morphism 
\begin{equation}\label{deftildes}
 \delta s_{\lambda_\sigma}:V(k_\sigma-2)(\lambda_\sigma)\longrightarrow D(k_\sigma),\qquad \delta s_{\lambda_\sigma}(\mu)=\imath_{\lambda_\sigma}^{-1}(\delta^T_\sigma(s_{\lambda_\sigma}(\mu))-s_{\lambda_\sigma}(\delta^T_\sigma\mu)).
\end{equation}
The tensor product of such a $\delta s_{\lambda_\sigma}$ provides
\begin{equation}\label{primerdeltas}
    \underline{\delta s}_{\lambda}:V(\underline{k}-2)(\lambda)\longrightarrow D(\underline{k}).
\end{equation}
Assume now that $\chi$ is locally polynomial of degree less that $\frac{\underline k-2}{2}$, namely, $\chi(t)=\chi_0(t)\cdot t^{\underline m}$ with $t\in T(F_\infty)$, for some $\frac{2-\underline k}{2}\leq {\underline m}\leq \frac{\underline k-2}{2}$ and some locally constant character $\chi_0$. By means of the morphism \eqref{PtoCoverC}, the character $t\mapsto t^{\underline m}$ corresponds to $\mu_{\underline m}=\bigotimes_\sigma\mu_{m_\sigma}\in V(\underline{k}-2)$
\begin{equation}\label{defmumglobal}
    \mu_{m_\sigma}\left(\left|\begin{array}{cc}X& Y\\  x&y \end{array}\right|^{k_\sigma-2}\right)=x^{\frac{k_\sigma-2}{2}-m_\sigma}y^{\frac{k_\sigma-2}{2}+m_\sigma},\quad\mbox{or simply}\quad \mu_{\underline m}\left(\left|\begin{array}{cc}X& Y\\  x&y \end{array}\right|^{\underline k-2}\right)=x^{\frac{\underline k-2}{2}-\underline m}y^{\frac{\underline k-2}{2}+\underline m}.
\end{equation}
Hence, given a pair of forms $\Phi_1,\Phi_2\in M_{\underline k}(U)$, the formula \eqref{WaldsconPet} allows us to compute the product of period integrals $\ell(\Phi_1(\underline{\delta s}_\lambda(\mu_{\underline m})),\chi)\cdot \ell(\Phi_2(\underline{\delta s}_\lambda(\mu_{-\underline m}),\chi^{-1})$. 
Nevertheless, applying directly \eqref{WaldsconPet} one obtains a formula that involves the pairing $\langle \Phi_1(\underline{\delta s}_\lambda(\mu_{\underline m})),\Phi_2(\underline{\delta s}_\lambda(\mu_{-\underline m}))\rangle$ depending on $\underline m$. In order to introduce a product only depending on $\Phi_1$ and $\Phi_2$, we consider the canonical $G(F_\infty)$-invariant element
\[
\bigotimes_{\sigma\mid\infty}\Upsilon_\sigma=\Upsilon=\left|\begin{array}{cc}
    x_1 & y_1 \\
    x_2 & y_2
\end{array}\right|^{\underline k-2}\in \cP(\underline{k}-2)\otimes \cP(\underline{k}-2)\simeq V(\underline{k}-2)\otimes V(\underline{k}-2).
\]
Thus, once interpreting $\Phi_1\Phi_2$ as an element of $H^0(G(F),\cA^{\infty}(D(\underline{k})\otimes D(\underline{k}),\C)^U)$, we define
\begin{equation}\label{eqpeterssonprod}
    \langle\Phi_1,\Phi_2\rangle:=\int_{G(F)\backslash G(\A_F)}\Phi_1\Phi_2\left(\underline{\delta s}_1(\Upsilon)\right)(g,g)d^\times g.
\end{equation}
by lemmas \ref{lemoonpairupsilon1} ad \ref{lemoonpairupsilon}, this provides a non-zero $G(\A^\infty_F)$-equivariant pairing.
\begin{remark}\label{classesASHilbertform}
    Assume that $G=\PGL_2$ and $F$ is totally real. Given $U\subset\PGL_2(\A_F^\infty)$ an open compact subgroup, we can identify $M_{\underline k}(U)$ with the space of Hilbert modular forms of weight $\underline k$ and level $U$ by means of the map
    \begin{equation}\label{identCMF}
        \Phi\longmapsto f_\Phi(z,g_j):=\Phi\left(\bigotimes_{\sigma}f_{\frac{k_\sigma}{2}}\right)((g_\sigma)_\sigma,g_j)\prod_\sigma \frac{k_\sigma-1}{2i^{\frac{k_\sigma-2}{2}}f_{\frac{k_\sigma}{2}}(g_\sigma)};\qquad z=(g_\sigma i)_\sigma\in\mathfrak{H}^d,
    \end{equation}
    where $\mathfrak{H}$ is the upper-half plane, the elements $g_j$ are representatives of the finite group ${\rm Pic}(U)=G(F)_+\backslash G(\A_F)\slash U$, and $f_{\frac{k_\sigma}{2}}\in D(k_\sigma)$ are the vectors defined in equation \eqref{fn}.
    
    Notice that we also have the classical Petersson inner pairing on Hilbert modular forms:
    \[
    (f,f)_{U}=\sum_{g_j\in {\rm Pic}(U)}\iint_{\Gamma_j\backslash\mathfrak{H}^d}|f(z,g_j)|^2y^{\underline k-2}dxdy;\qquad d=[F:\Q],\quad z=x+iy\in \mathfrak{H}^d,\quad \Gamma_j=g_jUg_j^{-1}\cap G(F)_+.
    \]
    Thus, there must be a relation between both pairings.
    Indeed, if we consider $\bar\Phi\in M_{\underline k}(U)$ given by $\bar\Phi(f)=\overline{\Phi(\bar f)}$,
    since the chosen Haar measure at any infinite place $\sigma$ is $dg_\sigma^\times=\frac{dx_\sigma dy_\sigma d\theta_\sigma}{y_\sigma^2}$ for $g_\sigma=\big(\begin{smallmatrix}y_\sigma^{1/2}&x_\sigma y_\sigma^{-1/2}\\&y_\sigma^{-1/2}\end{smallmatrix}\big)\big(\begin{smallmatrix}\cos(\theta_\sigma)&\sin(\theta_\sigma)\\-\sin(\theta_\sigma)&\cos(\theta_\sigma)\end{smallmatrix}\big)$, we obtain by lemmas \ref{lemoonpairupsilon1}, \ref{pairv021} and proposition \ref{keylemma1}:
    \begin{eqnarray*}
        \langle\Phi,\bar\Phi\rangle
        &=&2^{{\underline k}-2}{\rm vol}(U)\sum_{g_j\in {\rm Pic}(U)}\int_{\Gamma_j\backslash G(F_\infty)_+}\Phi\bar\Phi\left(\bigotimes_\sigma\left(\frac{(k_{\sigma}-1)}{2}\left((-i)^{\frac{k_{\sigma}-2}{2}}f_{\frac{k_{\sigma}}{2}}+i^{\frac{k_{\sigma}-2}{2}}f_{-\frac{k_{\sigma}}{2}}\right)\right)\right)(g_\infty,g_j)d^\times g_\infty\\
        &=&2^{{\underline k}-2}{\rm vol}(U)\prod_{\sigma}\frac{(k_{\sigma}-1)^2}{2}\sum_{g_j\in {\rm Pic}(U)}\int_{\Gamma_i\backslash G(F_\infty)_+}\Phi\left(\bigotimes_\sigma f_{\frac{k_{\sigma}}{2}}\right)\overline{\Phi\left(\bigotimes_\sigma f_{\frac{k_{\sigma}}{2}}\right)}(g_\infty,g_j)d^\times g_\infty\\
        &=&2^{{\underline k}}{\rm vol}(U)\left(\frac{\pi}{2}\right)^d\sum_{g_j\in {\rm Pic}(U)}\iint_{\Gamma_j\backslash \mathfrak{H}^{d}}|f_\Phi(z,g_j)|^2y^{\underline k-2}dxdy=2^{{\underline k}}{\rm vol}(U)\left(\frac{\pi}{2}\right)^d(f_\Phi,f_\Phi)_U.
    \end{eqnarray*}
\end{remark}

Finally, regarding the automorphic representation $\Pi$ of $\PGL_2$, we would like to specify a normalized newform $\mathfrak{f}$ in order to provide a more explicit formula. Write $N\subseteq \cO_F$ for the conductor of $\Pi$, and write 
\[
U_0(N)=\left\{\big(\begin{smallmatrix} a&b\\c&d\end{smallmatrix}\big)\in \GL_2(\hat\cO_F),\;c\in N\hat\cO_F\right\}.
\]
Let $\mathfrak{f}=\bigotimes_v\mathfrak{f}_{v}\in \mathfrak{A}(\PGL_2)^{U_0(N)}$ be the $U_0(N)$-invariant automorphic form of $\Pi$ normalized so that $\mathfrak{f}_{\sigma}$ is fixed by the diagonal matrices $\big(\begin{smallmatrix} a&\\&1\end{smallmatrix}\big)\in \GL_2(F_\sigma)$, $|a|_\sigma=1$, with weight minimal for all $\sigma\in\Sigma_F$, and such that
\[
\Lambda(s,\Pi)=|d_F|^{s-1/2}\int_{\A_F^\times/F^\times}\mathfrak{f}\left(\begin{array}{cc}a&\\&1\end{array}\right)|a|^{s-1/2}d^\times a,
\] 
where $|\cdot|:\A_F^\times\rightarrow\R_+$ is the standard adelic absolute value. In \cite{CST}, the values $\langle W_{\mathfrak{f},\sigma},W_{\mathfrak{f},\sigma}^- \rangle_\sigma$ are computed explicitly for such a form:
\begin{lemma}\label{lemonWs}
We have that
    \[
\langle W_{\mathfrak{f},\sigma},W^-_{\mathfrak{f},\sigma} \rangle_\sigma=\left\{\begin{array}{ll}
    2(4\pi)^{-k_\sigma}\Gamma(k_\sigma); & F_\sigma=\R; \\
    8(2\pi)^{1-k_{\sigma_1}-k_{\sigma_2}}\Gamma\left(\frac{k_{\sigma_1}+k_{\sigma_2}}{2}\right)^2\frac{\Gamma(k_{\sigma_1})\Gamma(k_{\sigma_2})}{\Gamma(k_{\sigma_1}+k_{\sigma_2})}; & F_\sigma=\C,\;\sigma_1,\sigma_2\mid\sigma.
\end{array}\right.
\]
\end{lemma}
\begin{proof}
    If $F_\sigma=\C$, we know by \cite[Theorem 6.2]{JL} that our representation $D(k_\sigma)$ is isomorphic to the principal series $\tilde\cB(\mu_1|\cdot |^{\frac{1}{2}})=\pi(\mu_1,\mu_2)$, where 
 \[
 \mu_1(t)=\mu_2(t)^{-1}=t^{\frac{k_{\sigma_1}-1}{2}}\bar t^{\frac{1-k_{\sigma_2}}{2}}=\ipa{t/|t|^{\frac{1}{2}}}^{\frac{k_{\sigma_1}+k_{\sigma_2}-2}{2}}|t|^{\frac{k_{\sigma_1}-k_{\sigma_2}}{4}};\qquad \sigma_1,\sigma_2\mid\sigma.
 \]
 In \cite[before lemma 3.14]{CST}, one can find a receipt to compute $\langle W_{\mathfrak{f},\sigma},W^-_{\mathfrak{f},\sigma} \rangle_\sigma$ in case $(F_\sigma,\Pi_\sigma)=(\R,D(k_\sigma))$ and $(F_\sigma,\Pi_\sigma)=(\C,\pi(\mu_1,\mu_2))$, hence, the result follows from the given computation. 
\end{proof}

Similarly as above, it is more consistent to consider a normalized element of $ M_{\underline k}(U_0(N))$ instead of the automorphic form $\mathfrak{f}$. 
Notice that $\underline {\delta s}_1(\mu_{\underline 0})\in D(\underline k)$ is fixed by the diagonal matrices $\big(\begin{smallmatrix} a&\\&1\end{smallmatrix}\big)\in \GL_2(F_\sigma)$, $|a|_\sigma=1$, with weight minimal. Hence, we can define the normalized element
\[
\Psi\in  M_{\underline k}(U_0(N))=H^0(\PGL_2(F),\cA^{\infty}(D(\underline{k}),\C)^{U_0(N)});\qquad \Psi(\underline {\delta s}_1(\mu_{\underline 0}))=\mathfrak{f},
\]
and we can consider the corresponding pairing $\langle\Psi,\Psi\rangle$ as in \eqref{eqpeterssonprod}. Notice that, in case $F$ totally real, $\Psi$ corresponds under the identification \eqref{identCMF} of remark \ref{classesASHilbertform} to the normalized Hilbert newform.
\begin{proposition}\label{WaldsforPhi2}
    Given decomposable $\Phi_1,\Phi_2\in \pi^\infty\subset \varinjlim_U M_{\underline k}(U)$ and a locally polynomial character $\chi$ such that $\chi\mid_{T(F_\infty)}=(\cdot)^{\underline m}\chi_0$, for some $\frac{2-\underline k}{2}\leq {\underline m}\leq \frac{\underline k-2}{2}$ and some locally constant character $\chi_0$, we have $\ell(\Phi_i(\underline{\delta s}_\lambda(\mu_{\underline m})),\chi)=0$ unless $\chi_0\mid_{T(F_\infty)}=\lambda$ and, if this is the case,
    \[
    \ell(\Phi_1(\underline{\delta s}_\lambda(\mu_{\underline m})),\chi)\cdot \ell(\Phi_2(\underline{\delta s}_\lambda(\mu_{-\underline m}),\chi^{-1})=C(\underline k,\underline m)\cdot L(1/2,\Pi,\chi)\cdot\frac{\langle \Phi_1,\Phi_2\rangle}{\langle \Psi,\Psi\rangle}\cdot\prod_{v\nmid\infty}\left(\beta_v(\Phi_{1,v},\Phi_{2,v})\cdot\alpha_{v}(W_{\mathfrak{f},v},W^-_{\mathfrak{f},v})\right),
    \]
    where $L(s,\Pi,\chi)$ is the L-function with the archimedean factors removed, $\langle\;,\;\rangle$ are the inner products of \eqref{eqpeterssonprod} and
    \[
    C(\underline k,\underline m)=(-1)^{\left(\sum_{\sigma\not\in\Sigma_B}\frac{k_{\sigma}-2}{2}\right)}4^{\#\Sigma_T^\R}\cdot(32\pi)^{\#\Sigma_T^\C}\left(\frac{1}{\pi}\right)^{\#(\infty\setminus\Sigma_B)}\prod_{\sigma\mid\infty}\prod_{\tilde\sigma\mid\sigma}\frac{\Gamma(\frac{k_{\tilde\sigma}}{2}-m_{\tilde\sigma})\Gamma(\frac{k_{\tilde\sigma}}{2}+m_{\tilde\sigma})}{(-1)^{m_{\tilde\sigma}}(2\pi)^{k_{\tilde\sigma}}}.
    \]
\end{proposition}
\begin{proof}
For the proof of this proposition we will use the computations performed in \S \ref{loctheory}.
For all $\sigma\mid\infty$, we write $\beta_\sigma=\beta_\sigma(\Phi_1(\underline{\delta s}_\lambda(\mu_{\underline m}))_\sigma,\Phi_2(\underline{\delta s}_\lambda(\mu_{-\underline m}))_\sigma)$ and $\alpha_{\sigma}=\alpha_{\sigma}(W_{\mathfrak{f},\sigma},W^-_{\mathfrak{f},\sigma})$. Since $\Phi_i(\underline{\delta s}_\lambda(\mu_{\underline m}))_\sigma$ is proportional to $\delta s_{\lambda_\sigma}(\mu_{m_\sigma})$ if $\sigma\in\Sigma_B$ and to $\mu_{m_\sigma}$ otherwise, we have that
\[
\frac{\langle\Psi,\Psi\rangle}{\langle\Phi_1,\Phi_2\rangle}\frac{\langle\Phi_1(\underline{\delta s}_\lambda(\mu_{\underline m})),\Phi_2(\underline{\delta s}_\lambda(\mu_{-\underline m}))\rangle}{\langle\mathfrak{f},\mathfrak{f}\rangle}\prod_{\sigma\mid\infty}\beta_\sigma\alpha_{\sigma}=\prod_{\sigma\mid\infty}\frac{C_\sigma}{L(1/2,\Pi_\sigma,\chi_\sigma)},
\]
where
\[
C_\sigma=\left\{\begin{array}{ll}
    \langle\delta s_1\Upsilon_\sigma\rangle_\sigma\langle \Upsilon_\sigma\rangle_\sigma^{-1}\langle W_{\mathfrak{f},\sigma},W^-_{\mathfrak{f},\sigma} \rangle_\sigma L(1,\eta_{T,\sigma}) \zeta_\sigma(1)^{-1}\langle \delta s_1\mu_{0},\delta s_1\mu_{0}\rangle_\sigma^{-1}\int_{T(F_\sigma)}\chi_\sigma(t)\langle t\mu_{m_\sigma},\mu_{-m_\sigma}\rangle_\sigma d^\times t & \sigma\not\in\Sigma_B; \\
    \langle W_{\mathfrak{f},\sigma},W^-_{\mathfrak{f},\sigma} \rangle_\sigma L(1,\eta_{T,\sigma})\zeta_\sigma(1)^{-1}\langle \delta s_1\mu_{0},\delta s_1\mu_{0}\rangle_\sigma^{-1}\int_{T(F_\sigma)}\chi_\sigma(t)\langle t\delta s_{\lambda_\sigma}(\mu_{m_\sigma}),\delta s_{\lambda_\sigma}(\mu_{-m_\sigma})\rangle_\sigma d^\times t, &\sigma\in\Sigma_B. 
\end{array}\right.
\]
Notice that, in the above expression, we have freedom to choose the pairings $\langle\;,\;\rangle_\sigma$ except of the pairing \eqref{calcWhitt} in the Whittaker model. Indeed, all such $G(F_\sigma)$-invariant pairings are proportional. 

If $T(F_\sigma)=\C^\times/\R^\times$ then $L(1,\eta_{T,\sigma})\zeta_\sigma(1)^{-1}=\pi^{-1}$. Moreover, we have $\chi_\sigma=t^{m_\sigma}$, $\pi_\sigma(t)\mu_{m_\sigma}=t^{-m_\sigma}\mu_{m_\sigma}$ and ${\rm vol}(T(F_\sigma))=2$ with respect to the measure \eqref{defdtx}. Hence, by lemmas \ref{lemonWs}, \ref{lemoonpairupsilon1}, \ref{pairv021} and equation \eqref{relmummum},
\[
C_\sigma=2\langle\delta s_1\Upsilon_\sigma\rangle_\sigma\langle \Upsilon_\sigma\rangle_\sigma^{-1}\langle W_{\mathfrak{f},\sigma},W^-_{\mathfrak{f},\sigma} \rangle_\sigma \pi^{-1}\langle \delta s_1\mu_{0},\delta s_1\mu_{0}\rangle_\sigma^{-1}\langle\mu_{m_\sigma},\mu_{-m_\sigma}\rangle_\sigma=\frac{2\cdot\Gamma(\frac{k_\sigma}{2}-m_\sigma)\Gamma(\frac{k_\sigma}{2}+m_\sigma)}{(-1)^{m_\sigma+\frac{k_\sigma-2}{2}}(2\pi)^{k_\sigma+1}}.
\]

If $T(F_\sigma)=\R^\times$, then $L(1,\eta_{T,\sigma})\zeta_\sigma(1)^{-1}=1$. Hence, we obtain by theorem \ref{calcLI1}, and lemmas \ref{lemonWs} and \ref{pairv021},
\[
C_\sigma=\left\{\begin{array}{ll}4(-1)^{m_\sigma}(2\pi)^{-k_\sigma}\Gamma(\frac{k_\sigma}{2}-m_\sigma)\Gamma(\frac{k_\sigma}{2}+m_\sigma),&\chi_{0,\sigma}=\lambda_\sigma;\\
    0,&\chi_{0,\sigma}\neq\lambda_\sigma.
    \end{array}\right.
\]


Finally, if $T(F_\sigma)=\C^\times$ then we have as well $L(1,\eta_{T,\sigma})\zeta_\sigma(1)^{-1}=1$. Moreover, by theorem \ref{calcLI2} and lemmas \ref{pairv02} and \ref{lemonWs},
\[
C_\sigma=\left\{\begin{array}{ll}(-1)^{m_{\sigma_1}+m_{\sigma_2}}16(2\pi)^{1-k_{\sigma_1}-k_{\sigma_2}}\Gamma(\frac{k_{\sigma_1}}{2}-m_{\sigma_1})\Gamma(\frac{k_{\sigma_1}}{2}+m_{\sigma_1})\Gamma(\frac{k_{\sigma_2}}{2}-m_{\sigma_2})\Gamma(\frac{k_{\sigma_2}}{2}+m_{\sigma_2}),&\chi_{0,\sigma}=1;\\
    0,&\chi_{0,\sigma}\neq1.
    \end{array}\right.
\]
Hence, the result follows from \eqref{WaldsconPet} since $C(\underline k,\underline m)=\prod_{\sigma\mid\infty}C_\sigma$.
\end{proof}

In \cite{CST}, Cai, Shu and Tian obtain a more explicit formula by specifying concrete test vectors.
Let us consider $c\subseteq\cO_F$ the conductor of the character $\chi$, namely, the bigger ideal such that $\chi$ is trivial on $\hat\cO_{c}^\times/\hat\cO_{F}^\times$, where $\cO_c\subset E$ is the order of conductor $c$. 
We define $S_1:=\{v\mid N \mbox{ nonsplit in }T;\;\ord_v(c)<\ord_v(N)\}$. Let $c_1:=\prod_{\mathfrak{p}\mid c,\mathfrak{p}\not\in S_1}\mathfrak{p}^{\ord_{\mathfrak{p}}c}$ be the $S_1$-off part of $c$. Then, for any finite place $v$, there exists a $\cO_{F_v}$-order $\cO_{N,v}\subset B_v$ of discriminant $N\cO_{F_v}$ such that $\cO_{N,v}\cap E_v=\cO_{c_1,v}$. Such an order $\cO_{N,v}$ is called \emph{admissible for $(N,\chi_v)$} if at places $v\mid (N,c_1)$, it is the intersection of two maximal orders $\cO_{B,v}'$ and $\cO_{B,v}''$ of $B_v$ such that 
    \[
    \cO_{B,v}'\cap E_v=\cO_{c,v},\qquad \cO_{B,v}''\cap E_v=\left\{\begin{array}{ll}\cO_{c/N,v},&\mbox{if }\ord_v(c/N)\geq 0\\\cO_{E,v},&\mbox{otherwise.}\end{array}\right.
    \]
    \begin{remark}
    Our concept of admissibility coincides with that of \cite{CST} because their
    the condition (2) (see \cite[definition 1.3]{CST}) does not apply in our situation since $\chi_v\mid_{F_v^\times}$ is trivial.
\end{remark}
Let $\cO_N$ be an admissible $\cO_F$-order of $B$ for $(N,\chi)$ in the sense that, for any finite place $v$, $\cO_{N,v}$ is admissible for $(N,\chi_v)$. It follows that $\cO_N$ is an $\cO_F$-order of $B$ of discriminant $N$ such that $\cO_N\cap E=\cO_{c_1}$.

\begin{definition}\label{defV1}
Write $U^{S_1}=\prod_{v\not\in S_1}\cO_{N,v}^\times$.
Let $V(\pi^\infty,\chi)\subset \pi^\infty$ be the subspace of elements $\bigotimes'_vf_v\in \pi^{U^{S_1}}$  such that
$f_v$ is a $\chi_v^{-1}$-eigenform under $T(F_v)$, for all places $v\in S_1$.
    %
\end{definition}
If we assume that the root number $\epsilon(1/2,\pi_v,\chi_v)=\chi_v\eta_{T,v}(-1)\epsilon(B_v)$ for all $v\nmid\infty$ then the space $V(\pi^\infty,\chi)$ is actually one-dimensional by \cite[Proposition 3.7]{CST}. The following result computes the remaining local terms for such one dimensional space of test vectors.
\begin{lemma}\cite[Lemma 3.13]{CST}\label{lemexpli}
Write $U=\hat\cO_N^\times\subset G(\A_F^\infty)$ and let $\Phi_1\in V(\pi^\infty,\chi)$ and $\Phi_2\in V(\pi^\infty,\chi^{-1})$.
If the root number $\epsilon(1/2,\pi_v,\chi_v)=\chi_v\eta_{T,v}(-1)\epsilon(B_v)$ for all $v\nmid\infty$ then
     \[
     \prod_{v\nmid\infty}\left(\beta_v(\Phi_{1,v},\Phi_{2,v})\cdot\alpha_{v}(W_{\mathfrak{f},v},W^-_{\mathfrak{f},v})\right)=\frac{{\rm vol}(U_0(N))}{{\rm vol}(U)}|c_1^2 D|^{-\frac{1}{2}}2^{\#S_D}L_{c_1}(1,\eta_{T})^2\prod_{v\in S}L(1/2,\Pi_v,\chi_v)^{-1},
     \]
     where $S:=\{v\mid (N,Dc);\mbox{ if }v\parallel N\mbox{ then }\ord_v(c/N)\geq 0\}$, $S_D:=\{v\mid (N,D);\;\ord_v(c)<\ord_v(N)\}$, $|\cdot|$ denotes the norm of an ideal and $L_{c_1}(1,\eta_{T})=\prod_{v\mid c_1}L(1,\eta_{T,v})$.
\end{lemma}

\subsection{Restriction of the Harder-Eichler-Shimura morphism}

Recall that, for a fixed character $\lambda:G(F)/G(F)_+=G(F_\infty)/G(F_\infty)_+\rightarrow\pm1$, we have the morphism ${\rm ES}_\lambda$ of \S \ref{ES}:
\begin{equation}\label{eqES}
{\rm ES}_\lambda:M_{\underline k}(U)\longrightarrow  H^{r}(G(F),\cA^{\infty}(V(\underline{k}-2))^U(\lambda)).
\end{equation}

Similarly as in \eqref{primerdeltas}, the tensor product of the morphisms $\delta s_{\lambda_\sigma}:V(k_\sigma-2)(\lambda_\sigma)\rightarrow D(k_\sigma)$ of \eqref{deftildes}
 provides
\[
\underline{\delta s}_\lambda:V(\underline{k}-2)(\lambda)\longrightarrow D_\Sigma^\lambda(\underline{k}),
\] 
for every $\Sigma\subseteq\Sigma_B$.
Hence, we have $T(F)$-equivariant morphisms  
\begin{eqnarray*}
\varphi_\Sigma:C^{0}(T(\A_F),\C)\otimes V(\underline{k}-2)(\lambda)\otimes \cA^{\infty}(D^\lambda_\Sigma(\underline{k}),\C )^U&\longrightarrow& C^{\Sigma}(T(\A_F),\C)\\
\varphi_\Sigma((f\otimes \mu)\otimes \Phi)(z,t)&=& f(z,t)\cdot \Phi(\underline{\delta s}_\lambda(z^{-1}\mu))(z ,t),
\end{eqnarray*}
for any $z\in T(F_\infty)$ and $t\in T(\A_F^\infty)$.

\begin{remark}\label{canonicalphi}
In the special case $\Sigma=\emptyset$, the morphism $\varphi_\emptyset$ gives rise to a $T(F)$-equivariant
\begin{eqnarray*}
\varphi:C^{0}(T(\A_F),\C)\otimes V(\underline{k}-2)\otimes \cA^{\infty}(V(\underline{k}-2))^U(\lambda)&\longrightarrow& C^{\emptyset}(T(\A_F),\C),\\
\varphi((f\otimes\mu)\otimes\Phi)(z,t)&=&f(z,t)\cdot \lambda(z)\cdot\Phi(t)(\mu(Xy-Yx)^{\underline{k}-2})
\end{eqnarray*}
for all $z\in T(F_\infty)$, $t\in T(\A_F^{\infty})$. Indeed, for any $(f\otimes\mu)\in C^{0}(T(\A_F),\C)\otimes V(\underline{k}-2)$, we have $(f\lambda\otimes\mu)\in C^{0}(T(\A_F),\C)\otimes V(\underline{k}-2)(\lambda)$. Thus, $\varphi$ can be described as $\varphi(f\otimes\mu)=\varphi_\emptyset(f\lambda\otimes\mu)$.
\end{remark}

Recall $\partial:H^0(T(F),C^{\Sigma_T}(T(\A_F),\C))\rightarrow H^u(T(F),C^{\emptyset}(T(\A_F),\C))$ from Definition \ref{partialT}.

\begin{proposition}\label{propESpart}
We have that
\[
\partial\left(\varphi_{\Sigma_B}(F\cup \Phi)\right)=\varphi_\emptyset\left(F\cup {\rm ES}_\lambda\Phi\right)\in H^u(T(F),C^{\emptyset}(T(\A_F),\C)),
\]
for all $F\in H^0(T(F),C^{0}(T(\A_F),\C)\otimes V(\underline{k}-2)(\lambda))$ and $\Phi\in M_{\underline k}(U)$.
\end{proposition}
\begin{proof}
For any  $\tau\in\Sigma\subseteq\Sigma_B=\Sigma_T$, we can construct the $T(F)$-equivariant morphism
\begin{eqnarray*}
\bar\varphi_\Sigma:C^{0}(T(\A_F),\C)\otimes V(\underline{k}-2)(\lambda)\otimes \cA^{\infty}(\cB^\lambda_{\Sigma,\tau},\C )^U&\longrightarrow& C^{\Sigma}(T(\A_F),\C)\\
\bar\varphi_\Sigma((f\otimes \mu)\otimes \Phi)(z,t)&=& f(z,t)\cdot \Phi(\underline{s}_\lambda( z^{-1}\mu))(z ,t),
\end{eqnarray*}
where $\underline{s}_\lambda:V(\underline{k}-2)(\lambda)\longrightarrow \cB_{\Sigma,\tau}^\lambda$ is defined by $s_{\lambda_\tau}$ at $\tau$, $\delta s_{\lambda_\sigma}$ at $\sigma\in\Sigma\setminus\tau$ and the identity elsewhere.
We claim that $\bar\varphi_\Sigma$ fits in the following commutative diagram of $T(F)$-modules
\[
\xymatrix{
0\ar[d]&&0\ar[d]\\C^{0}(T(\A_F),\C)\otimes V(\underline{k}-2)\otimes \cA^{\infty}(D^\lambda_{\Sigma\setminus\{\tau\}}(\underline{k}),\C)^U\ar[d]_{\rho_{\lambda_\tau}^\ast}\ar[rr]^{\qquad\qquad\qquad\varphi_{\Sigma\setminus\{\tau\}}}&&C^{\Sigma\setminus\{\tau\}}(T(\A_F),\C)\ar[d]^\iota\\C^{0}(T(\A_F),\C)\otimes V(\underline{k}-2)\otimes \cA^{\infty}(\tilde\cB^\lambda_{\Sigma,\tau},\C)^U \ar[rr]^{\qquad\qquad\qquad\bar\varphi_\Sigma}\ar[d]&&C^{\Sigma}(T(\A_F),\C)\ar[d]^{\delta_\tau}\\C^{0}(T(\A_F),\C)\otimes V(\underline{k}-2)\otimes \cA^{\infty}(D^\lambda_\Sigma(\underline{k}),\C )^U\ar[d]\ar[rr]^{\qquad\qquad\qquad\varphi_\Sigma}&&C^{\Sigma}(T(\A_F),\C)\ar[d]\\0&&0
}
\]
where $\iota$ is the natural inclusion and the left column is the exact sequence \eqref{exseqglob}. Indeed, the upper square commutes since
\begin{eqnarray*}
\bar\varphi_{\Sigma}(\rho_{\lambda_\tau}^\ast((f\otimes \mu)\otimes \Phi))(z,t)&=&f(z,t)\cdot \rho_{\lambda_\tau}^\ast\Phi(\underline{s}_\lambda(z^{-1}\mu))(z,t)\\
&=&f(z,t)\cdot \Phi(\underline{\delta s}_\lambda( z^{-1}\mu))(z ,t),
\end{eqnarray*}
because $s_{\lambda_\tau}$ is a section of $\rho_{\lambda_\tau}$. On the other side, to prove that the lower square commutes, we have
\begin{eqnarray*}
\delta_\tau\bar\varphi_\Sigma((f\otimes \mu)\otimes \Phi)(z,t)&=& f(z,t)\cdot \delta_\tau\left(\Phi(\underline{s}_\lambda(z^{-1}\mu))(z,t)\right)\\
&=& f(z,t)\cdot \frac{d}{ds}\left(\Phi(\underline{s}_\lambda(e^{-s}z^{-1}\mu))(z e^s,t)\right)_{s=0}.
\end{eqnarray*}
If we write $F(x,y):=\Phi(\underline{s}_\lambda(e^{-y}z^{-1}\mu))(ze^x,t)$, we obtain
\[
\frac{\partial F}{\partial x}(0,0)=\delta^T_\tau\Phi(\underline{s}_\lambda(z^{-1}\mu))(z,t)=\Phi(\delta^T_\tau \underline{s}_\lambda(z^{-1}\mu))(z,t),\quad \frac{\partial F}{\partial y}(0,0)=-\Phi( \underline{s}_\lambda( z^{-1}\delta^T_\tau\mu))(z,t).
\]
Hence,
\[
\delta_\tau\bar\varphi_\Sigma((f\otimes \mu)\otimes \Phi)(z,t)= f(z,t)\cdot \left(\frac{\partial F}{\partial x}(0,0)+\frac{\partial F}{\partial y}(0,0)\right)=f(z,t)\cdot \left(\Phi( \underline{\delta s}_\lambda(z^{-1}\mu))(z,t)\right) =\varphi_\Sigma((f\otimes \mu)\otimes \Phi)(z,t).
\]
We deduce the result from the commutativity of the above diagram and the definitions of $\partial$ and ${\rm ES}_\lambda$.
\end{proof}

\subsection{Higher cohomology classes and global period integrals}\label{s:WaldsForm}

As in \S \ref{classWalds},
let us consider a locally polynomial character $\chi:T(\A_F)/T(F)\rightarrow\C^\times$ of degree at most $\frac{\underline{k}-2}{2}$. Namely, a character such that, when restricted to a small neighborhood of 1 in $T(F_\infty)$, it is of the form 
\[
\chi(t_\infty)=t_\infty^{\underline{m}},\qquad \underline{m}\in\Z^{\Sigma_F},\quad \frac{2-\underline{k}}{2}\leq \underline{m}\leq\frac{\underline{k}-2}{2}.
\]
By means of the morphisms \eqref{dualVP} and \eqref{PtoCoverC}, we can see $\chi$ as an element $\chi\in H^0(T(F),C^{0}(T(\A_F),\C)\otimes V(\underline{k}-2))$. Hence, given $\Phi\in M_{\underline k}(U)$, we can consider the following cup product with respect to $\varphi$ of remark \ref{canonicalphi}:
\[
\varphi\left(\chi\cup {\rm ES}_\lambda(\Phi)\right)\in H^u(T(F),C^\emptyset(T(\A_F),\C)).
\]
\begin{theorem}\label{mainTHM1}
Let $\chi:T(\A_F)/T(F)\rightarrow\C^\times$ be a locally polynomial character  such that  $\chi(t_\infty)=t_\infty^{\underline{m}}\chi_0(t_\infty)$, for some $\underline{m}\in\Z^{\Sigma_F}$, with $\frac{2-\underline{k}}{2}\leq \underline{m}\leq\frac{\underline{k}-2}{2}$ and some locally constant character $\chi_0$.
Then we have that, for all $\Phi\in  M_{\underline k}(U)$,
\[
\varphi\left(\chi\cup {\rm ES}_\lambda(\Phi)\right)\cap\eta=h\cdot \ell(\Phi(\underline{\delta s}_\lambda(\mu_{\underline{m}})),\chi)\in\C,
\]
where $h=\#T(\cO_F)_{+,{\rm tors}}$. 
\end{theorem}
\begin{proof}
One can check using \eqref{dualVP} and \eqref{PtoCoverC} that $\chi$ corresponds to $\chi_0\chi_f\otimes\mu_{\underline{m}}\in C^{0}(T(\A_F),\C)\otimes V(\underline{k}-2)$, where $\chi_f=\chi\mid_{T(\A_F^\infty)}$.
By remark \ref{canonicalphi}, proposition \ref{propESpart} and proposition \ref{partcapeta}, 
\begin{eqnarray*}
    \varphi\left(\chi\cup{\rm ES}_\lambda\Phi\right)\cap\eta&=&\varphi_\emptyset\left((\lambda\chi_0\chi_f\otimes\mu_{\underline{m}})\cup{\rm ES}_\lambda\Phi\right)\cap\eta=\partial\left(\varphi_{\Sigma_B}((\lambda\chi_0\chi_f\otimes\mu_{\underline{m}})\otimes\Phi)\right)\cap\eta\\
    &=&h\int_{T(\A_F)/T(F)}\varphi_{\Sigma_B}\left((\lambda\chi_0\chi_f\otimes\mu_{\underline{m}})\otimes\Phi\right)(\tau)d^\times \tau.
\end{eqnarray*}
Moreover, $z^{-1}\mu_{\underline{m}}=\lambda(z)z^{\underline{m}}\mu_{\underline{m}}$.
Hence, for all $ z\in T(F_\infty)$, $t\in T(\A^\infty)$,
\begin{eqnarray*}
\varphi_{\Sigma_B}\left((\lambda\chi_0\chi_f\otimes\mu_{\underline{m}})\otimes\Phi\right)(z,t)&=&\lambda(z)\chi_0(z)\chi_f(t)\Phi(\underline{\delta s}_\lambda(z^{-1}\mu_{\underline{m}}))(z,t)=z^{\underline{m}}\chi_0(z)\chi_f(t)\Phi(\underline{\delta s}_\lambda(\mu_{\underline{m}}))(z,t)\\
&=&\chi(z,t)\Phi(\underline{\delta s}_\lambda(\mu_{\underline{m}}))(z,t).
\end{eqnarray*}
Thus, we obtain
\[
\varphi\left(\chi\cup{\rm ES}_\lambda\Phi\right)\cap\eta=h\int_{T(\A_F)/T(F)}\chi(\tau)\cdot\Phi(\underline{\delta s}_\lambda(\mu_{\underline{m}}))(\tau )d^\times \tau=h\cdot \ell(\Phi(\underline{\delta s}_\lambda(\mu_{\underline{m}})),\chi),
\]
and the result follows.
\end{proof}

\subsection{Waldspurger formulas in higher cohomology}\label{Finalsection}

As in previous sections, fix $\lambda:G(F)/G(F)_+\rightarrow\pm 1$, let $\pi$ be an automorphic representation of $G$ of even weight $\underline{k}$ and conductor $N$.
Thanks to the Harder-Eichler-Shimura morphism ${\rm ES}_\lambda$ explained in \S\ref{ES}, we can realize $\pi^\infty$
in the space $H_\ast^r(G(F),\cA^\infty(V(\underline{k}-2))(\lambda))$. Given a locally polynomial character $\chi$ of degree at most $\frac{\underline{k}-2}{2}$, we can consider 
\[
\cP(\cdot,\chi):\pi^\infty\subset H_\ast^r(G(F),\cA^\infty(V(\underline{k}-2))(\lambda))\longrightarrow\C,\qquad\cP(\phi^\lambda,\chi):=\varphi\left(\chi\cup\phi^\lambda\right)\cap\eta\in\C.
\]
The complex number $\cP(\phi^\lambda,\chi)$ can be seen as an analogy in higher cohomology of the periods $\ell(f,\chi)$ introduced in \S \ref{classWalds}. Furthermore, we can express $\cP(\phi^\lambda,\chi)$ in terms of certain period $\ell(f,\chi)$ thanks to Theorem \ref{mainTHM1}. We will use this fact together with classical Waldspurger formulas introduced in \S \ref{classWalds} to obtain formulas analogous to that of proposition \ref{WaldsforPhi2} involving $\cP(\phi_\lambda,\chi)$. Moreover, we will use lemma \ref{lemexpli} to make those formulas explicit. Hence, for a module $\cO_N\subset B$ admissible for $(N,\chi)$, we recall the space $V(\pi^\infty,\chi)$ of definition \ref{defV1}, we consider its realization in $H^r(G(F),\cA^\infty(V(\underline{k}-2)^U)(\lambda))$, for $U=\hat\cO_N^\times$, and we fix non-zero elements $\phi_{01}^\lambda\in V(\pi^\infty,\chi)$ and $\phi_{02}^\lambda\in V(\pi^\infty,\chi^{-1})$. Write $L^S(1/2,\Pi,\chi)$ for the L-function with the local factors at places of a finite set $S$ removed. The following result is a direct consequence of proposition \ref{WaldsforPhi2}, lemma \ref{lemexpli} and theorem \ref{mainTHM1}:
\begin{theorem}[Waldspurger formula in higher cohomology]\label{THMwaldsHC1}
Let $\chi:T(\A_F)/T(F)\rightarrow\C^\times$ be a locally polynomial character  such that  $\chi(t_\infty)=t_\infty^{\underline{m}}\chi_0(t_\infty)$, for some $\underline{m}\in\Z^{\Sigma_F}$, with $\frac{2-\underline{k}}{2}\leq \underline{m}\leq\frac{\underline{k}-2}{2}$ and some locally constant character $\chi_0$.  
Given decomposable $\phi_1^\lambda=\bigotimes_{v\nmid\infty}'\phi_{1,v}^\lambda,\phi_2^\lambda=\bigotimes_{v\nmid\infty}'\phi_{2,v}^\lambda\in\pi^\infty\subset H_\ast^r(G(F),\cA^\infty(V(\underline{k}-2))(\lambda))$ ,
we have $\cP(\phi_i^\lambda,\chi)=0$ unless $\chi_{0}=\lambda$ and the root number $\epsilon(1/2,\pi_v,\chi_v)=\chi_v\eta_{T,v}(-1)\epsilon(B_v)$, for all $v\nmid\infty$. In that case, if $\phi_i^\lambda$ differ from $\phi_{0i}^\lambda$ in a finite set of places $\mathfrak{S}$ then  
\[
\cP(\phi_1^\lambda,\chi)\cdot \cP(\phi_2^\lambda,\chi^{-1})=\frac{2^{\#S_D}L_{c_1}(1,\eta_{T})^2h^2
C(\underline k,\underline m)}{|c_1^2 D|^{\frac{1}{2}}}\cdot L^S(1/2,\Pi,\chi)\cdot\frac{\langle \Phi_1,\Phi_2\rangle}{\langle \Psi,\Psi\rangle}\cdot\frac{{\rm vol}(U_0(N))}{{\rm vol}(U)}\prod_{v\in\mathfrak{S}}\frac{\beta_{v}(\phi_{1,v}^\lambda,\phi_{2,v}^\lambda)}{\beta_{v}(\phi_{01,v}^\lambda,\phi_{02,v}^\lambda)},
\]
where $S:=\{v\mid (N,Dc);\mbox{ if }v\parallel N\mbox{ then }\ord_v(c/N)\geq 0\}$, $S_D:=\{v\mid (N,D);\;\ord_v(c)<\ord_v(N)\}$, $\Phi_i\in H^0_\ast(G(F),\cA^\infty(D(\underline k),C))$ are such that ${\rm ES}_\lambda(\Phi_i)=\phi_i^\lambda$ and
    \[
    C(\underline k,\underline m)=(-1)^{\left(\sum_{\sigma\not\in\Sigma_B}\frac{k_{\sigma}-2}{2}\right)}4^{r_\R}\cdot(32\pi)^{r_\C}\left(\frac{1}{\pi}\right)^{d-r-r_\C}\prod_{\sigma\mid\infty}\prod_{\tilde\sigma\mid\sigma}\frac{\Gamma(\frac{k_{\tilde\sigma}}{2}-m_{\tilde\sigma})\Gamma(\frac{k_{\tilde\sigma}}{2}+m_{\tilde\sigma})}{(-1)^{m_{\tilde\sigma}}(2\pi)^{k_{\tilde\sigma}}}.
    \]
\end{theorem}

\section{Local archimedean theory}\label{loctheory}

In this section we will study the local archimedean representations generated by an automorphic form of even weight $\underline{k}\in(2\N)^{[F:\Q]}$.

\subsection{The $(\cG,K)$-modules (of discrete series) of even weight $k$: Case $\PGL_2(\R)$}\label{GKR}

Let us consider $\PGL_2(\R)$ as a real Lie group. Its maximal compact subgroup $K$ is the image of ${\rm O}(2)=\SO(2)\ltimes\left\langle\mbox{\tiny$\left(\begin{array}{cc}1&\\&-1\end{array}\right)$}\right\rangle$, where
\[
\SO(2):=\left\{\left(\begin{array}{cc}a&b\\- b& a\end{array}\right):\;|a|^2+|b|^2=1\right\}\subset \SL_2(\R).
\]
Notice that $\SO(2)\simeq S^1$. For any $\theta\in S^1$ we will write
\[
\kappa(\theta):=\left(\begin{array}{cc}\cos\theta&\sin\theta\\-\sin\theta&\cos\theta\end{array}\right).
\]

Notice that
any $g\in\GL_2(\R)_+$ admits a decomposition
\begin{equation}\label{eqstar}
g=u\left(\begin{array}{cc}y&x\\&y^{-1}\end{array}\right)\kappa(\theta),\qquad y\in\R^\times,\;u\in\R_+,\;x\in\R,\;\theta\in S^1.
\end{equation}

For any character $\chi:\R^\times\rightarrow\C^\times$, let us consider the induced representation of $\PGL_2(\R)_+$
\[
\cB(\chi):=\left\{f:\GL_2(\R)_+\rightarrow\C:\;f\left(\left(\begin{array}{cc}t_1&x\\&t_2\end{array}\right)g\right)=\chi(t_1/t_2)\cdot f(g)\right\}.
\]
By the decomposition \eqref{eqstar}, we have a correspondence
\begin{equation}\label{eqBI}
\cB(\chi)\stackrel{\simeq}{\longrightarrow}I(\chi)=\{f:S^1\rightarrow\C:\;f(\theta+\pi)= f(\theta)\}.
\end{equation}

For any integer $n\in\Z$, let us consider the character
$e^{in\theta}:S^1\longrightarrow\C^\times$. 
Notice that the only characters appearing in $I(\chi)$ are those of the form $e^{2in\theta}$. The corresponding Lie algebra is $\cG_\R={\rm Lie}(\PGL_2(\R))\simeq\{g\in\M_2(\R),\;{\rm Tr}g=0\}$. 
\begin{definition}
Write $I(\chi,n)$ for the subspace $\C e^{2in\theta}$ in $I(\chi)$, then we consider
\[
\tilde I(\chi):=\bigoplus_{n\in\Z}I(\chi,n)\subseteq I(\chi).
\] 
Finally, we write $\tilde\cB(\chi)\subset\cB(\chi)$ for the preimage 
under the isomorphism \eqref{eqBI} of $\tilde I(\chi)$.
\end{definition}

\begin{proposition}
Assume that $\chi(t)=\chi_k(t):=t^{\frac{k}{2}}$, for an even integer $k\in 2\Z$. We have a morphism of $\GL_2(\R)_+$-representations:
\[
\rho:\cB(\chi_k)\longrightarrow V(k-2); \qquad \rho(f)(P):=\int_{S^1}f(\kappa(\theta))P(-\sin\theta,\cos\theta)d\theta
\]
\end{proposition}
\begin{proof}
The result follows from \cite[Lemma 2.6.1]{Bump} since $\tiny{\left(\begin{array}{cc}t_1&x\\&t_2\end{array}\right)}\mapsto t_1/t_2$ is the modular quasicharacter and the function
\[
h:\GL_2(\R)_+\longrightarrow\C,\qquad h(g)=f(g)\cdot P(c,d)\cdot\det(g)^{\frac{2-k}{2}},\quad g=\left(\begin{array}{cc}a&b\\c&d\end{array}\right),
\]
is in the induced representation by the quasicharacter. 
\end{proof}
The preimage under \eqref{eqBI} of the subspace $I(\chi_k,n)$ is generated  by the function 
\begin{equation}\label{fn}
f_n\left(u\left(\begin{array}{cc}y&x\\&y^{-1}\end{array}\right)\kappa(\theta)\right)=y^k\cdot e^{2in\theta}\in\cB(\chi_k).
\end{equation}
Notice that $\tilde \cB(\chi_k)$ has a natural structure of $(\cG_\R,\SO(2))$-module. To extend $\tilde \cB(\chi_k)$ to a $(\cG_\R,K)$-module one has to define an action of $w=\big(\begin{smallmatrix}1&\\&-1\end{smallmatrix}\big)$. We have two possibilities providing two extensions  $\cB(\chi_k)^\pm$ by defining $w(f_n)=\pm(-1)^{\frac{k-2}{2}} f_{-n}$  (see \cite[\S 1.3]{ESsanti}). It turns out that $D(k) :=\tilde \cB(\chi_k)\cap \ker(\rho) $ is the unique sub-$(\cG_\R,K)$-module for both $\tilde\cB(\chi_k)^\pm$. We have the exact sequences of $(\cG,K)$-modules:
\begin{eqnarray*}
0\longrightarrow &D(k)&\longrightarrow \tilde \cB(\chi_k)^+\stackrel{\rho_+}{\longrightarrow} V(k-2)\longrightarrow 0,\\
0\longrightarrow &D(k)&\longrightarrow \tilde \cB(\chi_k)^-\stackrel{\rho_-}{\longrightarrow} V(k-2)(-1)\longrightarrow 0,
\end{eqnarray*}
where $V(k-2)(-1)$ is the representation $V(k-2)$ twisted by the character $g\mapsto{\rm sign}\det(g)$. Indeed,
\begin{eqnarray*}
\rho(wf_n)(P)&=& (-1)^{\frac{k-2}{2}}\rho(\pm f_{-n})(P) =\frac{(-1)^{\frac{k-2}{2}}}{\pm 1}\int_{S^1}e^{-2in\theta}P(-\sin\theta,\cos\theta)d\theta= \frac{(-1)^{\frac{k-2}{2}}}{\pm 1}\int_{S^1}e^{2in\theta}P(\sin\theta,\cos\theta)d\theta\\
&=& \pm\int_{S^1}e^{2in\theta}(wP)(-\sin\theta,\cos\theta)d\theta= \pm(w\rho(f_n))(P).
\end{eqnarray*}

If we write $P_m(X,Y):=(Y+iX)^{m}(Y-iX)^{k-2-m}\in\cP(k-2)$, for $0\leq m\leq k-2$, we compute
\begin{equation}\label{norm1}
\rho(f_n)(P_m)=\int_{S^1}e^{(k-2-2m+2n)i\theta} d\theta=\left\{\begin{array}{ll}1,&n=m-\frac{k-2}{2},\\0,&n\neq m-\frac{k-2}{2},\end{array}\right.
\end{equation}
for the good normalization of the Haar measure of $S^1$. This implies that the kernel of $\rho_\pm$ is generated by $f_n$ with $|n|\geq\frac{k}{2}$, and we deduce that 
\begin{equation}\label{exseq1}
0\longrightarrow D(k)=\sum_{|n|\geq \frac{k}{2}}I(\chi_k,n)\stackrel{\iota_\pm}{\longrightarrow} \tilde \cB(\chi_k)^\pm=\sum_{n\in \Z}I(\chi_k,n) \stackrel{\rho_\pm}{\longrightarrow} V(k-2)(\pm)\longrightarrow 0,
\end{equation}
where $\iota_\pm(f_n)=(\pm 1)^{\frac{1-{\rm sign}(n)}{2}}f_n$.
\begin{remark}\label{remarkonESconmor}
    Notice that the morphism $\rho$ and, hence, the exact sequences \eqref{exseq1} depend on the choice of the Haar measure $d\theta$. In this note we have chosen $d\theta$ so that ${\rm vol}(S^1)=1$. Nevertheless, in \cite{ESsanti} it is chosen so that ${\rm vol}(S^1)=\pi$, so the second morphism in the exact sequence is $\pi\cdot\rho$ there.
\end{remark}

\subsubsection{Splittings of the $(\cG,K)$-module exact sequences}\label{splitGK}

In this section we aim to construct sections to the exact sequences \eqref{exseq1} and their dual counterparts
\[
0\longrightarrow V(k-2)(\pm)\stackrel{\rho_\pm^\vee}{\longrightarrow} \tilde \cB(\chi_{2-k})^\pm\stackrel{\imath^\vee_{\pm}}{\longrightarrow} D(k) \longrightarrow 0,
\]
where 
\[
\rho_\pm^\vee(\mu)\left(\begin{array}{cc}a&b\\c&d\end{array}\right)=\mu\left(\left|\begin{array}{cc}X&Y\\c&d\end{array}\right|^{k-2}\right)\cdot(ad-bc)^{\frac{2-k}{2}}.
\]
Of course such exact sequences do not split as $(\cG,{\rm O}(2))$-modules, but they split as ${\rm O}(2)$-modules. Hence, 
there exist unique ${\rm O}(2)$-equivariant sections of $\rho_\pm$ and $\imath^\vee_{\pm}$, namely, ${\rm O}(2)$-equivariant morphisms  
\[
s_\pm:V(k-2)(\pm)\longrightarrow \tilde\cB(\chi_k)^\pm,\qquad s'_\pm:D(k)\longrightarrow \tilde\cB(\chi_{2-k})^\pm
\]
such that $\rho_\pm\circ s_\pm={\rm id}$ and $s'_\pm\circ \imath^\vee_{\pm}={\rm id}$.
\begin{remark}\label{normalizationD}
    We still have to specify the morphism $\imath^\vee_{\pm}$ (and therefore $s_\pm'$). For now we only have that $\tilde \cB(\chi_{2-k})^\pm/\rho_\pm^\vee(V(k-2))$ is isomorphic to $ D(k)$, but the isomorphism is not described. By Schur's lemma this choice is unique up to constant.  
\end{remark}

We aim to write down $s_\pm$ explicitly.
Notice that, as $\SO(2)$-module, $V(k-2)(\pm)\simeq \sum_{\frac{2-k}{2}\leq n\leq\frac{k-2}{2}}I(\chi_k,n)$. Thus, for any $\mu\in V(k-2)(\pm)$, we must impose 
\[
s_\pm(\mu)\in  \sum_{\frac{2-k}{2}\leq n\leq\frac{k-2}{2}}I(\chi_k,n)\subset \tilde\cB(\chi_k)^\pm,\qquad s_\pm(\mu)(\kappa(\theta))=\sum_{\frac{2-k}{2}\leq n\leq\frac{k-2}{2}}a_n(\mu)\cdot e^{2in\theta}.
\]
Hence, finding $s_\pm(\mu)$ is equivalent to finding coefficients $a_n(\mu)$ for $\frac{2-k}{2}\leq n\leq\frac{k-2}{2}$. Since $\rho_\pm\circ s_\pm={\rm id}$ we have 
\[
\mu(P)=\rho_\pm(s_\pm(\mu))(P)=\int_{S^1}s_\pm(\mu)(\kappa(\theta))P(-\sin\theta,\cos\theta)d\theta=\sum_{\frac{2-k}{2}\leq n\leq\frac{k-2}{2}}a_n(\mu)\int_{S^1}e^{2ni\theta}P(-\sin\theta,\cos\theta)d\theta.
\]
Evaluating at the base $P_m(X,Y)$, $m=0,\cdots,k-2$, and using orthogonality we deduce that $a_n(\mu)=\mu\left(P_{n+\frac{k-2}{2}}\right)$. 
We conclude that
\begin{equation}\label{eqsmu}
s_\pm(\mu)=\sum_{\frac{2-k}{2}\leq n\leq\frac{k-2}{2}}\mu\left(P_{n+\frac{k-2}{2}}\right)f_n.
\end{equation}


\subsubsection{Diagonal torus and splittings
}\label{diagTor}

Write $\imath:T(\R)\hookrightarrow\PGL_2(\R)$ for the split diagonal torus 
\[
\imath(t)=\left(\begin{array}{cc}t&\\&1\end{array}\right).
\]
Write also 
\[
\cT_\R:={\rm Lie}(T(\R))=
\R \delta\subset \cG_\R,\qquad \delta:=\left(\begin{array}{cc}1&0\\0&0\end{array}\right),
\]
satisfying $\exp(t\delta)=\imath(e^t)$.
For the section $s_\pm:V(k-2)(\pm)\longrightarrow \tilde\cB(\chi_k)^\pm$ provided in  \S \ref{splitGK}, let us consider 
\[
\delta s_\pm:V(k-2)(\pm)\longrightarrow \iota_{\pm}D(k)\subset \cB(\chi_k)^\pm,\qquad \delta s_\pm(\mu):=\delta (s_\pm(\mu))-s_\pm(\delta\mu).
\]
Indeed, $\delta s_\pm(\mu)\in\iota_{\pm}D(k)$ since $\rho_\pm\circ\delta s_\pm=0$ because $\rho_\pm$ is a morphism of $(\cG,K)$-modules and $\rho_\pm\circ s_\pm={\rm id}$.

\begin{definition}\label{defmumR}
Let $\mu_m\in V(k-2)$ be such that
\[
\mu_m\left(\left|\begin{array}{cc}X& Y\\x&y\end{array}\right|^{k-2}\right)=x^{\frac{k-2}{2}-m}y^{\frac{k-2}{2}+m},\qquad \frac{2-k}{2}\leq m\leq\frac{k-2}{2}.
\]
It is easy to check that $\imath(t)\mu_m=t^{-m}\mu_m$, for all $t\in T(\R)$.
\end{definition}
\begin{remark}\label{remarkmumxy}
    Notice that $\mu_{m}$ corresponds to $x^{\frac{k-2}{2}-m}y^{\frac{k-2}{2}+m}$ under the isomorphism \eqref{dualVP}.
\end{remark}

\begin{proposition}\label{keylemma1}
For any $\frac{2-k}{2}\leq m\leq\frac{k-2}{2}$, we have that
\[
\delta s_\pm(\mu_m)=\frac{(k-1)}{2}\left((-i)^{\frac{k-2}{2}+m}f_{\frac{k}{2}}+i^{\frac{k-2}{2}+m}f_{-\frac{k}{2}}\right).
\]
\end{proposition}
\begin{proof}
We compute on the one hand, 
\[
\delta \mu_m=\frac{d}{dx}{\rm exp}(x\delta)\mu_m\mid_{x=0}=\frac{d}{dx}\imath(e^x)\mu_m\mid_{x=0}=\frac{d}{dx}e^{-mx}\mid_{x=0}\mu_m=-m\mu_m.
\]
On the other hand, by equation \eqref{eqsmu} we have
\[
s_\pm(\mu_m)=\sum_{\frac{2-k}{2}\leq n\leq\frac{k-2}{2}}\mu_m\left(P_{n+\frac{k-2}{2}}\right)f_n.
\]
Notice that 
\[
\left|\begin{array}{cc}X& Y\\x&y\end{array}\right|^{k-2}=
2^{2-k}\left|\begin{array}{cc}Y+iX& Y-iX\\x-yi&-iy-x\end{array}\right|^{k-2}=2^{2-k}\sum_{j=0}^{k-2}\binom{k-2}{j}P_{j}(X,Y)z^j\bar z^{k-2-j},
\]
where $z=(x+yi)$ and $\bar z=(x-yi)$. This implies by definition that
\begin{equation}\label{eqauxmum}
\frac{x^{\frac{k-2}{2}-m}y^{\frac{k-2}{2}+m}}{2^{2-k}}=\sum_{\frac{2-k}{2}\leq n\leq\frac{k-2}{2}}\binom{k-2}{n+\frac{k-2}{2}}\mu_m\left(P_{n+\frac{k-2}{2}}\right)z^{n+\frac{k-2}{2}}\bar z^{\frac{k-2}{2}-n}.
\end{equation}
On the other side, we compute
\[
\frac{x^{\frac{k-2}{2}-m}y^{\frac{k-2}{2}+m}}{2^{2-k}}=(z+\bar z)^{\frac{k-2}{2}-m}(z-\bar z)^{\frac{k-2}{2}+m}i^{\frac{2-k}{2}-m}=i^{\frac{2-k}{2}-m}\sum_{\frac{2-k}{2}\leq n\leq\frac{k-2}{2}}C(n)\cdot z^{n+\frac{k-2}{2}}\bar z^{\frac{k-2}{2}-n}
\]
where 
\[
C(n)=\sum_{j}(-1)^{n+j+m}\binom{\frac{k-2}{2}-m}{j}\binom{\frac{k-2}{2}+m}{m-n+j}.
\]
Comparing both expressions, we deduce 
\[
\mu_m(P_{n+\frac{k-2}{2}})=i^{\frac{2-k}{2}-m}C(n)\binom{k-2}{n+\frac{k-2}{2}}^{-1}=i^{\frac{2-k}{2}-m}\sum_{j}\frac{(-1)^{n+j+m}\binom{\frac{k-2}{2}-m}{j}\binom{\frac{k-2}{2}+m}{m-n+j}}{\binom{k-2}{n+\frac{k-2}{2}}}.
\]

Since $\imath(t)\mu_m=t^{-m}\mu_m$, we have that
\[
\delta s_\pm(\mu_m)=\delta(s_\pm(\mu_m))+ms_\pm(\mu_m)=\sum_{\frac{2-k}{2}\leq n\leq\frac{k-2}{2}}\mu_m(P_{n+\frac{k-2}{2}})\left(\delta(f_n)+mf_n\right)
\]

It is easy to check that, for any {\tiny$\left(\begin{array}{cc}x&y\\m&n\end{array}\right)$}$\in\GL_2(\R)_+$ of determinant $\Delta$, 
\[
f_n\left(\begin{array}{cc}x&y\\m&n\end{array}\right)=\frac{\Delta^{\frac{k}{2}}(m^2+n^2)^{n-\frac{k}{2}}}{(mi+n)^{2n}}.
\]
This implies that,
\begin{eqnarray*}
\delta f_n(\kappa(\theta))&=&\frac{d}{dt}f_n(\kappa(\theta)\cdot{\rm exp}(t\delta))\mid_{t=0}=\frac{d}{dt}f_n\left(\kappa(\theta)\cdot\imath(e^t)\right)\mid_{t=0}=\frac{d}{dt}f_n\left(\begin{array}{cc}\ast&\ast\\-e^t\sin(\theta)&\cos(\theta)\end{array}\right)_{t=0}\\
&=&\frac{d}{dt}\left(\frac{e^{\frac{kt}{2}}\left(e^{2t}\sin^2\theta+\cos^2\theta\right)^{n-\frac{k}{2}}}{\left(\cos\theta-e^{t}i\sin\theta \right)^{2n}}\right)_{t=0}=e^{2ni\theta}\left(\frac{k}{2}+(2n-k)\sin^2\theta+2nie^{i\theta}\sin\theta\right)
\end{eqnarray*}
Hence,
\[
\delta f_n=\frac{1}{2}\left(\frac{k}{2}+n\right)f_{n+1}+\frac{1}{2}\left(\frac{k}{2}-n\right)f_{n-1}.
\]
Thus, 
\[
\delta s_\pm(\mu_m)=\sum_{\frac{2-k}{2}\leq n\leq\frac{k-2}{2}}\mu_m(P_{n+\frac{k-2}{2}})\left(\frac{1}{2}\left(\frac{k}{2}+n\right)f_{n+1}+\frac{1}{2}\left(\frac{k}{2}-n\right)f_{n-1}+mf_n\right).
\]
Recall that $\delta s_\pm(\mu_m)\in \iota_\pm D(k)$, hence $\delta s_\pm(\mu_m)\in I(\chi_k,-\frac{k}{2})\oplus  I(\chi_k,\frac{k}{2})$. This implies that almost all terms cancel out and the only terms that survive are
\[
\delta s_\pm(\mu_m)=\frac{1}{2}\mu_m(P_{k-2})(k-1)f_{\frac{k}{2}}+\frac{1}{2}\mu_m(P_{0})(k-1)f_{-\frac{k}{2}}.
\]
Using the above computations 
\begin{eqnarray*}
\mu_m(P_{0})&=&i^{\frac{2-k}{2}-m}\sum_{j}(-1)^{\frac{k-2}{2}+j+m}\binom{\frac{k-2}{2}-m}{j}\binom{\frac{k-2}{2}+m}{m+\frac{k-2}{2}+j}=i^{\frac{k-2}{2}+m},\\
\mu_m(P_{k-2})&=&i^{\frac{2-k}{2}-m}\sum_{j}(-1)^{\frac{k-2}{2}+j+m}\binom{\frac{k-2}{2}-m}{j}\binom{\frac{k-2}{2}+m}{m-\frac{k-2}{2}+j}=i^{\frac{2-k}{2}-m},
\end{eqnarray*}
and the result follows.
\end{proof}

\subsubsection{Inner products and pairings
}\label{calcI(chi)}

Let us consider the $\PGL_2(\R)$-invariant pairing
\[
\langle\;,\;\rangle_\cB:\cB(\chi_k)^\pm\times \cB(\chi_{2-k})^\pm\longrightarrow\C;\qquad \langle f,h\rangle_\cB:=\int_{S^1}f(\kappa(\theta))\cdot h(\kappa(\theta))d\theta.
\]
It satisfies the property
\[
\langle \imath_\pm f,\rho_\pm^\vee\mu\rangle_\cB=0, \qquad\mbox{for all }\quad f\in D(k),\;\mu\in V(k-2)(\pm).
\]
Thus, it provides (the unique up to constant) $\PGL_2(\R)$-invariant bilinear inner products on $D(\underline{k})$ and $V(\underline{k}-2)(\pm)$: 
\begin{eqnarray*}
\langle\;,\;\rangle:D(k)\times D(k)\longrightarrow\C;&\qquad& \langle f,h\rangle:=\langle  \iota_\pm f,s'_\pm h\rangle_\cB,\\
\langle\;,\;\rangle_V:V(k-2)(\pm)\times V(k-2)(\pm)\longrightarrow\C;&\qquad& \langle \mu_1,\mu_2\rangle_V:=\langle s_\pm(\mu_1),\rho_\pm^\vee(\mu_2)\rangle_\cB.
\end{eqnarray*}
\begin{remark}
    As discussed in remark \ref{normalizationD}, the morphism $s_\pm'$ is not completely specified. Notice that, since $w(f_{\frac{k}{2}})=\pm (-1)^{\frac{k-2}{2}}f_{-\frac{k}{2}}$, we have that $s'_\pm(f_{\frac{k}{2}})=Cf_{\frac{k}{2}}$ and $s'_\pm(f_{-\frac{k}{2}})=Cf_{-\frac{k}{2}}$  for some non-zero constant $C$. We will normalize $s_\pm'$ so that $C=1$. Notice that this implies $\langle f_{\frac{k}{2}},f_{-\frac{k}{2}}\rangle_\cB=1$.   
\end{remark}
\begin{remark}\label{rempairVR}
Notice that the isomorphism \eqref{dualVP} gives rise to a different $\GL_2(\R)$-invariant pairing 
\[
\langle\;,\;\rangle':V(k-2) (\pm)\times V(k-2) (\pm)\longrightarrow\C,\qquad \langle\mu_1,\mu_2\rangle'=\mu_2\mu_1\left(\left|\begin{array}{cc}X_1&Y_1\\X_2&Y_2\end{array}\right|^{k-2}\right).
\]
Given $\mu_m,\mu_{m'}$ as in Definition \ref{defmumR}, we have that
\begin{equation}\label{relmummum}
\langle\mu_m,\mu_{m'}\rangle'=\left\{\begin{array}{ll}
(-1)^{m+\frac{k-2}{2}}\binom{k-2}{\frac{k-2}{2}-m}^{-1},&m=-m',\\
0,&m\neq-m'.
\end{array}\right.
\end{equation}
By Schur's Lemma both pairings $\langle\cdot,\cdot\rangle'$ and $\langle\cdot,\cdot\rangle_V$ must coincide up to constant. To compute such a constant let us consider $(y^{k-2})^\vee\in V(k-2)(\pm)$, the element of the dual basis of $\{x^iy^{k-2-i}\}_i$ corresponding to $y^{k-2}$. It is clear that 
\[
\left\langle\mu_{\frac{k-2}{2}},(y^{k-2})^\vee\right\rangle'=(y^{k-2})^\vee\mu_{\frac{k-2}{2}}\left(\left|\begin{array}{cc}X&Y\\x&y\end{array}\right|^{k-2}\right)=(y^{k-2})^\vee\left(y^{k-2}\right)=1.
\]
On the other side, 
\begin{eqnarray*}
\left\langle\mu_{\frac{k-2}{2}},(y^{k-2})^\vee\right\rangle_V&=&\int_{S^1}s_\pm\left(\mu_{\frac{k-2}{2}}\right)(\kappa(\theta))\cdot \rho_\pm^\vee((y^{k-2})^\vee)(\kappa(\theta))d\theta\\
&=&\sum_{\frac{2-k}{2}\leq n\leq \frac{k-2}{2}}\mu_{\frac{k-2}{2}}(P_{n+\frac{k-2}{2}})\int_{S^1}e^{2in\theta}\cdot (y^{k-2})^\vee\left(\left|\begin{array}{cc}x&y\\-\sin\theta&\cos\theta\end{array}\right|^{k-2}\right)d\theta\\
&=&\sum_{\frac{2-k}{2}\leq n\leq \frac{k-2}{2}}\mu_{\frac{k-2}{2}}(P_{n+\frac{k-2}{2}})(2i)^{2-k}\binom{k-2}{\frac{k-2}{2}-n}(-1)^{\frac{k-2}{2}-n}=1,
\end{eqnarray*}
by equation \eqref{eqauxmum}.
This implies that both pairings coincide, and relations \eqref{relmummum} also hold for $\langle\cdot,\cdot\rangle_V$.
\end{remark}

Given the morphism $\delta s_\pm:V(k-2)(\pm)\longrightarrow \iota_{\pm}D(k)$, 
we aim to compute
\[
\cI(\chi,m):=\int_{T(\R)}\chi(t)\left\langle \imath(t)\delta s_\pm(\mu_m),\delta s_\pm(\mu_{-m})\right\rangle_{\cB} t^md^\times t, 
\]
where $\chi$ is a locally constant character and $dt^\times$ is the measure of \eqref{defdtx}.

Recall that $\delta s_\pm(\mu_m):=\delta (s_\pm(\mu_m))-s_\pm(\delta\mu_m)$. Moreover, 
\[
\delta\mu_m=\frac{d}{dx}(\imath(e^x)\mu_m)\mid_{x=0}=\frac{d}{dx}(e^{-mx}\mu_m)\mid_{x=0}=-m\cdot\mu_m.
\] 
Hence, if we consider the function 
\[
F(t)=t^m\langle \imath(t)s_\pm(\mu_m),s'_\pm{\delta s_\pm(\mu_{-m})}\rangle_\cB,
\]
we compute
\begin{eqnarray*}
t\frac{d}{dt}F(t)&=&mt^m\langle \imath(t) s_\pm(\mu_m),s'_\pm{\delta s_\pm(\mu_{-m})}\rangle_\cB+t^m\frac{d}{dx}\langle \imath(te^x) s_\pm(\mu_m),s'_\pm{\delta s_\pm(\mu_{-m})}\rangle_\cB\mid_{x=0}\\
&=&t^m\left(-\langle \imath(t) s_\pm(\delta\mu_m),s'{\delta s_\pm(\mu_{-m})}\rangle_\cB+\langle \imath(t) \delta(s_\pm(\mu_m)),s'_\pm{\delta s_\pm(\mu_{-m})}\rangle_\cB\right)\\
&=&t^m\langle \imath(t)\delta s_\pm(\mu_m),\delta s_\pm(\mu_{-m})\rangle_\cB.
\end{eqnarray*}

\begin{proposition}\label{calcinfty0}
We have that 
\[
F(\infty)-F(0^+)=(-1)^{m}2^{k-1}(k-1)\binom{k-2}{\frac{k-2}{2}-m}^{-1}.
\]
where $F(\ast):=\lim_{t\to\ast} F(t)$.
\end{proposition}
\begin{proof}
By proposition \ref{keylemma1}, we have that 
\[
s'_\pm{\delta s_\pm(\mu_{-m})}=\frac{(k-1)}{2}\left((-i)^{\frac{k-2}{2}-m}f_{\frac{k}{2}}+i^{\frac{k-2}{2}-m}f_{-\frac{k}{2}}\right),
\]
where now $f_{\pm\frac{k}{2}}$ are the corresponding functions in $\cB(\chi_{2-k})^\pm$. In fact, for any $f_n$ in $\cB(\chi_{2-k})^\pm$,
\[
f_n\left(\begin{array}{cc}x&y\\r&s\end{array}\right)=\Delta^{\frac{2-k}{2}}\frac{(s-ri)^{n-\frac{2-k}{2}}}{(s+ri)^{n+\frac{2-k}{2}}},\qquad \Delta=(xs-ry)>0,
\]
This implies that, if $t>0$,
\begin{eqnarray*}
&&\imath(t^{-1})s'_\pm{\delta s_\pm(\mu_{-m})}(\kappa(\theta))=s'_\pm{\delta s_\pm(\mu_{-m})}\left(\begin{array}{cc}t^{-1}\cos\theta&\sin\theta\\-t^{-1}\sin\theta&\cos\theta\end{array}\right)=\\
&=&\frac{(k-1)t^{\frac{k-2}{2}}}{2}\left((-i)^{\frac{k-2}{2}-m}\frac{(\cos\theta+t^{-1}\sin\theta i)^{k-1}}{(\cos\theta-t^{-1}\sin\theta i)}+i^{\frac{k-2}{2}-m}\frac{(\cos\theta-t^{-1}\sin\theta i)^{k-1}}{(\cos\theta+t^{-1}\sin\theta i)}\right)\\
&=&\frac{(k-1)t^{\frac{k-2}{2}}}{i^{\frac{k-2}{2}-m}}\sum_{0\leq j\equiv \frac{k-2}{2}-m\;({\rm mod}\;2)}\binom{k}{j}\left(\frac{i}{t}\right)^j\frac{\cos\theta^{k-j}\sin\theta^j}{(\cos^2\theta+t^{-2}\sin^2\theta)}
\end{eqnarray*}
Hence we have that $F(t)=t^m\langle s_\pm(\mu_m),\imath(t^{-1})s'_\pm{\delta s_\pm(\mu_{-m})}\rangle_\cB$ is given by
\[
F(t)=(-1)^{\frac{k}{2}}(it)^{\frac{2-k}{2}+m}(k-1)\sum_{0\leq j\equiv \frac{k-2}{2}-m\;({\rm mod}\;2)}\binom{k}{j}\int_{S^1}s_\pm(\mu_m)(\kappa(\theta))\frac{(it)^j(-\sin\theta)^{k-j}\cos\theta^j}{(\sin^2\theta+t^2\cos^2\theta)}d\theta.
\]

If $t\to 0^+$, we have that $(\sin^2\theta+t^2\cos^2\theta)^{-1}=\sin^{-2}\theta\sum_{s\geq 0}(it\cos\theta\sin^{-1}\theta)^{2s}$, hence 
\[
F(t)=(-1)^{\frac{k}{2}}(it)^{\frac{2-k}{2}+m}(k-1)\sum_{0\leq j\equiv \frac{k-2}{2}+m\;({\rm mod}\;2)}\binom{k}{j}\sum_{s\geq 0}C(j+2s)(it)^{j+2s}
\]
where 
\[
C(n):=\int_{S^1}s_\pm(\mu_m)(\kappa(\theta))(-\sin\theta)^{k-2-n}\cos\theta^{n}d\theta.
\]
When $0\leq n\leq k-2$, we have
\[
C(n)=\int_{S^1}s_\pm(\mu_m)(\kappa(\theta))\rho_\pm^\vee\left(\mu_{n-\frac{k-2}{2}}\right)(\kappa(\theta))d\theta=\left\langle \mu_m,\mu_{n-\frac{k-2}{2}}\right\rangle_V.
\]
Hence, by remark \ref{rempairVR} and equation \eqref{relmummum}, we have that $C(n)=0$ except when $-m=n-\frac{k-2}{2}$ where $C(\frac{k-2}{2}-m)=(-1)^{m+\frac{k-2}{2}}\binom{k-2}{\frac{k-2}{2}-m}^{-1}$. This implies that
\[
F(0^+)=(-1)^{m-1}(k-1)\binom{k-2}{\frac{k-2}{2}-m}^{-1}\sum_{s\geq 0}\binom{k}{\frac{k-2}{2}-m-2s}.
\]

If $t\to \infty$, we have that $(\sin^2\theta+t^2\cos^2\theta)^{-1}=(t\cos\theta)^{-2}\sum_{s\geq 0}((it\cos\theta)^{-1}\sin\theta)^{2s}$, hence 
\[
F(t)=-(-1)^{\frac{k}{2}}(it)^{\frac{2-k}{2}+m}(k-1)\sum_{0\leq j\equiv \frac{k-2}{2}+m\;({\rm mod}\;2)}\binom{k}{j}\sum_{s\geq 0}C(j-2-2s)(it)^{j-2-2s}.
\]
Using the same calculations, we obtain
\[
F(\infty)=(-1)^{m}(k-1)\binom{k-2}{\frac{k-2}{2}-m}^{-1}\sum_{s\geq 0}\binom{k}{\frac{k-2}{2}-m+2s+2}.
\]

We conclude that 
\[
F(\infty)-F(0^+)=\frac{(-1)^{m}(k-1)}{\binom{k-2}{\frac{k-2}{2}-m}}\sum_{s\in\Z}\binom{k}{\frac{k-2}{2}-m+2s}=\frac{(-1)^{m}\frac{(k-1)}{2}}{\binom{k-2}{\frac{k-2}{2}-m}}\left((1+1)^k+\frac{(1-1)^k}{(-1)^{\frac{k-2}{2}+m}}\right)=\frac{2^{k-1}(-1)^{m}(k-1)}{\binom{k-2}{\frac{k-2}{2}-m}},
\]
and the result follows.
\end{proof}

We write $ \tilde\cB(\chi_k)^\lambda$, for $\lambda=\pm 1$. The following result computes the local integral we are looking for: 

\begin{theorem}\label{calcLI1}
Given a locally constant character $\chi:T(\R)\rightarrow\pm 1$, we consider 
\[
\cI(\chi,m):=\int_{T(\R)}\chi(t)\langle\imath(t)\delta s_\lambda(\mu_m),\delta s_\lambda(\mu_{-m})\rangle t^md^\times t,
\]
where $d^\times t$ is the Haar measure \eqref{defdtx}. Then 
\[
\cI(\chi,m)=\left\{\begin{array}{ll}(-1)^{m}2^{k}(k-1)\binom{k-2}{\frac{k-2}{2}-m}^{-1},&\chi(-1)=\lambda,\\0,&\chi(-1)=\lambda.\end{array}\right.
\]
\end{theorem}
\begin{proof}
By the explicit description of $\delta s_\lambda(\mu_m)$ given in proposition \ref{keylemma1}, we compute
\begin{equation}\label{eqauxwdeltas}
    w\delta s_\lambda(\mu_m)=\lambda(-1)^{m}\delta s_\lambda(\mu_m),\qquad w=\left(\begin{array}{cc}1&\\&-1\end{array}\right).
\end{equation}
This implies that, for $t<0$,
\begin{eqnarray*}
\chi(t)\langle\imath(t)\delta s_\lambda(\mu_m),\delta s_\lambda(\mu_{-m})\rangle t^m&=&\chi(-1)\langle\imath(-t)w\delta s_\lambda(\mu_m),\delta s_\lambda(\mu_{-m})\rangle t^m\\
&=&\chi(-1)\lambda\langle\imath(-t)\delta s_\lambda(\mu_m),\delta s_\lambda(\mu_{-m})\rangle (-t)^m.
\end{eqnarray*}
Since $d^\times t=|t|^{-1}dt$, we obtain
\[
I(\chi,m)=\left(1+\chi(-1)\lambda\right)\int_0^\infty t^m\langle\imath(t)\delta s_\lambda(\mu_m),\delta s_\lambda(\mu_{-m})\rangle d^\times t.
\]
We have seen that 
$t\frac{dF(t)}{dt}=t^m\langle\imath(t)\delta s_\lambda(\mu_m),\delta s_\lambda(\mu_{-m})\rangle$.
Thus, by the above calculations,
\[
I(\chi,m)=\left(1+\chi(-1)\lambda\right)\int_0^{\infty}t\frac{dF(t)}{dt}\frac{dt}{t}=\left(1+\chi(-1)\lambda\right)\left(F(\infty)-F(0^+)\right)
\]
hence the result follows from proposition \ref{calcinfty0}.
\end{proof}


\subsubsection{Pairings of canonical sections}

Let us consider the canonical $\PGL_2(\R)$-invariant element
\[
\Upsilon=\iota^{-1}\left|\begin{array}{cc}
    x_1 & y_1 \\
     x_2 &  y_2
\end{array}\right|^{k-2}\in V(k-2)\otimes V(k-2),
\]
where $\iota:V(k-2)\otimes V(k-2)\stackrel{\simeq}{\rightarrow}\cP(k-2)\otimes \cP(k-2)$ is the isomorphism of \eqref{dualVP}.
Since we have the morphism $\delta s_+:V(k-2)\longrightarrow D(k)$, we can consider $\delta s_+(\Upsilon)\in D(k)\otimes D(k)$ and apply the inner product $\langle\;;\;\rangle$ defined in the beginning of \S \ref{calcI(chi)}. Similarly, we can apply the pairing $\langle\;,\;\rangle'$ of remark \ref{rempairVR} to $\Upsilon$. 
\begin{lemma}\label{lemoonpairupsilon1}
    We have that 
    \[
    \frac{\langle\delta s_+(\Upsilon)\rangle}{\langle\Upsilon\rangle'}=(k-1)\cdot2^{k-3}.
    \]
\end{lemma}
\begin{proof}
On the one side, we have by remark \ref{remarkmumxy} and proposition \ref{keylemma1}
\begin{eqnarray*}
    \langle\delta s_+(\Upsilon)\rangle&=&\sum_{j=0}^{k-2}\binom{k_\sigma-2}{j}(-1)^{k-2-j}\left\langle\delta s_+\left(\mu_{\frac{k-2}{2}-j}\right),\delta s_+\left(\mu_{\frac{k-2}{2}+j}\right)\right\rangle\\
    &=&\frac{(k-1)^2}{4}\sum_{j=0}^{k_\sigma-2}\binom{k-2}{j}(-1)^{k-2-j}\left\langle i^{2-k+j}f_{\frac{k}{2}}+i^{k_\sigma-2-j}f_{-\frac{k}{2}},i^{2-k-j}f_{\frac{k}{2}}+i^{k-2+j}f_{-\frac{k}{2}}\right\rangle_\cB\\
    &=&\frac{(k-1)^2}{2}\sum_{j=0}^{k-2}\binom{k-2}{j}=(k-1)^2\cdot2^{k-3},
\end{eqnarray*}
since $\langle f_{\frac{k}{2}},f_{-\frac{k}{2}}\rangle_\cB=\langle f_{-\frac{k}{2}},f_{\frac{k}{2}}\rangle_\cB=1$ and $\langle f_{-\frac{k}{2}},f_{-\frac{k}{2}}\rangle_\cB=\langle f_{\frac{k}{2}},f_{\frac{k}{2}}\rangle_\cB=0$. 

On the other side, by remark \ref{remarkmumxy},
\[
\langle\Upsilon\rangle'=\sum_{j=0}^{k-2}\binom{k-2}{j}(-1)^j\langle \mu_{\frac{k-2}{2}-j},\mu_{j-\frac{k-2}{2}}\rangle`=\sum_{j=0}^{k-2}1=(k-1)
\]
Hence, the result follows.
\end{proof}
\begin{remark}
    As observed in the proof of lemma \ref{lemoonpairupsilon1} above, we have that $\langle\delta s_+(\Upsilon)\rangle=\langle\delta s_-(\Upsilon)\rangle$. 
\end{remark}
Another interesting element to consider is $\delta s_+(\mu_{0})$. It is clear by the description of $\delta s_+$ given in proposition \ref{keylemma1} that $\big(\begin{smallmatrix}
    1&\\&-1\end{smallmatrix}\big)\delta s_+(\mu_{0})=\delta s_+(\mu_{0})$ (see equation \eqref{eqauxwdeltas}). In fact, $\delta s_+(\mu_{0})$ generates the subspace of $D(k)$ invariant under $\big(\begin{smallmatrix}
    1&\\&-1\end{smallmatrix}\big)\in O(2)$ with minimal weight. 
\begin{lemma}\label{pairv021}
We have that
\[
  \langle\delta s_+(\mu_{0}),\delta s_+(\mu_{ 0})\rangle=\frac{(k-1)^2}{2}.
\]
\end{lemma}
\begin{proof}
    By proposition \ref{keylemma1},
    \[
    \langle\delta s_+(\mu_0),\delta s_+(\mu_0)\rangle=\frac{(k-1)^2}{4}\left\langle(-i)^{\frac{k-2}{2}}f_{\frac{k}{2}}+i^{\frac{k-2}{2}}f_{-\frac{k}{2}},(-i)^{\frac{k-2}{2}}f_{\frac{k}{2}}+i^{\frac{k-2}{2}}f_{-\frac{k}{2}}\right\rangle_\cB=\frac{(k-1)^2}{2},
    \]
    and the result follows.
\end{proof}



\subsection{The $(\cG,K)$-modules of even weight $\underline{k}$: Case $\PGL_2(\C)$}\label{GKC}

Let us consider $\PGL_2(\C)$ as a real Lie group. Its maximal compact subgroup $K$ is the image of 
\[
\SU(2):=\left\{\left(\begin{array}{cc}a&b\\-\bar b&\bar a\end{array}\right):\;|a|^2+|b|^2=1\right\}\subset \SL_2(\C).
\]
Notice that $\SU(2)\simeq S^3:=\{(a,b\in\C^2);|a|^2+|b|^2=1\}$. For any $(a,b)\in S^3$ we will write
\[
\kappa(a,b):=\left(\begin{array}{cc}a&b\\-\bar b&\bar a\end{array}\right).
\]

Similarly as in \eqref{eqstar},
any $g\in\GL_2(\C)$ admits a decomposition
\begin{equation}\label{eqstar2}
g=u\left(\begin{array}{cc}r&x\\&r^{-1}\end{array}\right)\kappa(a,b),\qquad u\in\C^\times\;,x\in\C\;,r\in\R^\times\;,(a,b)\in S^3.    
\end{equation}

\subsubsection{Haar measure of $\SU(2)$}


Notice that $S^3$ is the boundary of $D\subset\C^2$, where
\[
D=\{(z_1,z_2)\in\C^2, \;|z_1|^2+|z_2|^2\leq 1\}.
\]
Moreover, if we consider $\C^2$ with the standard right action of $\GL_2(\C)$, then both subspaces $D$ and $S^3$ are clearly $\SU(2)$-invariant. The measure on $D$ given by $dz_1dz_2d\bar z_1d\bar z_2$ is $\SU(2)$-invariant. Indeed, the action of $\kappa(a,b)$ provides a change of variables given by the same matrix $\kappa(a,b)$, thus the action on the differential $dz_1dz_2$ (and $d\bar z_1d\bar z_2$) is given by the determinant of $\kappa(a,b)$, which is $1$. If we perform the change of variables
\[
z_1=R\cos\theta e^{i\alpha},\quad z_2=R\sin\theta e^{i\beta},\quad \bar z_1=R\cos\theta e^{-i\alpha},\quad \bar z_2=R\sin\theta e^{-i\beta}.
\]
where $\alpha,\beta\in [0,2\pi]$, $R\in [0,1]$ and $\theta\in [0,\pi/2]$, the corresponding Jacobian determinant is
\[
\left|\begin{array}{cccc}\cos\theta e^{i\alpha}&-R\sin\theta e^{i\alpha}&iR\cos\theta e^{i\alpha}&\\
\sin\theta e^{i\beta}&R\cos\theta e^{i\beta}&&iR\sin\theta e^{i\beta}\\
\cos\theta e^{-i\alpha}&-R\sin\theta e^{-i\alpha}&-iR\cos\theta e^{-i\alpha}&\\
\sin\theta e^{-i\beta}&R\cos\theta e^{-i\beta}&&-iR\sin\theta e^{-i\beta}\end{array}\right|=-R^3\sin2\theta.
\] 
Notice that any $f\in C^\infty(S^3,\C)$ can be seen as a function of $D$ that is constant in $R$, therefore
\begin{eqnarray*}
\int_D f(z_1,z_2)dz_1dz_2d\bar z_1d\bar z_2&=&-\int_0^1\int_{S^1}\int_{S^1}\int_0^{\pi/2}\sin2\theta \cdot f(\cos\theta e^{i\alpha},\sin\theta e^{i\beta}) d\theta d\alpha d\beta R^3dR\\
&=&\frac{-1}{4}\int_{S^1}\int_{S^1}\int_0^{\pi/2}\sin2\theta \cdot f(\cos\theta e^{i\alpha},\sin\theta e^{i\beta}) d\theta d\alpha d\beta.
\end{eqnarray*}
From the $\SU(2)$-invariance of $dz_1dz_2d\bar z_1d\bar z_2$, we deduce that 
\[
\int_{S^3}f(a,b)d(a,b):=\int_{S^1}\int_{S^1}\int_0^{\pi/2}\sin2\theta \cdot f(\cos\theta e^{i\alpha},\sin\theta e^{i\beta}) d\theta d\alpha d\beta,
\]
is a Haar measure for $\SU(2)$.

\begin{lemma}\label{auxlem}
We have that
\[
I_{n_1,n_2,m_1,m_2}:=\int_{S^3}a^{n_1}b^{m_1}\bar a^{n_2}\bar b^{m_2}d(a,b)=\left\{\begin{array}{lc}\frac{1}{n_1+m_1+1}\binom{n_1+m_1}{n_1}^{-1},&n_1=n_2,\;m_1=m_2,\\
0,&\mbox{otherwise.}\end{array}\right.
\]
\end{lemma}
\begin{proof}
By the above calculations
\[
I_{n_1,n_2,m_1,m_2}=2\int_{S^1}\int_{S^1}\int_0^{\pi/2}\sin\theta^{1+m_1+m_2} \cos\theta^{1+n_1+n_2} e^{(n_1-n_2)i\alpha}e^{(m_1-m_2)i\beta} d\theta d\alpha d\beta,
\]
hence clearly $I_{n_1,n_2,m_1,m_2}=0$ unless $m_1=m_2$ and $n_1=n_2$. If this is the case, the integral can be solved by parts using a simple induction:
\begin{eqnarray*}
I_{n_1,n_1,m_1,m_1}=2\int_0^{\pi/2}\sin\theta^{1+2m_1} \cos\theta^{1+2n_1} d\theta=\frac{1}{n_1+m_1+1}\binom{n_1+m_1}{n_1}^{-1},
\end{eqnarray*}
hence the result follows.
\end{proof}

\subsubsection{Quotients of Principal series}

Let us consider the finite dimensional $\C$-representation of $\SU(2)$:
\[
\cP(n):={\rm Sym}^{n}(\C^2)=\{P\in\C[x,y],\; P(ax,ay)=a^nP(x,y)\},\quad V(n)=\cP(n)^\vee, 
\]
with action
\[
(\kappa(a,b)P)(x,y):=P((x,y)\kappa(a,b))=P(ax-\bar by,bx+\bar ay),\qquad (a,b)\in S^3. 
\]
Notice that by \eqref{dualVP} we have $\cP(n)\simeq V(n)$. By \cite{Hall}, these are all the irreducible representations of $\SU(2)$. Thus, the irreducible representations of the compact subgroup of $\PGL_2(\C)$ are $V(2n)$, where $n\in\N$. 

For any character $\chi:\C^\times\rightarrow\C^\times$,
let us consider the induced $\PGL_2(\C)$-representation
\[
\cB(\chi):=\left\{f:\GL_2(\C)\rightarrow\C:\;f\left(\left(\begin{array}{cc}t_1&x\\&t_2\end{array}\right)g\right)=\chi(t_1/t_2)\cdot f(g)\right\}.
\]
By equation \eqref{eqstar2} we have a correspondence
\begin{equation}\label{corrBfunct}
\cB(\chi)\stackrel{\simeq}{\longrightarrow}\{f:S^3\rightarrow\C:\;f(e^{i\theta}a,e^{i\theta}b)=\chi(e^{2i\theta})\cdot f(a,b)\}.
\end{equation}

Let us consider the embedding $S^1\hookrightarrow S^3$, $e^{i\theta}\mapsto (e^{i\theta},0)$.
By \eqref{corrBfunct}, the $\SU(2)$-representation $\cB(\chi)$ is induced by the character $\chi^2$ of $S^1$. By Frobenius reciprocity,
\[
\Hom_{\SU(2)}(V(2n),\cB(\chi))\stackrel{\simeq}{\longrightarrow} \cP(2n)(\chi):=\{P\in \cP(2n),\;\kappa(e^{i\theta},0)P=\chi(e^{-2i\theta})\cdot P\}
\]
If we write $\chi(e^{i\theta})=e^{i\lambda\theta}$, we can easily compute the subspace $\cP(2n)(\chi)$:
\[
\cP(2n)(\chi)=\{P\in \cP(2n),\;P(e^{i\theta}x,e^{-i\theta}y)=e^{-2\lambda i\theta}\cdot P(x,y)\}=\C x^{n-\lambda}y^{n+\lambda}.
\]
This implies that 
\begin{equation}\label{homsV}
\Hom_{\SU(2)}(V(2n),\cB(\chi))=\left\{\begin{array}{cc}\C,&|\lambda|\leq n\\
0,&\mbox{otherwise.}\end{array}\right.
\end{equation}
\begin{definition}\label{defvarphi}
Assume that $\chi(e^{i\theta})=e^{i\lambda\theta}$. For all $n\geq|\lambda|$
we will fix $\varphi_n\in\Hom_{\SU(2)}(V(2n),\cB(\chi))$ to be 
\begin{equation}\label{morfvarphn}
\varphi_n(\mu)(\alpha,\beta):=\mu\left(\left|\begin{array}{cc}\alpha&\beta\\x&y\end{array}\right|^{n+\lambda}\left|\begin{array}{cc}-\bar\beta&\bar\alpha\\x&y\end{array}\right|^{n-\lambda}\right).
\end{equation}
\end{definition}

Let us consider the Lie algebra $\cG_\C={\rm Lie}(\PGL_2(\C))\simeq\{g\in\M_2(\C),\;{\rm Tr}g=0\}$. 
\begin{definition}
Write $\cB(\chi,n)$ for the image of $V(2n)$ through $\varphi_n$. Hence the subspace
\[
\tilde \cB(\chi):=\bigoplus_{n\geq |\lambda|}\cB(\chi,n)\subseteq \cB(\chi)
\] 
is a natural $(\cG_\C,K)$-module.
\end{definition}

\begin{proposition}
Let $\Sigma$ be the set of $\R$-isomorphisms $\sigma:\C\simeq\C$, and
let $\underline{k}=(k_\sigma)_{\sigma\in\Sigma}\in (2\N)^2$ with $k_\sigma\geq 2$.
Assume that $\chi_{\underline{k}}(t):=\prod_{\sigma\in\Sigma}\sigma(t)^{\frac{k_\sigma}{2}}$. We have a morphism of $\GL_2(\C)$-representations
\begin{eqnarray*}
\rho&:&\cB(\chi_{\underline{k}})\longrightarrow V(\underline{k}-2):=\bigotimes_{\sigma\in\Sigma}V(k_\sigma-2); \\
 \rho(f)\left(\bigotimes_{\sigma\in\Sigma}P_\sigma\right)&=& \int_{S^3}f(\alpha,\beta)\left(\prod_{\sigma\in\Sigma} P_\sigma(-\sigma(\bar\beta),\sigma(\bar\alpha))\right)d(\alpha,\beta),
\end{eqnarray*}
where $\GL_2(\C)$ acts on each $V(k_\sigma-2)$ by means of $\sigma$.
\end{proposition}
\begin{proof}
Let us consider $\delta$ the modular quasicharacter
\[
\delta: P:=\left\{\left(\begin{array}{cc}r&x\\&r^{-1}\end{array}\right), r\in \R_{>0},\;x\in\C\right\}\longrightarrow\R,\qquad \delta\left(\begin{array}{cc}r&x\\&r^{-1}\end{array}\right)=r^4.
\]
Thus, the function $h:\GL_2(\C)\rightarrow\C$
 \[
 h(g)=f(g)\left(\prod_{\sigma\in\Sigma} P_\sigma(\sigma(c),\sigma(d))\sigma(\det(g))^{\frac{2-k_\sigma}{2}}\right),\qquad g=\left(\begin{array}{cc}a&b\\c&d\end{array}\right),
 \]
lies in the induced representation by $\delta$. 
The result follows from \cite[lemma 2.6.1]{Bump}.
\end{proof}

Write $D(\underline{k}):=\tilde \cB(\chi_{\underline{k}})\cap \ker(\rho)$. It is the unique sub-$(\cG_\C,K)$-module of $\tilde\cB(\chi_{\underline{k}})$ and, by definition, we have the exact sequence of $(\cG_\C,K)$-modules:
\begin{equation}\label{exseq2}
0\longrightarrow D({\underline{k}})\stackrel{\iota}{\longrightarrow} \tilde \cB(\chi_{\underline{k}})\stackrel{\rho}{\longrightarrow} V({\underline{k}}-2)\longrightarrow 0.
\end{equation}
Notice that, in this case, $\lambda=\frac{k_{\rm id}-k_{c}}{2}$, where $c$ is the complex conjugation. If we write
\[
x^{\underline m}y^{\underline k-2-\underline m}:=\bigotimes_{\sigma\in\Sigma}x_\sigma^{m_\sigma}y_\sigma^{k_\sigma-2-m_\sigma}\in\cP(\underline{k}-2):=\bigotimes_{\sigma\in\Sigma}\cP(k_\sigma-2),\qquad \underline{m}=(m_\sigma) \leq \underline{k}-2,
\]
and we consider $(x^{2n})^\vee\in V(2n)$, the element of the dual basis of $\{x^my^{2n-m}\}_m$ corresponding to $x^{2n}$, one obtains 
\begin{eqnarray*}
&&\rho(\varphi_n((x^{2n})^\vee))(x^{\underline m}y^{\underline k-2-\underline m})=\int_{S^3}\bar a^{n-\frac{k_{\rm id}-k_{c}}{2}}b^{n+\frac{k_{\rm id}-k_{c}}{2}}(-\bar b)^{m_{\rm id}}\bar a^{k_{\rm id}-2-m_{\rm id}}(-b)^{m_c}a^{k_c-2-m_c}d(a,b).
\end{eqnarray*}
Thus, by lemma \ref{auxlem}, 
\begin{equation}\label{norm2}
\rho(\varphi_n((x^{2n})^\vee))(x^{\underline m}y^{\underline k-2-\underline m})=\left\{\begin{array}{ll}\frac{(- 1)^{m_{\rm id}+m_c}}{n+\frac{k_{\rm id}-1+k_c-1}{2}}\binom{n+\frac{k_{\rm id}-2+k_c-2}{2}}{m_{\rm id}}^{-1},&n=m_{\rm id}-m_c-\frac{k_{\rm id}-k_{c}}{2},\\0,&n\neq m_{\rm id}-m_c-\frac{k_{\rm id}-k_{c}}{2},\end{array}\right.
\end{equation}
Since $\kappa(a,b)(x^{2n})^\vee$ generates $V(2n)$ and $0\leq m_\sigma\leq k_\sigma-2$, we deduce that 
\[
0\longrightarrow D(\underline{k})=\bigoplus_{n>\frac{k_{\rm id}-2+k_c-2}{2}}\cB(\chi_{\underline{k}},n)\longrightarrow \tilde \cB(\chi_{\underline{k}})=\bigoplus_{n\geq \left|\frac{k_{\rm id}-k_c}{2}\right|}\cB(\chi_{\underline{k}},n)\longrightarrow V(\underline{k}-2)\longrightarrow 0.
\]

\subsubsection{Splittings of the $(\cG,K)$-module exact sequences}\label{splitGK2}

As in \S \ref{splitGK}, we can construct $\SU(2)$-equivariant sections of the exact sequence \eqref{exseq2} and its dual counterpart
we  have 
\[
0\longrightarrow V(\underline{k}-2)\stackrel{\rho^\vee}{\longrightarrow} \tilde \cB(\chi_{2-\underline{k}})\stackrel{\imath^\vee}{\longrightarrow} D(\underline{k}) \longrightarrow 0,
\]
where $\chi_{2-\underline{k}}(t)=\prod_{\sigma\in\Sigma}\sigma(t)^{\frac{2-k_\sigma}{2}}$ and
\[
\rho^\vee(\mu)\left(\begin{array}{cc}a&b\\c&d\end{array}\right)=\mu\left(\left|\begin{array}{cc}X&Y\\c&d\end{array}\right|^{\underline{k}-2}\right)\cdot\chi_{2-\underline k}(ad-bc),\qquad  \left|\begin{array}{cc}X&Y\\c&d\end{array}\right|^{\underline{k}-2}:=\bigotimes_{\sigma\in\Sigma}\left|\begin{array}{cc}X_\sigma&Y_\sigma\\\sigma(c)&\sigma(d)\end{array}\right|^{k_\sigma-2}\in\cP(\underline{k}-2).
\]

\begin{definition}
Write $M:=\frac{\sum_\sigma (k_\sigma-2)}{2}$ and $\lambda:=\frac{k_{\rm id}-k_c}{2}$.
Notice that as a $\SU(2)$-representation $V(\underline{k}-2)\simeq \bigoplus_{|\lambda|\leq n\leq M}V(2n)$. Indeed, if $|\lambda|\leq n\leq M$ then the integer numbers $r_1:=n+\lambda$, $r_2:=n-\lambda$ and $r_3:=M-n$ are positive. In this situation we say that the even integers $k_{\rm id}-2$, $k_c-2$ and $2n$ are \emph{balanced}. Hence we can consider the $\SU(2)$-equivariant morphism
\begin{eqnarray*}
t_n:V(\underline{k}-2)=V(k_{\rm id}-2)\otimes V(k_c-2)\longrightarrow \cP(2n)\simeq V(2n),\quad t_n(\mu_{\rm id}\otimes\mu_c)=\mu_{\rm id}\mu_c(\Delta_n),\\
\Delta_n(X_{\rm id},Y_{\rm id},X_{c},Y_{c},x,y):=\left|\begin{array}{cc}X_{\rm id}&Y_{\rm id}\\x&y\end{array}\right|^{r_1}\left|\begin{array}{cc}-Y_{c}&X_{c}\\x&y\end{array}\right|^{r_2}\left|\begin{array}{cc}X_{\rm id}&Y_{\rm id}\\-Y_c&X_c\end{array}\right|^{r_3}.\qquad\qquad
\end{eqnarray*}
The polynomial $\Delta$ is called \emph{Clebsch-Gordan element} (see \cite[\S 3]{GS} for other situations where such an element naturally appears). 
The morphism $t_n$ is unique up to constant. 
\end{definition}
\begin{remark}\label{normalizationD2}
    Similarly as in remark \ref{normalizationD}, we still have to specify the morphism $\imath^\vee$. By Schur's lemma the morphism $\imath^\vee$ is univocally determined if we impose that $\iota^\vee(\varphi_{M+1}(\mu))=\varphi_{M+1}(\mu)$, which is what we will assume from now on.
\end{remark}

\begin{lemma}\label{lemsect2}
There exist unique $\SU(2)$-equivariant sections of $\rho$ and $\imath^\vee$, namely, $\SU(2)$-equivariant morphisms  
\[
s:V(\underline{k}-2)\longrightarrow \tilde\cB(\chi_{\underline{k}}),\qquad s':D(\underline{k})\longrightarrow \tilde\cB(\chi_{2-\underline{k}})
\]
such that $\rho\circ s={\rm id}$ and $s'\circ \imath^\vee={\rm id}$. More precisely,
\[
s(\mu)=\sum_{|\lambda|\leq n\leq M}(2n+1)\binom{2n}{n+\lambda}\cdot\varphi_n\circ t_n(\mu).
\]
\end{lemma}
\begin{proof}
The decomposition of $\tilde\cB(\chi)$ as a direct sum of irreducible $\SU(2)$-representations and the resultant decomposition of the subrepresentations $V(\underline{k}-2)$ and $D(\underline{k})$ imply the existence of $s$ and $s'$. Let us show that $s$ is as described:
The unicity of $\varphi_n$ and $t_n$ implies that there exist $C_n\in\C$ such that 
\[
s=\sum_{|\lambda|\leq n\leq M}C_n\cdot\varphi_n\circ t_n.
\]
Let us fix $n$ in the range $[|\lambda|, M]$. Assume that $\mu\in V(\underline{k}-2)$ is such that $t_n(\mu)=(x^{2n})^\vee$ and $t_{n'}(\mu)=0$ for $n'\neq n$. Hence, by relation \eqref{norm2}, for any $\underline{m}$ with $m_{\rm id}-m_c-\lambda=n$ we have 
\[
\mu(X^{\underline m}Y^{\underline k-2-\underline m})=\rho\circ s(\mu)(X^{\underline m}Y^{\underline k-2-\underline m})=C_n\cdot \rho\circ\varphi_n((x^{2n})^\vee)(X^{\underline m}Y^{\underline k-2-\underline m})=C_n\frac{(-1)^{m_{\rm id}+m_c}}{n+M+1}\binom{n+M}{m_{\rm id}}^{-1}
\]
and $\mu(X^{\underline m}Y^{\underline k-2-\underline m})=0$ if $m_{\rm id}-m_c-\lambda\neq n$. We compute that 
\begin{eqnarray*}
t_n(\mu)&=&\mu\left(\sum_{i,j,s}\binom{r_1}{i-j-s}\binom{r_2}{j+s}\binom{r_3}{s}X_{\rm id}^{i-j}Y_{\rm id}^{r_1+r_3-i+j}X_c^{r_2-j}Y_c^{r_3+j}y^{i}x^{2n-i}(-1)^{i+j+s}\right)\\
&=&\frac{(-1)^{r_3}C_n}{n+M+1}y^{2n}\sum_{j,s}(-1)^{s+j}\binom{r_1}{2n-j-s}\binom{r_2}{j+s}\binom{r_3}{s}\binom{n+M}{2n-j}^{-1}.
\end{eqnarray*}
Recall that $r_1=n+\lambda$, $r_2=n-\lambda$ and $r_3=M-n$. This implies that $2n-r_1=r_2\leq j+s \leq r_2$, hence we have $j+s=r_2$. Since under the identification \eqref{dualVP} $(x^{2n})^\vee$ corresponds to $y^{2n}$, we obtain
\begin{eqnarray*}
y^{2n}&=&t_n(\mu)=\frac{C_n}{n+M+1}y^{2n}\sum_{s}\binom{r_3}{s}\binom{n+M}{r_1+s}^{-1}=C_ny^{2n}\frac{r_3!r_2!r_1!}{(M+n+1)!}\sum_{s}\binom{M-D-s}{r_2}\binom{r_1+s}{r_1}\\
&=&C_ny^{2n}\frac{r_3!r_2!r_1!}{(M+n+1)!}\binom{M+n+1}{2n+1}=\frac{C_n}{2n+1}y^{2n}\binom{2n}{r_1}^{-1},
\end{eqnarray*}
and the result follows.
\end{proof}

\subsubsection{Diagonal torus and splittings
}\label{diagTor2}

Assume that we have the diagonal torus $\imath:T(\C)\hookrightarrow\PGL_2(\C)$, 
$\imath(t)=\big(\begin{smallmatrix}t&\\&1\end{smallmatrix}\big)$.
We write
\[
\cT_\C:={\rm Lie}(T(\C))=
\C \delta\subset \cG_\C,\qquad \delta:=\left(\begin{array}{cc}1&\\&\end{array}\right).
\]
Thus $\exp(z\delta)=\imath(e^z)$, for all $z\in\C$. 

Again, given the section $s:V(\underline{k}-2)\rightarrow \tilde\cB(\chi_{\underline{k}})$, we consider the well defined morphism
\[
\delta s:V(\underline{k}-2)\longrightarrow D(\underline{k})\subset \cB(\chi_{\underline{k}}),\qquad \delta s(\mu):=\delta (s(\mu))-s(\delta\mu).
\]


\begin{definition}\label{defmum}
For any $\frac{2-\underline{k}}{2}\leq\underline{m}=(m_{\rm id},m_c) \leq\frac{\underline{k}-2}{2} $,
let $\mu_{\underline{m}}\in V(\underline{k}-2)$ be such that
\[
\mu_{\underline{m}}\left(\left|\begin{array}{cc}X& Y\\x&y\end{array}\right|^{\underline{k}-2}\right)=x^{\frac{\underline k-2}{2}-\underline m}y^{\frac{\underline k-2}{2}+\underline m},\qquad \left|\begin{array}{cc}X& Y\\x&y\end{array}\right|^{\underline{k}-2}:=\prod_\sigma \left|\begin{array}{cc}X_\sigma& Y_\sigma\\x_\sigma&y_\sigma\end{array}\right|^{k_\sigma-2}.
\]
It can be checked analogously as in \S \ref{diagTor} that $\imath(t)\mu_{\underline{m}}=t^{-\underline{m}}\mu_{\underline{m}}$. 
\end{definition}
\begin{remark}\label{remarkmumxy2}
    Similarly as in remark \ref{remarkmumxy}, $\mu_{\underline m}$ corresponds to $x^{\frac{\underline k-2}{2}-\underline m}y^{\frac{\underline k-2}{2}+\underline m}$ under the isomorphism $V(\underline k-2)\simeq \cP(\underline k-2)$ induced by \eqref{dualVP}.
\end{remark}


\begin{proposition}\label{keylemma2}
For any $\underline{\mu}\in V(\underline{k}-2)$ we have that 
\[
\delta s(\underline{\mu})=-2\binom{2M}{k_{\rm id}-2}\varphi_{M+1}(t_M(\underline{\mu})^*),
\]
where $t_M(\underline{\mu})^*\in V(2M+2)$ is given by
$t_M(\underline{\mu})^*(P):=t_M(\underline{\mu})\left(\frac{\partial P}{\partial x\partial y}\right)$.
\end{proposition}
\begin{proof}
Notice that, for any {\tiny$\left(\begin{array}{cc}\ast&\ast\\r&s\end{array}\right)$}$\in\GL_2(\C)$ of determinant $\Delta$,
\begin{equation}\label{descvarphi}
\varphi_n(\mu)\left(\begin{array}{cc}\ast&\ast\\r&s\end{array}\right):=\Delta^{\frac{\underline{k}}{2}}\frac{\mu\left(P_n\right)}{(|r|^2+|s|^2)^{M+n+2}},\qquad P_n(r,s):=\left|\begin{array}{cc}\bar s&-\bar r\\x&y\end{array}\right|^{n+\lambda}\left|\begin{array}{cc} r&s\\x&y\end{array}\right|^{n-\lambda}.
\end{equation}
Hence,
\begin{eqnarray*}
\delta \varphi_n(\mu)(\alpha,\beta)&=&\frac{d}{dr}\varphi_n(\mu)(\kappa(\alpha,\beta)\imath(e^r))\mid_{r=0}
=\frac{d}{dr}\left(\frac{e^{r(M+2)}\mu\left(\left(\alpha y-e^r\beta x\right)^{n+\lambda}\left(-e^r\bar\beta y-\bar\alpha x\right)^{n-\lambda}\right)}{\left(|-e^r\bar \beta|^2+|\bar \alpha|^2\right)^{M+n+2}}\right)_{r=0}\\
&=&\left((M+2)-2|\beta|^2(M+n+2)\right)\mu(P_{n}(-\bar\beta,\bar\alpha))+(n+\lambda)\mu\left(\beta x\left(\bar\beta y+\bar\alpha x\right)P_{n-1}(-\bar\beta,\bar\alpha)\right)-\\
&&-(n-\lambda)\mu\left(\bar\beta y\left(\alpha y-\beta x\right)P_{n-1}(-\bar\beta,\bar\alpha)\right)\\
&=&(M+n+2)\left(1-2|\beta|^2\right)\mu(P_{n}(-\bar\beta,\bar\alpha))+(n+\lambda)\mu\left(xyP_{n-1}(-\bar\beta,\bar\alpha)\right)-\\
&&-\lambda\mu\left((\alpha y-\beta x)(\bar\alpha x-\bar\beta y))P_{n-1}(-\bar\beta,\bar\alpha)\right),
\end{eqnarray*}
where in the last equality we have used the identity $(\bar\beta y+\bar \alpha x)\beta=y-(\alpha y-\beta x)\bar\alpha$. If we use $(\bar\beta y+\bar \alpha x)\alpha=x+(\alpha y-\beta x)\bar\beta$ as well, we obtain
\[
\frac{\partial P_n}{\partial x}(-\bar\beta,\bar\alpha)=\left((n+\lambda)(\bar\beta y+\bar \alpha x)\beta-(n-\lambda)(\alpha y-\beta x)\bar\alpha\right)P_{n-1}(-\bar\beta,\bar\alpha)=\left((n+\lambda)y-2n(\alpha y-\beta x)\bar\alpha\right)P_{n-1}(-\bar\beta,\bar\alpha),\]\[
\frac{\partial P_n}{\partial y}(-\bar\beta,\bar\alpha)=\left((\lambda-n)(\alpha y-\beta x)\bar\beta-(n+\lambda)(\bar\beta y+\bar \alpha x)\alpha\right)P_{n-1}(-\bar\beta,\bar\alpha)=\left(2n(\beta x-\alpha y)\bar\beta-(n+\lambda)x\right)P_{n-1}(-\bar\beta,\bar\alpha).
\]
Thus,
\begin{eqnarray*}
\frac{\partial^2 P_{n+1}\mbox{\tiny$(-\bar\beta,\bar\alpha)$}}{\partial y\partial x}
&=&\left(-(n+1+\lambda)+2(n+1)|\beta|^2\right)P_n(-\bar\beta,\bar\alpha)+\left(-(n+1+\lambda)x-2(n+1)(\alpha y-\beta x)\bar\beta\right)\frac{\partial P_n}{\partial x}(-\bar\beta,\bar\alpha)\\
&=&(n+1)(2n+1)(2|\beta|^2-1)P_n(-\bar\beta,\bar\alpha)+\lambda(2n+1)(\bar\alpha x-\bar\beta y)(\alpha y-\beta x)P_{n-1}(-\bar\beta,\bar\alpha)-\\
&&-(n+\lambda)(n+\lambda+1)xyP_{n-1}(-\bar\beta,\bar\alpha).
\end{eqnarray*}
We deduce 
\begin{eqnarray*}
\delta \varphi_n(\mu)(\alpha,\beta)&=&
\frac{\lambda(M+1)}{2n(n+1)}\mu\left(\left(y\frac{\partial P_n}{\partial y}\mbox{\tiny$(-\bar\beta,\bar\alpha)$}-x\frac{\partial P_n}{\partial x}\mbox{\tiny$(-\bar\beta,\bar\alpha)$}\right)\right)-\frac{(M+n+2)}{(n+1)(2n+1)}\mu\left(\frac{\partial P_{n+1}}{\partial x\partial y}\mbox{\tiny$(-\bar\beta,\bar\alpha)$}\right)+\\
&&+\frac{(n+\lambda)(n-\lambda)}{n}\left(\frac{n-M-1}{2n+1}\right)\mu\left(xyP_{n-1}(-\bar\beta,\bar\alpha)\right).
\end{eqnarray*}
Therefore,
$\delta \varphi_n(\mu)=\varphi_n(\mu_0)+\varphi_{n+1}(\mu_1)+\varphi_{n-1}(\mu_{-1})$,
for $\mu_0\in V(2n)$, $\mu_1\in V(2n+2)$ and $\mu_{-1}\in V(2n-2)$, where 
\begin{eqnarray*}
\mu_0(P)&:=&\frac{\lambda(M+1)}{2n(n+1)}\mu\left(
y\frac{\partial P}{\partial y}-x\frac{\partial P}{\partial x}\right),\qquad\mu_1(P):=-\frac{(M+n+2)}{(n+1)(2n+1)}\mu\left(\frac{\partial P}{\partial x\partial y}\right),\\
\mu_{-1}(P)&=&\frac{(n+\lambda)(n-\lambda)}{n}\left(\frac{n-M-1}{2n+1}\right)\mu\left(xyP\right).
\end{eqnarray*}

Recall that $\rho(\delta s(\underline{\mu}))=0$, hence $(1-\rho)\delta s(\underline{\mu})=\delta s(\underline{\mu})$. But $(1-\rho)(V(2n))=0$ for all $|\lambda|\leq n\leq M$. We conclude
\begin{eqnarray*}
\delta s(\underline{\mu})&=&(1-\rho)\sum_{|\lambda|\leq n\leq M}(2n+1)\binom{2n}{n+\lambda}\left(\varphi_n(t_n(\underline{\mu})_0)+\varphi_{n+1}(t_n(\underline{\mu})_1)+\varphi_{n-1}(t_n(\underline{\mu})_{-1})-\varphi_n(t_n(\delta\underline{\mu}))\right)\\
&=&(2M+1)\binom{2M}{k_{\rm id}-2}\varphi_{M+1}(t_M(\underline{\mu})_1),
\end{eqnarray*}
and the result follows.
\end{proof}



\subsubsection{Inner products and pairings
}\label{calcI(chi)2}

Let us consider the $\PGL_2(\C)$-invariant pairing
\[
\langle:\;,\;\rangle_\cB:\cB(\chi_{\underline{k}})\times \cB(\chi_{2-{\underline{k}}})\longrightarrow\C;\qquad \langle f,h\rangle_\cB:=\int_{S^3}f(\kappa(\alpha,\beta))\cdot h(\kappa(\alpha,\beta))d(\alpha,\beta).
\]
It satisfies the property
\[
\langle \imath f,\rho^\vee\mu\rangle_\cB=0, \qquad\mbox{for all }\quad f\in D(\underline{k}),\;\mu\in V(\underline{k}-2).
\]
Again it provides (the unique up to constant) $\PGL_2(\C)$-invariant bilinear inner products on $D(\underline{k})$ and $V(\underline{k}-2)$: 
\begin{eqnarray*}
\langle\;,\;\rangle:D(\underline{k})\times D(\underline{k})\longrightarrow\C;&\qquad& \langle f,h\rangle:=\langle \imath f,s'h\rangle_\cB,\\
\langle\;,\;\rangle_V:V(\underline{k}-2)\times V(\underline{k}-2)\longrightarrow\C;&\qquad& \langle \mu_1,\mu_2\rangle_V:=\langle s(\mu_1),\rho^\vee(\mu_2)\rangle_\cB.
\end{eqnarray*}


As in previous sections, we aim to compute as well
\[
\cI(\chi,\underline{m}):=\int_{T(\C)}\chi(t)\langle \imath(t)\delta s(\mu_{\underline{m}}),\delta s(\mu_{-\underline{m}})\rangle t^{\underline{m}}d^\times t, 
\]
for a locally constant character $\chi$ and the Haar measure $d^\times t$ of \eqref{defdtx}.
Since $\iota(t)\mu_{\underline{m}}=t^{-\underline{m}}\mu_{\underline{m}}$, we have
\[
\delta\mu_{\underline{m}}=\frac{d}{dr}(\imath(e^r)\mu_{\underline{m}})\mid_{r=0}=\frac{d}{dr}(e^{-mr}\mu_{\underline{m}})\mid_{x=0}=-m\cdot\mu_{\underline{m}},
\] 
where $m:=m_{\rm id}+m_c$.
Hence, let us consider again the function 
\[
F(t)=t^{\underline{m}}\langle \imath(t)s(\mu_{\underline{m}}),s'{\delta s(\mu_{-\underline{m}})}\rangle_\cB.
\]

By remark \ref{normalizationD2}, $s'(\varphi_{M+1}(\mu))=\varphi_{M+1}(\mu)$ for all $\mu\in V(2M+2)$. Thus, by proposition \ref{keylemma2}, we have
\begin{equation}\label{ssdelta}
s'\delta s(\mu_{-\underline{m}})=-2\binom{2M}{k_{\rm id}-2}\varphi_{M+1}(t_M(\mu_{-\underline{m}})^*).
\end{equation}
Hence,
\[
s'\delta s(\mu_{-\underline{m}})\left(\begin{array}{cc}\ast&\ast\\R&S\end{array}\right)=-2\binom{2M}{k_{\rm id}-2}\frac{t_M(\mu_{-\underline{m}})^\ast\left(\left|\begin{array}{cc}\bar S&-\bar R\\x&y\end{array}\right|^{k_c-1}\left|\begin{array}{cc} R&S\\x&y\end{array}\right|^{k_{\rm id}-1}\right)}{\Delta^{\frac{\underline{k}}{2}-1}\cdot \left(\left|R\right|^2+\left|S\right|^2\right)},
\]
where $\Delta$ is the deteminant of $\left(\begin{array}{cc}\ast&\ast\\R&S\end{array}\right)$. 
\begin{remark}\label{pairP}
The isomorphism $\iota:V(2n)\stackrel{\simeq}{\rightarrow}\cP(2n)$ of \eqref{dualVP} provides a $\GL_2(\C)$-equivariant pairing
\[
\langle\;,\;\rangle_{\cP}:\cP(2n)\times\cP(2n)\longrightarrow\C,
\]
where $\langle P_1,P_2\rangle_{\cP}=\iota^{-1}(P_1)(P_2)$. 
In the usual basis
\[
\langle x^{i}y^{2n-i},x^{j}y^{2n-j}\rangle_{\cP}=\left\{\begin{array}{ll}(-1)^{i}\binom{2n}{i}^{-1},&i=2n-j\\0,&i\neq 2n-j\end{array}\right.
\]
\end{remark}

\begin{remark}\label{pairV}
Similarly as in previous remark \ref{pairP}, the isomorphism $\iota:V(k_{\sigma}-2)\stackrel{\simeq}{\rightarrow}\cP(k_{\sigma}-2)$ provides a $\GL_2(\C)$-equivariant pairing
\[
\langle\;,\;\rangle':V(\underline{k}-2)\times V(\underline{k}-2)\longrightarrow\C.
\]
such that, for $\mu_{\underline{m}}$, $\mu_{\underline{m}'}$ as in Definition \ref{defmum},
\begin{equation}\label{pairmum}
\langle \mu_{\underline{m}},\mu_{\underline{m}'}\rangle'=\left\{\begin{array}{ll}(-1)^{M+m}\binom{k_{\rm id}-2}{\frac{k_{\rm id}-2}{2}-m_{\rm id}}^{-1}\binom{k_{c}-2}{\frac{k_{c}-2}{2}-m_{c}}^{-1},&\underline{m}=-\underline{m}'\\0,&\underline{m}\neq -\underline{m}'\end{array}\right.
\end{equation}
On the other side, we have defined another pairing $\langle\;,\;\rangle_V$ on $V(\underline{k}-2)$. By Schur's lemma both pairings must be the same up to constant. To compute such a constant we can consider the elements $\mu_1:=\mu_{\underline{m}_0}$, where $\underline{m}_0:=(\frac{2-k_{\rm id}}{2},\frac{k_c-2}{2})$, and $\mu_2\in V(\underline{k}-2)$ such that $t_n(\mu_2)=0$ for $n\neq M$ and $t_M(\mu_2)=(x^{2M})^\vee$. In this situation 
\begin{eqnarray*}
\langle \mu_2,\mu_1\rangle'&=&\mu_2\mu_1\left(\left|\begin{array}{cc}X'&Y'\\X&Y\end{array}\right|^{\underline{k}-2}\right)=\rho\circ s(\mu_2)(X_{\rm id}^{k_{\rm id}-2}Y_c^{k_c-2})\\
&=&\rho\left((2M+1)\binom{2M}{M+\lambda}\varphi_M((x^{2M})^\vee)\right)(X_{\rm id}^{k_{\rm id}-2}Y_c^{k_c-2})=1,
\end{eqnarray*}
by \eqref{norm2}. On the other side, we compute 
\begin{eqnarray*}
\langle \mu_2,\mu_1\rangle_V&=&\int_{S^3}s( \mu_2)(\kappa(\alpha,\beta)) \cdot\mu_1\left(\left|\begin{array}{cc}X&Y\\-\bar\beta&\bar\alpha\end{array}\right|^{\underline{k}-2}\right)d(\alpha,\beta)\\
&=&(2M+1)\binom{2M}{M+\lambda}^{-1}\int_{S^3}\varphi_M((x^{2M})^\vee)\left(\begin{array}{cc}\alpha&\beta\\-\bar\beta&\bar\alpha\end{array}\right)(-\bar\beta)^{k_{\rm id}-2}\alpha^{k_c-2}d(\alpha,\beta)\\
&=&\frac{(2M+1)}{\binom{2M}{M+\lambda}^{-1}}\int_{S^3}(x^{2M})^\vee\left(\left|\begin{array}{cc}\alpha&\beta\\x&y\end{array}\right|^{M+\lambda}\left|\begin{array}{cc}-\bar\beta&\bar\alpha\\x&y\end{array}\right|^{M-\lambda}\right)(-\bar\beta)^{k_{\rm id}-2}\alpha^{k_c-2}d(\alpha,\beta)\\
&=&(2M+1)\binom{2M}{M+\lambda}^{-1}\int_{S^3}(-\beta)^{M+\lambda}(-\bar\alpha)^{M-\lambda}(-\bar\beta)^{k_{\rm id}-2}\alpha^{k_c-2}d(\alpha,\beta)=1,
\end{eqnarray*}
by lemma \ref{auxlem}. Thus, both pairings coincide and \eqref{pairmum} also holds for $\langle\;,\;\rangle_V$.
\end{remark}

\begin{proposition}\label{calcinfty02}
We have that 
\[
F(\infty)-F(0)=(-1)^{m+M}2(2M+1)(2M+2)\binom{2M}{k_{\rm id}-2}\binom{k_{\rm id}-2}{\frac{k_{\rm id}-2}{2}-m_{\rm id}}^{-1}\binom{k_{c}-2}{\frac{k_{c}-2}{2}-m_{c}}^{-1}.
\]
where again $F(\ast)=\lim_{t\to\ast} F(t)$.
\end{proposition}
\begin{proof}
By remark \ref{pairP} we have that
\[
t_M(\mu_{-\underline{m}})(P)=\langle\mu_{-\underline{m}}(\Delta_M),P \rangle_\cP=(-1)^{\frac{k_c-2}{2}-m_c}\langle y^{M-m_{\rm id}+m_c}x^{M+m_{\rm id}-m_c},P \rangle_\cP,
\]
for all $P\in\cP(2M)$.
Hence, again by remark \ref{pairP}, if we write $\bar m:=m_{\rm id}-m_c$ then
\begin{eqnarray*}
&&t_M(\mu_{-\underline{m}})^\ast\left(\mbox{\tiny$\left|\begin{array}{cc}A&-B\\x&y\end{array}\right|^{k_c-1}\left|\begin{array}{cc} C&-D\\x&y\end{array}\right|^{k_{\rm id}-1}$}\right)=t_M(\mu_{-\underline{m}})\left(\mbox{\tiny$2AB\binom{k_c-1}{2}\left|\begin{array}{cc}A&-B\\x&y\end{array}\right|^{k_c-3}\left|\begin{array}{cc} C&-D\\x&y\end{array}\right|^{k_{\rm id}-1}$}+\right.\\
&+&\left.\mbox{\tiny$(AD+BC)(k_{\rm id}-1)(k_c-1)\left|\begin{array}{cc}A&-B\\x&y\end{array}\right|^{k_c-2}\left|\begin{array}{cc} C&-D\\x&y\end{array}\right|^{k_{\rm id}-2}$}+\mbox{\tiny$2CD\binom{k_{\rm id}-1}{2}\left|\begin{array}{cc}A&-B\\x&y\end{array}\right|^{k_c-1}\left|\begin{array}{cc} C&-D\\x&y\end{array}\right|^{k_{\rm id}-3}$}\right)\\
&=&\frac{(-1)^{\frac{k_{\rm id}-2}{2}-m_{\rm id}}}{\binom{2M}{M+\bar m}}\left(2\binom{k_c-1}{2}\sum_i\binom{k_c-3}{i}\binom{k_{\rm id}-1}{M+\bar m-i}A^{i+1}B^{k_c-2-i}C^{M+\bar m-i}D^{i+\lambda-\bar m+1}\right.+\\
&+&\sum_i(AD+BC)(k_{\rm id}-1)(k_c-1)\binom{k_c-2}{i}\binom{k_{\rm id}-2}{M+\bar m-i}A^{i}B^{k_c-2-i}C^{M+\bar m-i}D^{\lambda-\bar m+i}+\\
&+&\left.2\binom{k_{\rm id}-1}{2}\sum_i\binom{k_c-1}{i}\binom{k_{\rm id}-3}{M+\bar m-i}A^{i}B^{k_c-1-i}C^{M+\bar m-i+1}D^{i+\lambda-\bar m}\right)\\
&=&\frac{(-1)^{\frac{k_{\rm id}-2}{2}-m_{\rm id}}(2M+1)(2M+2)}{\binom{2M+2}{M+\bar m+1}}\left(\sum_i\binom{k_c-1}{i}\binom{k_{\rm id}-1}{\lambda-\bar m+i}A^{i}B^{k_c-1-i}C^{M+\bar m-i+1}D^{i+\lambda-\bar m}\right).
\end{eqnarray*}

Thus, by equation \eqref{ssdelta}, the function $F(t)=t^{\underline{m}}\langle s(\mu_{\underline{m}}),\imath(t^{-1})s'{\delta s(\mu_{-\underline{m}})}\rangle_\cB$ is given by
\begin{eqnarray*}
F(t)&=&-2\binom{2M}{k_{\rm id}-2}t^{\underline{m}+\frac{\underline{k}-2}{2}}\int_{S^3}s(\mu_{\underline{m}})(\kappa(\alpha,\beta))\frac{t_M(\mu_{-\underline{m}})^\ast\left(\mbox{\tiny$\left|\begin{array}{cc}\alpha&\bar  t^{-1}\beta\\x&y\end{array}\right|^{k_c-1}\left|\begin{array}{cc} -t^{-1}\bar\beta&\bar\alpha\\x&y\end{array}\right|^{k_{\rm id}-1}$}\right)}{ \left(\left|\alpha\right|^2+r^{-2}\left|\beta\right|^2\right)}d(\alpha,\beta)\\
&=&k_M\sum_i\binom{k_c-1}{i}\binom{k_{\rm id}-1}{\lambda-\bar m+i}r^{2i+2+2m_c-k_c}\int_{S^3}s(\mu_{\underline{m}})(\kappa(\alpha,\beta))\frac{\alpha^{i}(-\beta)^{k_c-1-i}(-\bar\beta)^{M+\bar m-i+1}(-\bar\alpha)^{i+\lambda-\bar m}}{\left(r^2\left|\alpha\right|^2+\left|\beta\right|^2\right)}d(\alpha,\beta)
\end{eqnarray*}
where $k_M:=\frac{2(-1)^{\frac{k_{\rm id}}{2}-m_{\rm id}}(2M+1)(2M+2)\binom{2M}{k_{\rm id}-2}}{\binom{2M+2}{M+\bar m+1}}$.

If $r<<0$, we use $\left(r^2\left|\alpha\right|^2+\left|\beta\right|^2\right)^{-1}=|\beta|^{-2}\sum_{j\geq 0}\left(-r^2|\alpha|^2|\beta|^{-2}\right)^j$ to obtain
\begin{eqnarray*}
F(t)&=&k_M\sum_i\binom{k_c-1}{i}\binom{k_{\rm id}-1}{\lambda-\bar m+i}\sum_{j\geq 0}r^{2m_c-k_c+2i+2+2j}p(i+j)\\
&=&k_M\sum_{N\geq 0}p(N)\cdot r^{2m_c-k_c+2N+2}\sum_{i\leq N}\binom{k_c-1}{i}\binom{k_{\rm id}-1}{\lambda-\bar m+i}
\end{eqnarray*}
where 
\[
p(N):=\int_{S^3}s(\mu_{\underline{m}})(\kappa(\alpha,\beta))\cdot\alpha^{N}(-\beta)^{k_c-2-N}(-\bar\alpha)^{N+\lambda-\bar m}(-\bar\beta)^{M+\bar m-N}d(\alpha,\beta)
\]
By definition, for any $\underline{m}'=(m_{\rm id}',m_c')$,
\[
\rho^\vee(\mu_{\underline{m}'})(\kappa(\alpha,\beta))=\mu_{\underline{m}'}\left(\left|\begin{array}{cc}X&Y\\-\bar\beta&\bar\alpha\end{array}\right|^{\underline{k}-2}\right)=(-\bar\beta)^{\frac{k_{\rm id}-2}{2}-m_{\rm id}'}\bar\alpha^{\frac{k_{\rm id}-2}{2}+m_{\rm id}'}(-\beta)^{\frac{k_{c}-2}{2}-m_{c}'}\alpha^{\frac{k_c-2}{2}+m_{c}'}.
\]
Hence, if we assume $N\leq \frac{k_c-2}{2}-m_c$ then we have that  
\[
\frac{2-\underline k}{2}\leq\underline{m}_N:=\left(N-\frac{k_c-2}{2}-\bar m,N-\frac{k_c-2}{2}\right)\leq \frac{\underline k-2}{2}. 
\]
Moreover, by remark \ref{pairV},
\[
p(N):=(-1)^{N+\lambda+\bar m}\langle s(\mu_{\underline{m}}),\rho^\vee(\mu_{\underline{m}_N})\rangle_\cB=(-1)^{N+M+m}\langle \mu_{\underline{m}},\mu_{\underline{m}_N}\rangle_V.
\]
Hence, 
we conclude that 
\[
F(0)=k_M(-1)^{\frac{k_{\rm id}-2}{2}-m_{\rm id}}\langle \mu_{\underline{m}},\mu_{-\underline{m}}\rangle_V\sum_{i\leq \frac{k_c-2}{2}-m_c}\binom{k_c-1}{i}\binom{k_{\rm id}-1}{\lambda-\bar m+i}.
\]

If $r>>0$, we use $\left(r^2\left|\alpha\right|^2+\left|\beta\right|^2\right)^{-1}=r^{-2}|\alpha|^{-2}\sum_{j\geq 0}\left(-r^{-2}|\beta|^2|\alpha|^{-2}\right)^j$ to obtain
\begin{eqnarray*}
F(t)&=&-k_M\sum_i\binom{k_c-1}{i}\binom{k_{\rm id}-1}{\lambda-\bar m+i}\sum_{j\geq 0}r^{2m_c-k_c+2i-2j}p(i-j-1)\\
&=&-k_M\sum_{N\leq k_c-2}p(N)\cdot r^{2m_c-k_c+2+2N}\sum_{i\geq N+1}\binom{k_c-1}{i}\binom{k_{\rm id}-1}{\lambda-\bar m+i}
\end{eqnarray*}
Again for $N\geq \frac{k_c-2}{2}-m_c$, we check that $\frac{2-\underline k}{2}\leq\underline{m}_N\leq \frac{\underline k-2}{2}$. 
Similarly as above we obtain
\[
F(\infty)=-k_M(-1)^{\frac{k_{\rm id}-2}{2}-m_{\rm id}}\langle \mu_{\underline{m}},\mu_{-\underline{m}}\rangle_V\sum_{i\geq \frac{k_c}{2}-m_c}\binom{k_c-1}{i}\binom{k_{\rm id}-1}{\lambda-\bar m+i}.
\]

We conclude again by remark \ref{pairV} that
\begin{eqnarray*}
F(\infty)-F(0)&=&k_M(-1)^{m_{\rm id}+\frac{k_{\rm id}}{2}}\langle \mu_{\underline{m}},\mu_{-\underline{m}}\rangle_V\sum_{i}\binom{k_c-1}{i}\binom{k_{\rm id}-1}{\lambda-\bar m+i}\\
&=&k_M(-1)^{m_{\rm id}+\frac{k_{\rm id}}{2}}\langle \mu_{\underline{m}},\mu_{-\underline{m}}\rangle_V\binom{2M+2}{M-\bar m+1}=(-1)^{m+M}\frac{2(2M+1)(2M+2)\binom{2M}{k_{\rm id}-2}}{\binom{k_{\rm id}-2}{\frac{k_{\rm id}-2}{2}-m_{\rm id}}\binom{k_{c}-2}{\frac{k_{c}-2}{2}-m_{c}}}.
\end{eqnarray*}
 and the result follows.
\end{proof}

\begin{theorem}\label{calcLI2}
Given a locally constant character $\chi:T(\C)\rightarrow\C^\times$, we denote 
\[
\cI(\chi,{\underline m}):=\int_{T(\C)}\chi(t)\langle\imath(t)\delta s(\mu_{\underline m}),\delta s(\mu_{-\underline m})\rangle t^{\underline m}d^\times t,
\]
where $d^\times t$ is the Haar measure \eqref{defdtx}. Then 
\[
\cI(\chi,{\underline m})=\left\{\begin{array}{ll}(-1)^{M+m}8(2M+1)(2M+2)\binom{2M}{k_{\rm id}-2}\binom{k_{\rm id}-2}{\frac{k_{\rm id}-2}{2}-m_{\rm id}}^{-1}\binom{k_{c}-2}{\frac{k_{c}-2}{2}-m_{c}}^{-1},&\chi=1,\\0,&\chi\neq 1.\end{array}\right.
\]
\end{theorem}
\begin{proof}
Using polar coordinates $t=re^{i\theta}$, the function $F(t)$
satisfies
\begin{eqnarray*}
r\frac{\partial}{\partial r}F(t)&=&mt^{\underline{m}}\langle \imath(t) s(\mu_{\underline{m}}),s'{\delta s(\mu_{-\underline{m}})}\rangle_\cB+t^{\underline{m}}\frac{d}{d x}\langle \imath(te^x) s(\mu_{\underline{m}}),s'{\delta s(\mu_{-\underline{m}})}\rangle_\cB\mid_{x=0}\\
&=&t^{\underline{m}}\left(-\langle \imath(t) s(\delta\mu_{\underline{m}}),s'{\delta s(\mu_{-\underline{m}})}\rangle_\cB+\langle \imath(t) \delta(s(\mu_{\underline{m}})),s'{\delta s(\mu_{-\underline{m}})}\rangle_\cB\right)\\
&=&t^{\underline{m}}\langle \imath(t)\delta s(\mu_{\underline{m}}),\delta s(\mu_{-\underline{m}})\rangle.
\end{eqnarray*}
Hence, 
\begin{eqnarray*}
I(\chi,\underline{m})&=&\frac{2}{\pi}\int_{0}^{2\pi}\chi(\theta)\int_0^\infty \langle \imath(t)\delta s(\mu_{\underline{m}}),\delta s(\mu_{-\underline{m}})\rangle t^{\underline{m}}\frac{dr}{r} d\theta\\
&=&\frac{2}{\pi}\int_{0}^{2\pi}\chi(\theta) \int_0^\infty r\frac{\partial}{\partial r}F(t)\frac{dr}{r}d\theta=\left\{\begin{array}{ll}4(F(\infty)-F(0)),&\chi=1,\\0,&\chi\neq 1,\end{array}\right.
\end{eqnarray*}
and the result follows from proposition \ref{calcinfty02}.
\end{proof}

\subsubsection{Pairings of canonical sections}

Let us consider the canonical $\GL_2(\C)$-invariant element
\[
\Upsilon=\iota^{-1}\left|\begin{array}{cc}
    \underline x_1 & \underline y_1 \\
    \underline x_2 & \underline y_2
\end{array}\right|^{\underline k-2}\in V(\underline{k}-2)\otimes V(\underline{k}-2),
\]
where $\iota:V(\underline{k}-2)\otimes V(\underline{k}-2)\stackrel{\simeq}{\rightarrow}\cP(\underline{k}-2)\otimes \cP(\underline{k}-2)$ is the isomorphism of \eqref{dualVP}
and $\underline x_i=(x_{i,{\rm id}},x_{i,c})$.
Since we have the morphism $\delta s:V(\underline{k}-2)\longrightarrow D(\underline{k})$, we can consider $\delta s(\Upsilon)\in D(\underline{k})\otimes D(\underline{k})$ and apply the inner product $\langle\;;\;\rangle$ defined in the beginning of \S \ref{calcI(chi)2}. 
\begin{lemma}\label{lemoonpairupsilon}
    We have that 
    \[
    \langle\delta s(\Upsilon)\rangle=\frac{2}{3}\binom{2M}{k_{\rm id}-2}^2\binom{2M+2}{k_{\rm id}-1}^{-1}(2M+2)^{2}(2M+1)^{2}.
    \]
\end{lemma}
\begin{proof}
By proposition \ref{keylemma2} and equation \eqref{ssdelta}, 
\[
\langle\delta s(\underline{\Upsilon})\rangle=4\binom{2M}{k_{\rm id}-2}^2\langle\varphi_{M+1}(t_M(\underline{\Upsilon})^*)\rangle_\cB,
\]
where $\varphi_{M+1}:V(2M+2)\otimes V(2M+2)\rightarrow \cB(\chi_{\underline{k}})\times \cB(\chi_{2-{\underline{k}}})$.
In order to control $t_M(\Upsilon)^\ast$ we compute its image through the isomorphism $\iota$ of \eqref{dualVP}:
\begin{eqnarray*}
    \iota t_M(\Upsilon)^\ast(x_1,y_1,x_2,y_2)&:=&t_M(\Upsilon)^\ast\left(\left|\begin{array}{cc}
    X_1 & Y_1 \\
    x_1 & y_1
\end{array}\right|^{2M+2}\left|\begin{array}{cc}
    X_2 & Y_2 \\
    x_2 & y_2
\end{array}\right|^{2M+2}\right)\\
&=&(2M+2)^{2}(2M+1)^{2}x_1y_1x_2y_2\cdot t_M(\Upsilon)\left(\left|\begin{array}{cc}
    X_1 & Y_1 \\
    x_1 & y_1
\end{array}\right|^{2M}\left|\begin{array}{cc}
    X_2 & Y_2 \\
    x_2 & y_2
\end{array}\right|^{2M}\right)\\
&=&(2M+2)^{2}(2M+1)^{2}x_1y_1x_2y_2\cdot \Upsilon\left(\left|\begin{array}{cc}
    \underline X_1 & \underline Y_1 \\
    \underline z_1 & \underline w_1
\end{array}\right|^{\underline k-2}\left|\begin{array}{cc}
    \underline X_2 & \underline Y_2 \\
    \underline z_2 & \underline w_2
\end{array}\right|^{\underline k-2}\right)\\
&=&(2M+2)^{2}(2M+1)^{2}x_1y_1x_2y_2\left|\begin{array}{cc}
    x_1 & y_1 \\
     x_2 & y_2
\end{array}\right|^{2M},
\end{eqnarray*}
where $\underline z_1=(x_1,y_1)$, $\underline w_1=(y_1,-x_1)$, $\underline z_2=(x_2,y_2)$ and  $\underline w_2=(y_2,-x_2)$. 
To compute $\langle\varphi_{M+1}(t_M(\underline{\Upsilon})^*)\rangle_\cB$, notice that
\[
\langle \varphi_{M+1}(\cdot),\varphi_{M+1}(\cdot)\rangle_\cB\in \Hom_{\SU(2)}(V(2M+2)\otimes V(2M+2),\C)=\Hom_{\SU(2)}(V(2M+2), V(2M+2))=\C,
\]
by Schur lemma, but remark \ref{pairP} provides a rather simpler $\SU(2)$-equivariant pairing $\langle\cdot,\cdot\rangle_\cP$. Thus, both pairings must coincide up to constant. To find such a constant we use $\mu_1=(y^{M+1})^\vee,\mu_2=(x^{M+1})^\vee\in V(2M+2)$ corresponding to $x^{2M+2},y^{2M+2}\in\cP(2M+2)$, respectively. We compute
\begin{eqnarray*}
    \langle x^{2M+2},y^{2M+2}\rangle_\cP&=&1\\
    \langle \varphi_{M+1}(\mu_1),\varphi_{M+1}(\mu_2)\rangle&=&\int_{s^3}\mu_1\left(\left|\begin{array}{cc}\alpha&\beta\\x&y\end{array}\right|^{k_{\rm id}-1}\left|\begin{array}{cc}-\bar\beta&\bar\alpha\\x&y\end{array}\right|^{k_c-1}\right)\mu_2\left(\left|\begin{array}{cc}\alpha&\beta\\x&y\end{array}\right|^{k_c-1}\left|\begin{array}{cc}-\bar\beta&\bar\alpha\\x&y\end{array}\right|^{k_{\rm id}-1}\right)d(\alpha,\beta)\\
    &=&-\int_{S^3}|\alpha|^{2k_{\rm id}-2}|\beta|^{2k_c-2}d(\alpha,\beta)=\frac{-1}{2M+3}\binom{2M+2}{k_{\rm id}-1}^{-1}.
\end{eqnarray*}
Hence, we deduce
\begin{equation}\label{pairPpairB}
    \langle \varphi_{M+1}(\cdot),\varphi_{M+1}(\cdot)\rangle_\cB=\frac{-1}{2M+3}\binom{2M+2}{k_{\rm id}-1}^{-1}\langle\cdot,\cdot\rangle_\cP.
\end{equation}
This implies that
\begin{eqnarray*}
\frac{\langle \varphi_{M+1}(t_M(\Upsilon)^\ast)\rangle}{(2M+2)^{2}(2M+1)^{2}}&=&\frac{-1}{2M+3}\binom{2M+2}{k_{\rm id}-1}^{-1}\langle \iota t_M(\Upsilon)^\ast\rangle_\cP\\
&=&\frac{1}{2M+3}\binom{2M+2}{k_{\rm id}-1}^{-1}\sum_{i=0}^{2M}\binom{2M}{i}\langle x_1^{i+1}y_1^{2M-i+1},(-x_2)^{2M-i+1}y_2^{i+1} \rangle_\cP
\\
&=&\frac{1}{2M+3}\binom{2M+2}{k_{\rm id}-1}^{-1}\sum_{i=0}^{2M}\binom{2M}{i}\binom{2M+2}{i+1}^{-1}=\frac{1}{6}\binom{2M+2}{k_{\rm id}-1}^{-1}.
\end{eqnarray*}
Hence, the result follows.  
\end{proof}

Another interesting element to consider is $\delta s(\mu_{\underline 0})$, where $\underline 0=(0,0)$. Since $s$ is $\SU(2)$-equivariant, $\delta$ is in the Lie algebra of the diagonal torus $T(\C)$ and $\mu_{\underline 0}$ is invariant under the action of $T(\C)$, we deduce that $\delta s(\mu_{\underline 0})$ generates the subspace of $D(\underline k)$ invariant under $\kappa(e^{i\theta},0)\in\SU(2)$ with minimal weight. 
\begin{lemma}\label{pairv02}
We have that
\[
  \langle\delta s(\mu_{\underline 0}),\delta s(\mu_{\underline 0})\rangle=4(-1)^{M}(2M+2)^2(2M+1)^2(2M+3)^{-1}\binom{2M+2}{k_{\rm id}-1}^{-1}\binom{2M}{k_{\rm id}-2}^2\binom{2M+2}{M+1}^{-1}
\]
\end{lemma}
\begin{proof}
    By proposition \ref{keylemma2} and equation \eqref{pairPpairB},
    \[
     \langle\delta s(\mu_{\underline 0}),\delta s(\mu_{\underline 0})\rangle=4\binom{2M}{k_{\rm id}-2}^2\langle\varphi_{M+1}(t_M(\mu_{\underline 0})^*),\varphi_{M+1}(t_M(\mu_{\underline 0})^*)\rangle_\cB=\frac{-4\binom{2M+2}{k_{\rm id}-1}^{-1}\binom{2M}{k_{\rm id}-2}^2}{2M+3}\langle \iota t_M(\mu_{\underline 0})^*,\iota t_M(\mu_{\underline 0})^*\rangle_\cP.
    \]
    Moreover,
    \begin{eqnarray*}
        \iota t_M(\mu_{\underline 0})^*&=&t_M(\mu_{\underline 0})^*\left((Xy-Yx)^{2M+2}\right)=(2M+2)(2M+1)t_M(\mu_{\underline 0})\left(-xy(Xy-Yx)^{2M}\right)\\
        &=&-xy(2M+2)(2M+1)t_M(\mu_{\underline 0})\left((Xy-Yx)^{2M}\right)\\
        &=&-xy(2M+2)(2M+1)\mu_{\underline 0}\left(\left|\begin{array}{cc}X_{\rm id}&Y_{\rm id}\\x&y\end{array}\right|^{k_{\rm id}-2}\left|\begin{array}{cc}X_{c}&Y_{c}\\y&-x\end{array}\right|^{k_c-2}\right)\\
        &=&(-1)^{\frac{k_{c}}{2}}(2M+2)(2M+1)x^{M+1}y^{M+1}.
    \end{eqnarray*}
    Thus, 
    \[
    \langle \iota t_M(\mu_{\underline 0})^*,\iota t_M(\mu_{\underline 0})^*\rangle_\cP=(-1)^{M+1}(2M+2)^2(2M+1)^2\binom{2M+2}{M+1}^{-1},
    \]
    and the result follows.
\end{proof}

\Addresses

\printbibliography

\end{document}